\documentclass[10pt,leqno]{amsart}
\usepackage{graphicx}  
\usepackage{enumerate}
\usepackage[colorlinks=true]{hyperref}
\hypersetup{urlcolor=blue, citecolor=red}
\usepackage{mathabx} 

\newtheorem{theorem}{Theorem}[section]
\newtheorem{corollary}[theorem]{Corollary}
\newtheorem{lemma}[theorem]{Lemma}
\newtheorem{proposition}[theorem]{Proposition}
\theoremstyle{definition}
\newtheorem{remark}[theorem]{Remark}
\newtheorem{ex}[theorem]{Example}

\usepackage[indentafter]{titlesec}
\titleformat{name=\section}[runin]{}{\thetitle.}{0.5em}{\bfseries}[.]
\titleformat{name=\subsection}[runin]{}{\thetitle.}{0.5em}{\bfseries}[.]

\numberwithin{equation}{section}

\newcommand{\p}{\begin{pmatrix}}
\newcommand{\pp}{\end{pmatrix}}
\newcommand{\din}{\mathrm{in}}
\newcommand{\dout}{\mathrm{out}}

\title[Number and Stability of Relaxation Oscillations]%
{Number and Stability of Relaxation Oscillations
\\
for Predator-Prey Systems
\\
with Small Death Rates
}

\author[Ting-Hao Hsu]{}

\thanks{%
$^\dag$Research supported by
Natural Sciences and Engineering Council (NSERC)
of Canada  Discovery Grant Accelerator Supplement,
awarded to Gail~S.~K.~Wolkowicz.
}

\email{hsut1@math.mcmaster.ca}

\usepackage[shortlabels]{enumitem} 

\renewcommand{\[}{\begin{equation}\notag\begin{aligned}}
\renewcommand{\]}{\end{aligned}\end{equation}}
\newcommand{\beq}[1]{\begin{equation}\label{#1}\begin{aligned}}



\begin{document}

\maketitle

\centerline{\scshape Ting-Hao Hsu$^\dag$}
\medskip
{\footnotesize
\centerline{Department of Mathematics and Statistics}
\centerline{McMaster University}
\centerline{Hamilton, Ontario, L8S 4K1, Canada}
}

\begin{abstract}
We consider planar systems of predator-prey models
with small predator death rate $\epsilon>0$.
Using geometric singular perturbation theory
and Floquet theory,
we derive characteristic functions
that determines the location and the stability of
relaxation oscillations as $\epsilon\to 0$.
When the prey-isocline has a single interior local extremum,
we prove that the system has a unique nontrivial periodic orbit,
which forms a relaxation oscillation.
For some systems
with prey-isocline possessing two interior local extrema,
we show that
either the positive equilibrium is globally stable,
or the system has exact two periodic orbits.
In particular,
for a predator-prey model with the Holling type IV functional response
we derive a threshold value of the carrying capacity
that separates these two outcomes.
This result supports the so-called paradox of enrichment.
\end{abstract}

\section{Introduction}
\label{sec_intro}

\begin{figure}[t]
\centering
\begin{tabular}{cc}
\frame
{\includegraphics[trim = 2.3cm 1.2cm 1cm .2cm, clip, width=.44\textwidth]{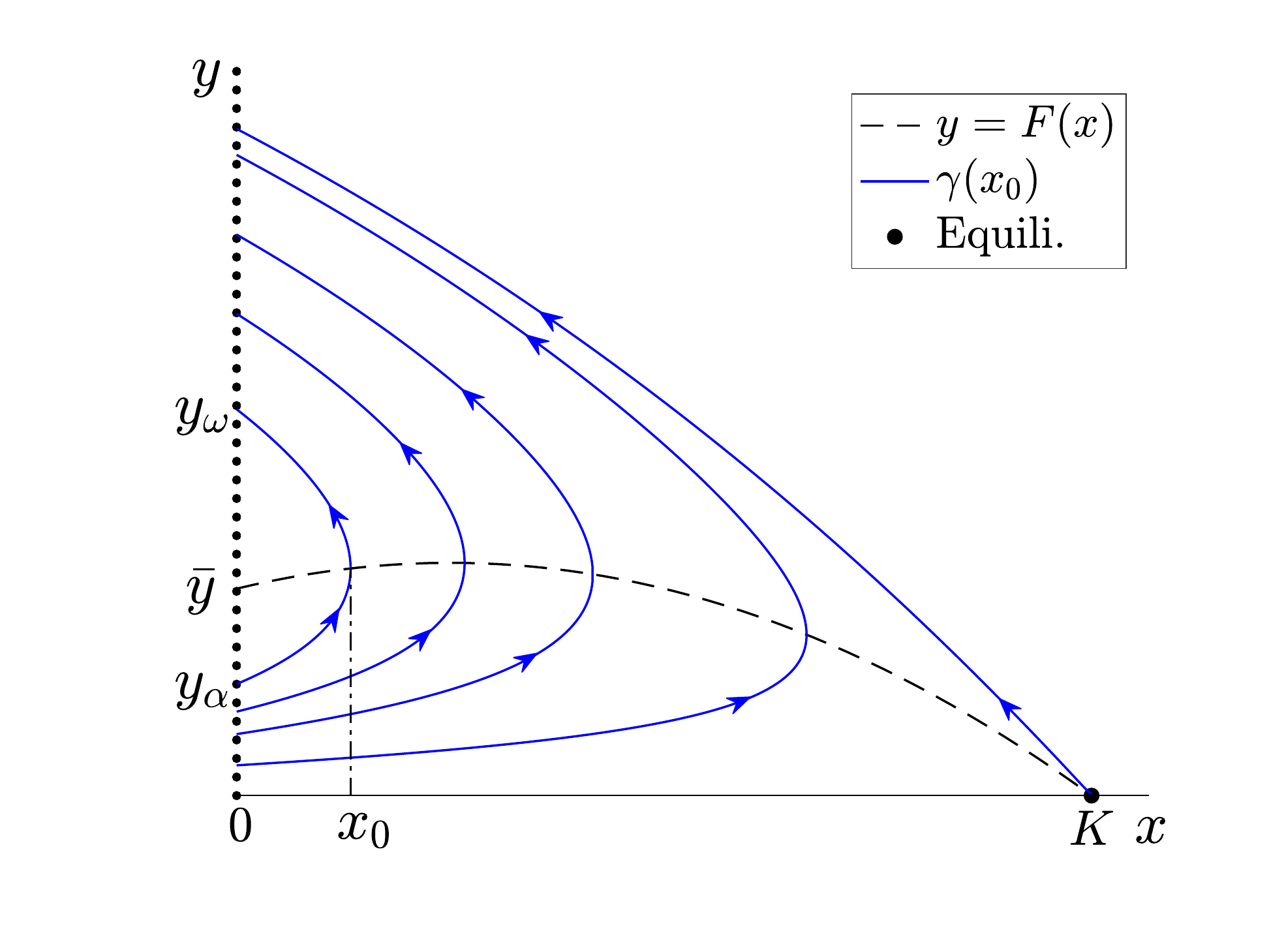}}
&
\frame
{\includegraphics[trim = 2.3cm 1.2cm 1cm .2cm, clip, width=.44\textwidth]{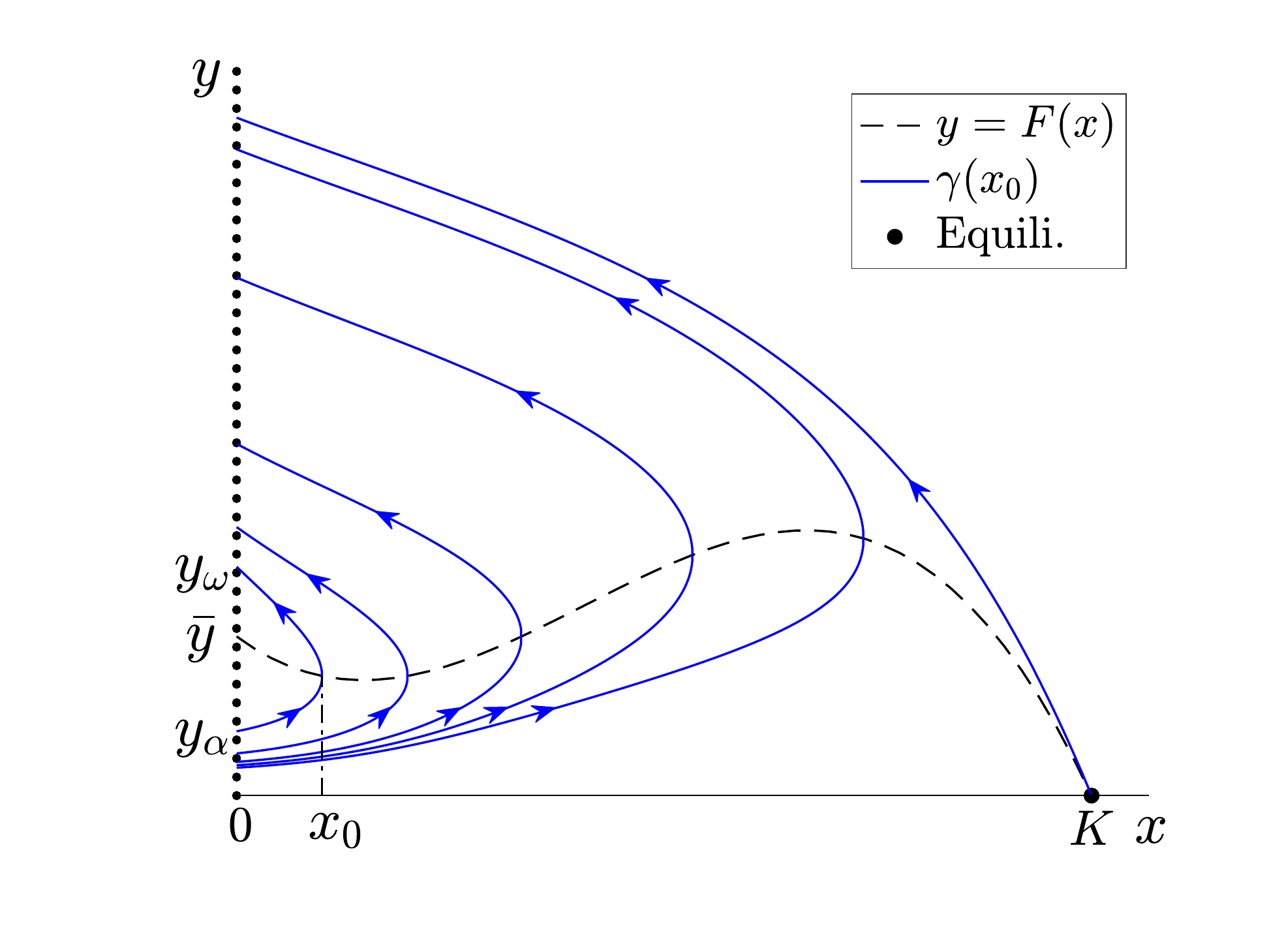}}
\\
(A)
&
(B)
\end{tabular}
\caption{
Typical phase portraits for \eqref{fast_xy} with
(A) $p(x)=mx/(a+x)$
and (B) $p(x)=mx/(x^2+a)$, $a>0$.
The $y$-axis is a set of critical points.
A family of heteroclinic orbits 
is bounded
by the $x$- and $y$-axis
and a trajectory with alpha-limit set being the point $(0,K)$.
For each $x_0\in (0,K)$,
the trajectory $\gamma(x_0)$
passing through
the prey-isocline at
$(x_0,F(x_0))$
has the alpha- and omega-limit sets
being the points $(0,y_\alpha(x_0))$ and $(0,y_\omega(x_0))$, respectively.
}
\label{fig_Gamma}
\end{figure}

The classical Gause-type predator-prey system
with logistic growth of the prey
takes the form
\beq{deq_xy}
  &\dot{x}= rx\left(1-\frac{x}{K}\right)- yp(x)\equiv p(x)(F(x)-y)\\
  &\dot{y}= y(-\epsilon+ cp(x))
\]
where $x(t)$ and $y(t)$ denote densities of the prey and predator populations at time $t$, respectively.
Parameters in the system are:
intrinsic growth rate of the prey $r$,
carrying capacity of the environment $K$,
yield rate $c$,
and death rate of the predator $\epsilon$.
The functional response $p(x)$,
which describes the change in the density of the prey attacked per unit time per predator,
is continuously differentiable and satisfies
\beq{cond_p}
  p(0)=0,\;\;
  p'(0)>0,\;\;
  \text{and}\;\;
  p(x)>0\;\;\forall\;x>0.
\]
Predator-prey systems that possess limit cycles
can be used to explain oscillatory phenomena
in real-world data,
such as the lynx-snowshoe hare cycles \cite{Albrecht:1973,May:1972}.

Limit cycles 
of predator-prey systems
have been studied extensively in the literature.
For certain classes of systems,
the uniqueness of limit cycles has been proved by
Cheng \cite{Cheng:1981},
Kuang and Freedman \cite{Kuang:1988},
Hsu and Huang \cite{Hsu:1995},
Kooij and Zegeling \cite{Kooij:1996},
Sugie \cite{Sugie:1998},
and Xiao and Zhang \cite{Xiao:2003}.
Bifurcation analysis for predator-prey systems
has been investigated
by many researcher, e.g.\ 
Wolkowicz \cite{Wolkowicz:1988},
Zhu, Campbell and Wolkowicz \cite{Zhu:2002},
and Ruan and Xiao \cite{Ruan:2001}.

Singular perturbations in predator-prey systems
have been studied in various contexts.
For a model of two predators competing for the same prey,
when the prey population grows much faster than the predator populations,
the existence of a relaxation oscillation
has been proved by Liu, Xiao and Yi
\cite{Liu:2003}.
For predator-prey systems
with Holling type III or IV,
when the death and the yield rates of the predator
are small and proportional to each other,
the canard phenomenon
and the cyclicity of limit periodic sets
have been investigated by 
Li and Zhu \cite{Zhu:2013}.
For a class of
the Holling-Tanner model,
when the intrinsic growth rate of the predator is sufficiently small,
the existence of a relaxation oscillation
has been proved by
Ghazaryan, Manukian and Schecter
\cite{Ghazaryan:2015}.
For a model of one predator and two prey
with rapid predator evolution,
singular periodic orbits
that correspond to approximated periodic orbits 
have been constructed by Piltz et~al.\ \cite{Piltz:2017}.
Another relevant work was done by Li et~al. \cite{Li:2016},
in which 
the existence of
relaxation oscillations has been proved for some epidemic models,
but the uniqueness of limit cycles was not provided.

In this paper we study the dynamics of \eqref{deq_xy} as $\epsilon\to 0$.
When $\epsilon=0$,
the system is reduced to \beq{fast_xy}
  \dot{x}
  = p(x)(F(x)-y),
  \quad
  \dot{y}= cyp(x),
\]
where $F(x)=rx(1-x/K)/p(x)$.
As indicated by Li and Zhu \cite{Zhu:2013},
system \eqref{fast_xy} has a family of heteroclinic orbits as described below.
The $x$-isocline
for \eqref{fast_xy}
consists of the $y$-axis
and the curve $y=F(x)$.
By condition \eqref{cond_p},
the limit of $F(x)$
as $x\to 0$ exists and equals $r/p'(0)>0$.
We define \[
\bar{y}=F(0)=\lim_{x\to 0}F(x)>0.
\]
Also note that $F(x)>0$ for all $x\in [0,K)$ and $F(K)=0$.
It is easy to show that there is a unique trajectory
that has the point $(K,0)$ as its alpha-limit set.
Bounded by that trajectory
and the $x$- and $y$-axes
is a family of heteroclinic orbits connecting points on the $y$-axis
(see Figure \ref{fig_Gamma}).

For each $x_0\in (0,K)$,
let $\gamma(x_0)$
be the trajectory of \eqref{fast_xy}
passing through the point $(x_0,F(x_0))$.
Note that $\gamma(x_0)$
approaches the $y$-axis
in both positive and negative time. 
Define $y_\alpha(x_0)$ and $y_\omega(x_0)$
to be the values such that 
the points $(0,y_\alpha(x_0))$
and $(0,y_\omega(x_0))$
are the alpha-
and omega-limit points,
respectively, of ($x_0,F(x_0))$
for \eqref{fast_xy}.
Clearly $y_\alpha(x_0)< \bar{y}< y_\omega(x_0)$
for all $x\in (0,K)$,
where $\bar{y}=F(0)$.

Near the invariant set $\{(x,y): x=0\}$ for \eqref{deq_xy},
the system is mainly governed by
\begin{equation}\label{slow_xy}
  x=0,\quad
  \dot{y}= -\epsilon y.
\end{equation}
Let $\sigma(x_0)$ be the segment of the orbit of \eqref{slow_xy}
going from $(0,y_\omega(x_0))$ to $(0,y_\alpha(x_0))$.

The idea of {\em geometric singular perturbation theory}
\cite{Fenichel:1979,Jones:1995,Kuehn:2016}
is that solutions of the full system
can potentially be obtained
by joining trajectories of the limiting systems.
The limiting systems \eqref{fast_xy} and \eqref{slow_xy}
provide a family of uncountably many loops
following the route
$(0,y_\alpha)\overset{\gamma}{\longrightarrow}
(0,y_\omega)\overset{\sigma}{\longrightarrow}
(0,y_\alpha)$
defined by
\begin{equation}\label{def_gamma12}
  \Gamma(x_0)=\gamma(x_0)\cup \sigma(x_0)
\end{equation}
that depends continuously on $x_0\in (0,K)$.
Each $\Gamma(x_0)$ is a candidate of
the limiting configuration of periodic orbits of \eqref{deq_xy} as $\epsilon\to 0$.

Among this family of
candidates $\Gamma(x_0)$,
using a variation of the phenomenon of bifurcation delay,
which we will describe in Section \ref{sec_criteria},
generically all but finitely many of the candidates can be  excluded:
A necessary condition for
$\Gamma(x_0)$ 
to admit a relaxation oscillation,
i.e.\ for $\Gamma(x_0)$
to be the limit of periodic orbits of \eqref{deq_xy} as $\epsilon\to 0$,
is \beq{chi_zero}
  \chi(x_0)=0,
\]
where \beq{def_chi}
  \chi(x_0)
  = \int_{y_\alpha(x_0)}^{y_\omega(x_0)}
  \frac{y-\bar{y}}{y}\;dy.
\] 
Furthermore,
using a variation of the Folquet theory,
we show that if, additionally, \beq{lambda_zero}
  \lambda(x_0)\ne 0,
\] where  \beq{def_lambda}
  \lambda(x_0)
  = \int_{y_\alpha(x_0)}^{y_\omega(x_0)}
  \frac{F'(X(y,x_0))}{y}\;dy
\] and $X(y,x_0)$ is the parametrization of $\gamma(x_0)$,
then
\eqref{chi_zero} is also a sufficient condition.

\begin{figure}[t]
\begin{center}
\begin{tabular}{cc}
\frame
{\includegraphics[trim = 1.2cm .5cm 1.2cm .5cm, clip, width=.46\textwidth]{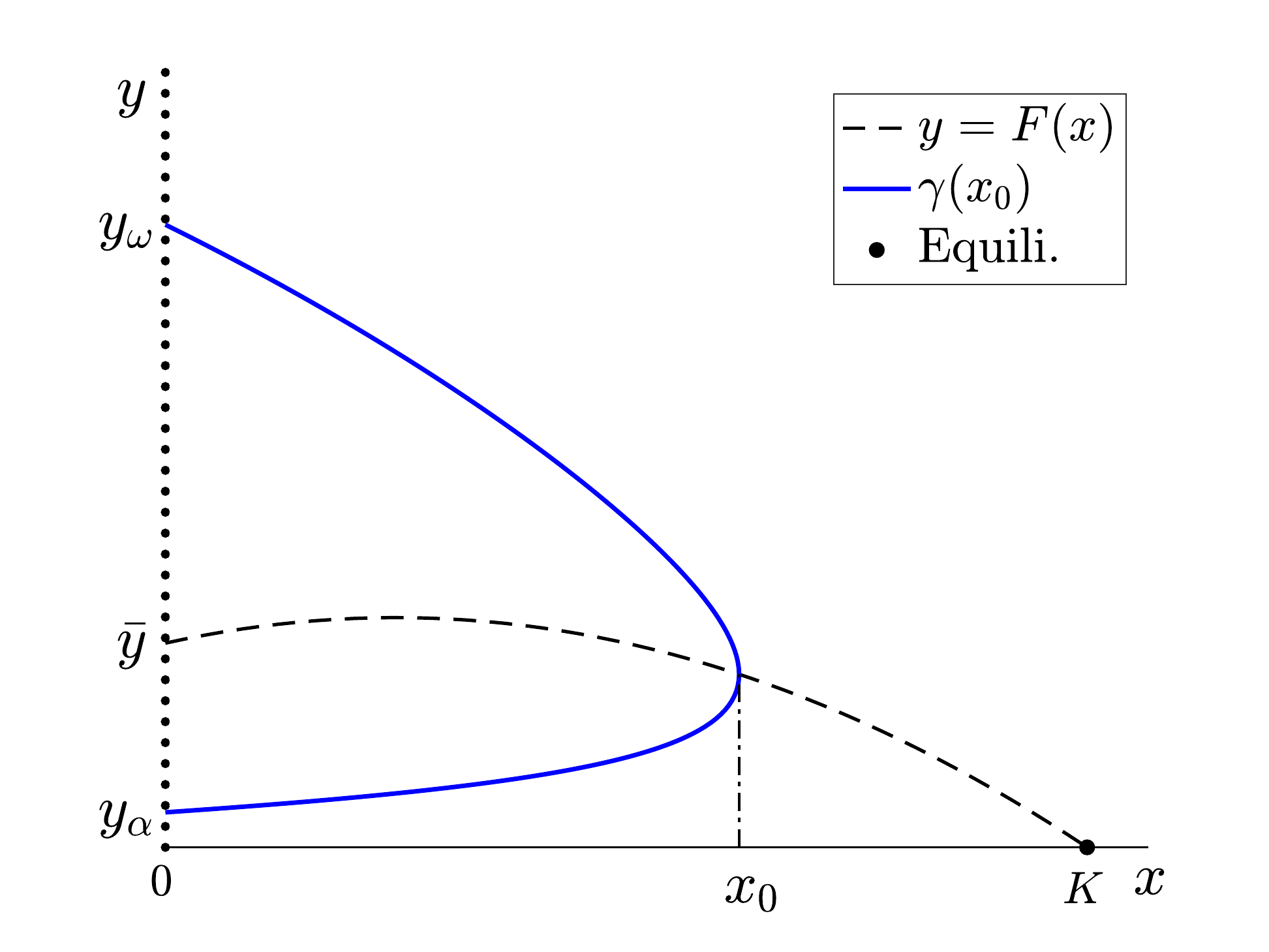}}
&
\frame
{\includegraphics[trim = 1.2cm .5cm 1.2cm .5cm, clip, width=.46\textwidth]{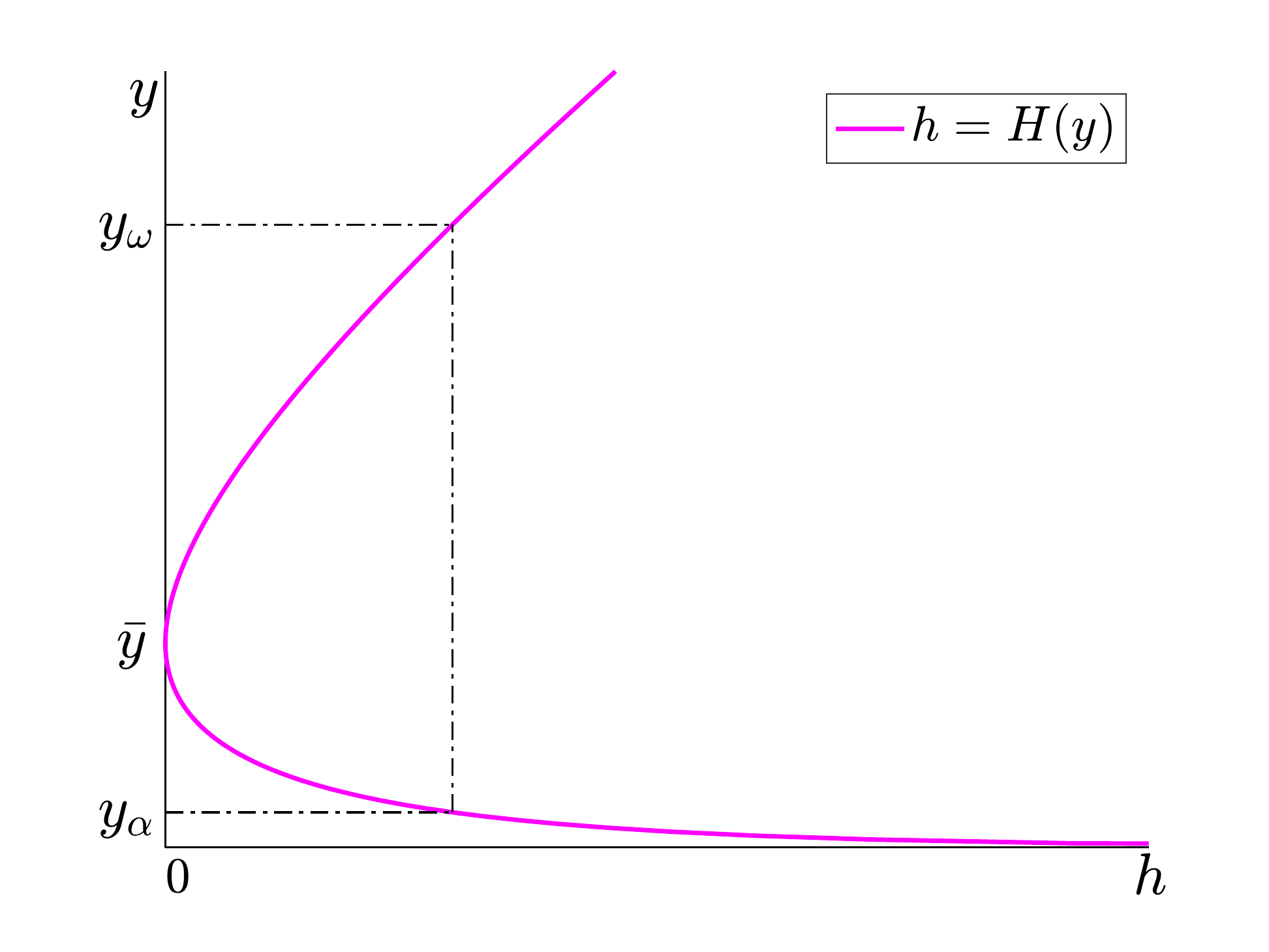}}
\end{tabular}
\end{center}
\caption{
The equation $\chi(x_0)=0$ means that
$H(y_\omega(x_0))= H(y_\alpha(x_0))$.
}
\label{fig_H2_HLog}
\end{figure}

In the case of the Holling type II functional response $p(x)=mx/(a+x)$ for \eqref{deq_xy},
the existence of a periodic orbit $\ell_\epsilon$
that is attracting
is known \cite{Cheng:1981,Kuang:1988nonuniqueness},
and Hsu and Shi \cite{Hsu:2009} proved that
the minimal period $T_\epsilon$ of $\ell_\epsilon$
satisfies $C_1\epsilon^{-1}< T_\epsilon< C_2\epsilon^{-1}$ for some $0<C_1<C_2$,
which implies that $\ell_\epsilon$ forms a relaxation oscillation.
Our main theorems complete their results.
We show that $\epsilon T_\epsilon$ actually converges,
and the trajectory of $\ell_\epsilon$ approaches a certain configuration.
Our results
also cover models
with $p(x)=m(1-e^{-ax})$ and $p(x)=m\log(1+ax)$.

In the case of the Holling type IV functional response $p(x)=mx/(ax^2+1)$ for \eqref{deq_xy},
the function $F(x)=rx(1-x/K)/p(x)$ can have two interior local extrema.
When $\epsilon$ is smaller than and close to the interior local minimal point of $F(x)$,
Xiao~and~Zhu \cite{Xiao:2006} proved
that, under certain conditions,
the system has exactly two limit cycles
for parameters
near a subcritical Hopf bifurcation point.
Our theorem is complementary to their results.
Assuming
the value of $aK^2$ is large enough,
and $c>0$ is small enough,
we show that
system \eqref{deq_xy} has exactly two limit cycles
when $\epsilon$ is sufficiently small.
Multiple limit cycles
were also observed
for various functional responses
by other authors \cite{Hofbauer:1990,Huang:2014,Kuang:1988,Wrzosek:1990,Zhu:2002}.

This result supports the {\em paradox of enrichment}
proposed by Rosenzweig \cite{Rosenzweig:1971}
and studied by Freedman and Wolkowicz \cite{Freedman:1986}.
This paradox says that
enrichment of the environment
(i.e.\ increasing the carrying capacity $K$)
may lead to destabilization of the coexistence equilibrium.
When $K$ is below a threshold $K_*$,
our theorem shows that,
for any fixed small $c>0$
and all sufficiently small $\epsilon$,
the coexistence equilibrium attracts all trajectories
in the interior of the first quadrant.
Hence the prey population eventually remains
of order $\epsilon$,
and does not become exponentially small.
On the contrary,
when $K$ is above $K_*$,
there are two limit cycles,
and all trajectories outside the inner limit cycle
are attracted by the outer limit cycle.
The prey population
remains a small size of order $\exp(-1/\epsilon)$
on the outer limit cycle
for a long timespan of order $1/\epsilon$.
Hence the prey is vulnerable to catastrophic perturbations.

In Section \ref{sec_criteria} 
we state our main criteria
for the location and stability of relaxation oscillations,
and the proof is given in Section \ref{sec_proof}.
Applications of the criteria
in the cases where
the prey-isocline has
one and two interior local extrema are discussed
in Sections \ref{sec_1hump} and \ref{sec_2hump}, respectively.
A discussion of our results is in Section \ref{sec_discussion}.

\begin{remark}
The growth function in \eqref{deq_xy},
instead of being the logistic growth $rx(1-x/K)$,
can be replace by any function $q(x)$
that satisfies $q(0)=0$, $q'(0)=r>0$ and $(K-x)q(K)>0$ for $x\ne K$.
Our analysis is still valid with $F(x)=q(x)/p(x)$.
\end{remark}

\begin{remark}
The function $\chi(x)$
defined in \eqref{def_chi} 
can be written as \[
  \chi(x)=H(y_\omega(x))-H(y_\alpha(x)), 
\] where \beq{def_H}
  H(y)=y-\bar{y}- \bar{y}\log(y/\bar{y}),
\]
which takes the form as a part of the Lyapunov function
introduced in \cite{Hsu:1978global}.

Another way to rewrite $\chi(x)$ is \beq{chi_int2}
  \chi(x_0)
  = \int_{y_\alpha(x_0)}^{y_\omega(x_0)}\frac{F(X(y,x_0)))-\bar{y}}{y}\;dy,
\]
where $X(y,x_0)$ is the parametrization of $\gamma(x_0)$.
This expression can be derived from \eqref{def_chi}
using \[
  \int_{y_\alpha(x_0)}^{y_\omega(x_0)}\frac{F(X(y))-y}{y}\;dy
  =\int_{-\infty}^{\infty} c\,\frac{\dot{x}(t)}{\dot{y}(t)}\;\dot{y}(t)dt
  =cx(t)\big|_{t=-\infty}^{\infty}
  =0.
\]
Note the integrand in expression \eqref{def_lambda} of $\lambda(x_0)$
is formally the derivative,
with respect to $X$,
of the integrand in \eqref{chi_int2} of $\chi(x_0)$.
Thus $\lambda(x_0)$ can be interpreted
as a type of derivative of $\chi(x_0)$.
\end{remark}

\begin{remark}
Our criteria
of the existence of periodic orbits
can be related
to the finiteness part of Hilbert's 16th problem
\cite{Dumortier:1994,Roussarie:1998},
which conjectures the boundedness of the number of limit cycles
of polynomial vector fields
of a fixed order,
in the following way.
Consider the functional response $p(x)$ in \eqref{deq_xy}
as a rational function where the
numerator and denominator are polynomials of fixed orders.
Then \eqref{deq_xy} is equivalent to a polynomial vector field of a fixed order.
If one is able to vary
the coefficients of $p(x)$
to obtain arbitrarily many non-degenerate roots of $\chi(x)$ in the interval $(0,K)$
while keeping the
orders of the numerator and denominator of $p(x)$ fixed,
then one obtains a negative answer to the conjecture.
\end{remark}

\section{The Criteria}
\label{sec_criteria}

The phenomenon of 
{\em bifurcation delay}
(see \cite{De-Maesschalck:2008,De-Maesschalck:2016,THHsu:2017}
and the references therein),
also known as {\em Pontryagin delay},
or {\em delay of stability loss},
occurs typically in systems of the form
 \beq{sf_ab}
  \dot{a}= \epsilon f(a,b,\epsilon),\quad \dot{b}=b\, g(a,b,\epsilon),
\]
where $f$ and $g$ are $C^1$ functions
that satisfy \beq{cond_turning_ab}
  f(a,0,0)> 0\quad\forall\; a\in\mathbb R
  \qquad\text{and}\qquad
  a\,g(a,0,0)>0\quad\forall\; a\ne 0.
\]
When $\epsilon=0$, the limiting system is \beq{fast_ab}
  \dot{a}= 0,\quad \dot{b}=b\, g(a,b,0).
\]
As illustrated in Figure \ref{fig_bifdelay}{(A)},
the half line $\{a<0,b=0\}$ is a set of attracting critical points for \eqref{fast_ab},
and $\{a>0,b=0\}$ is a set of repelling critical points.
Fix an $a_0<0$ and $\delta>0$.
For each $\epsilon>0$,
let $\Gamma_\epsilon(a_0)$ be a trajectory of \eqref{sf_ab}
that starts from the point $(a_0,\delta)$.
Denote by $(a_{1,\epsilon},\delta)$
the point where $\Gamma_\epsilon$
intersects the cross section $\{b=\delta\}$.
Then it can be proved that \beq{limit_a1eps}
  a_{1,\epsilon}\to a_1
  \quad\text{as }\epsilon\to 0,
\]
where $a_1$ is implicitly defined by \beq{entryexit_ab}
  \int_{a_0}^{a_1} \frac{g(a,0,0)}{f(a,0,0)}\; da= 0,
\]
or $a_1=\infty$ if \eqref{entryexit_ab} does not hold for any finite $a_1>0$.
The relation between $a_0$ and $a_1$
is called the {\em entry-exit relation}.

To deal with \eqref{deq_xy},
a tempting approach is to try to find a change of coordinates near the $y$-axis
that converts the system into \eqref{sf_ab}.
However, such a transformation is not possible
because the tangent lines of the trajectories of \eqref{fast_xy}
approach the $y$-axis near $y=\bar{y}$
(see Figure \ref{fig_bifdelay}{(B)}).
Therefore a variation of bifurcation delay is needed.
We will show in Theorem \ref{thm_entryexit2} that
the assertion \eqref{limit_a1eps} still holds
for systems of the form
\beq{sf_ab2}
  \dot{a}= \epsilon f(a,b,\epsilon)+ b\, h(a,b,\epsilon),
  \quad
  \dot{b}= b\,g(a,b,\epsilon),
\]
where $f$, $g$ and $h$ are $C^2$ functions that satisfy \eqref{cond_turning_ab}.

Throughout this paper
we assume
the functional response $p(x)$
is a $C^2$ function.
Then the entry-exit relation \eqref{entryexit_ab} for \eqref{deq_xy},
with $(y,x)$ playing the role of $(a,b)$ in \eqref{sf_ab2},
can be stated as follows:
Trajectories of \eqref{deq_xy} entering the vicinity of the $y$-axis near $(0,y_\omega)$
must leave near the point $(0,y_0)$
satisfying $H(y_\omega)=H(y_0)$,
where $H(y)$ is defined in \eqref{def_H}.

Here we give a heuristic argument
showing that
if $\chi(x_0)\ne 0$, $x_0\in (0,K)$,
then no trajectory of \eqref{deq_xy}
lies entirely near $\Gamma(x_0)$.
For any $x_0\in (0,K)$,
if $\chi(x_0)>0$ (resp.\ $\chi(x_0)<0$),
then $H(y_\omega(x_0))>H(y_\alpha(x_0))$ (resp.\ $H(y_\omega(x_0))<H(y_\alpha(x_0))$).
Because $H(y)$ is monotone on the interval $(0,\bar{y})$
and $\lim_{y\to 0^+}H(y)=\infty$,
$H(y_\omega(x_0))=H(y_0)$
for a unique $y_0\in (0,y_\alpha(x_0))$  (resp.\ $y_0\in (y_\alpha(x_0),\bar{y})$).
Since the point $(0,y_0)$ is not contained in $\gamma(x_0)$,
a trajectory of \eqref{deq_xy} starting near $\gamma(x_0)$
does not stay near $\Gamma(x_0)=\gamma(x_0)\cup \sigma(x_0)$
after the first time it leaves the vicinity of the $y$-axis.

\begin{figure}[t]
\begin{center}
\begin{tabular}{cc}
\frame
{\includegraphics[trim = 2.7cm .9cm 1cm .7cm, clip, width=.46\textwidth]{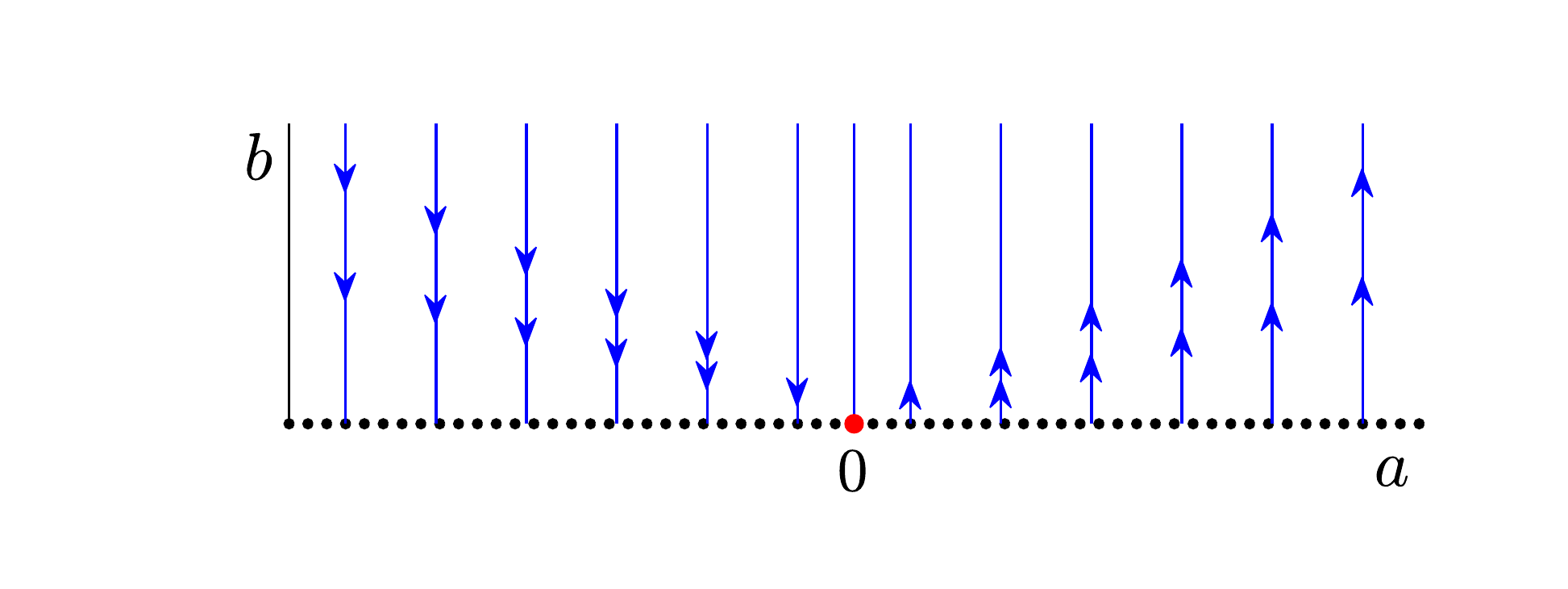}}
&
\frame
{\includegraphics[trim = 2.7cm .9cm 1cm .7cm, clip, width=.46\textwidth]{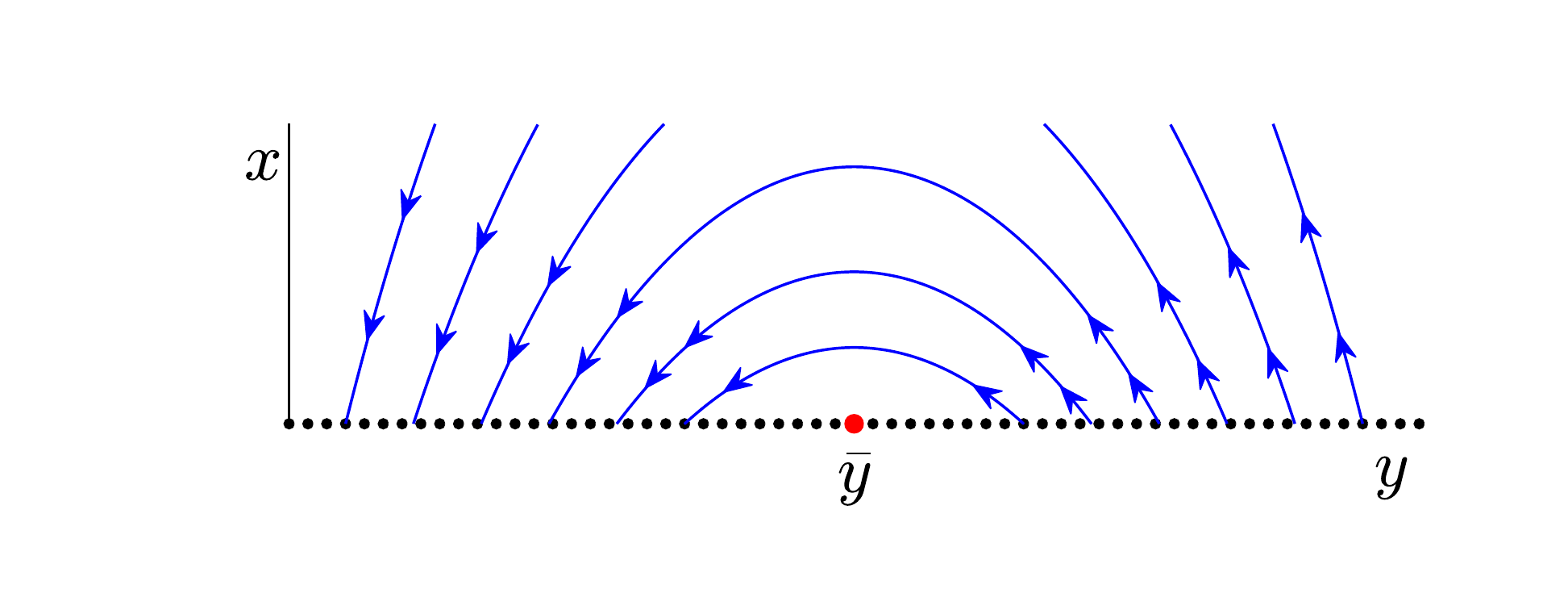}}
\\
(A)
&
(B)
\end{tabular}
\end{center}
\caption{
Typical phase portraits for
(A) the limiting system \eqref{fast_ab} of \eqref{sf_ab},
(B) the limiting system \eqref{fast_xy} of \eqref{deq_xy}
near the $y$-axis.
The tangent lines of the trajectories in (B) approach the $y$-axis near $y=\bar{y}$,
so \eqref{deq_xy} and \eqref{sf_ab} are not equivalent near the $y$-axis.
}
\label{fig_bifdelay}
\end{figure}

\begin{theorem}\label{thm_y0}
Assume that $x_0\in (0,K)$ satisfies $\chi(x_0)=0$,
where $\chi(x)$ is defined in \eqref{def_chi}.
Let $\lambda(x)$ be the function defined in \eqref{def_lambda}.
If $\lambda(x_0)\ne 0$,
then for any sufficiently small $\epsilon>0$,
there is a periodic orbit ${\ell}_\epsilon$ of \eqref{deq_xy}
in a $O(\epsilon)$-neighborhood of $\Gamma(x_0)=\gamma(x_0)\cup \sigma(x_0)$
defined in \eqref{def_gamma12}.
Moreover, 
${\ell}_\epsilon$ is orbitally locally asymptotically stable if $\lambda(x_0)<0$,
and is orbitally unstable if $\lambda(x_0)>0$.
The minimal period of $\ell_\epsilon$, denoted by $T_\epsilon$, satisfies \beq{est_Teps}
  T_\epsilon= \frac{1}{\epsilon}\Big(
    \log\left(\frac{y_\omega(x_0)}{y_\alpha(x_0)}\right)+o(1)
  \Big)
  \quad\text{as }\epsilon\to 0.
\]
Conversely, 
if $\chi(x_0)\ne 0$,
then for any point $z_1$ in the interior of the trajectory $\gamma(x_0)$,
there is a neighborhood $U$ of $z_1$
such that no periodic orbit of \eqref{deq_xy}
intersects $U$ for any sufficiently small $\epsilon>0$.
\end{theorem}

The proof of Theorem \ref{thm_y0} is deferred to Section \ref{sec_proof}.

By Theorem \ref{thm_y0}
and the uniform boundedness of solutions of \eqref{deq_xy},
given any neighborhood $V$ of the union of curves \beq{union_Gamma}
  \bigcup_{x_0:\; \chi(x_0)=0}\Gamma(x_0),
\]
for all sufficiently small $\epsilon>0$,
all periodic orbits of \eqref{deq_xy} lie entirely in $V$.
Some particular cases of this observation are given in the corollary below.

By condition \eqref{cond_p},
system \eqref{deq_xy}
has a unique positive equilibrium $E_*$
for all sufficiently small $\epsilon>0$.
By calculating the Jacobian matrix,
$E_*$ is locally asymptotically stable if $F'(0)<0$,
and is unstable if $F'(0)>0$.
By a phase portrait analysis,
it is easy to show that every periodic orbit must surround $E_*$.

\begin{corollary}\label{cor_gas}
The following statements hold.
\begin{enumerate}[label=$\mathrm{(\roman*)}$]
\item
Assume that $F'(0)>0$ and $\chi(x)$ has no root
in the interval $(0,K)$.
Then,
for all sufficiently small $\epsilon>0$,
the positive equilibrium of \eqref{deq_xy}
is globally asymptotically stable.
\item
Assume that $F'(0)\ne 0$
and $\chi(x)$ has exactly $n$ distinct roots,
$x_0<x_1<\cdots<x_{n-1}$, in $(0,K)$.
If \beq{lambda_switch}
  \lambda(x_{j-1})\lambda(x_{j})<0
  \quad\text{for }\; j=1,2,\dots,n-1,
\] then \eqref{deq_xy} has exactly $n$ periodic orbits
for any sufficiently small $\epsilon>0$,
and these periodic orbits form $n$ relaxation oscillations
as $\epsilon\to 0$.
\end{enumerate}
\end{corollary}

Throughout this paper,
global stability of equilibria or periodic orbits
is with respect to the set of non-stationary points in the first quadrant.
Since the stable manifold of the two boundary equilibria
are the $x$- or $y$-axes,
system \eqref{deq_xy} 
has no closed loop formed by heteroclinic orbits.
Therefore, the only possible limit cycles for \eqref{deq_xy} are periodic orbits,
and, by the Poincar\'{e}-Bendixon Theorem,
the global stability of the unique positive equilibrium,
for sufficiently small $\epsilon>0$,
is equivalent to the non-existence of a periodic orbit.

\begin{remark}
In statement (i) of Corollary \ref{cor_gas},
the condition $F'(0)>0$
can be replaced by $F'(0)\ne 0$
for the same conclusion.
However, by Lemma \ref{lem_xmin} below,
when $F'(0)<0$,
$\chi(x)$ must have at least one root in $(0,K)$.
In statement (ii),
condition \eqref{lambda_switch}
can be replaced by $\lambda(x_j)\ne 0\;\forall\;j$.
However, since two adjacent periodic orbits cannot have
the same non-neutral stability,
these two conditions are equivalent.
\end{remark}

To prove Corollary \ref{cor_gas}, we need the following lemma,
which is covered by Wolkowicz \cite{Wolkowicz:1988}.
We provide a proof of this lemma here for completeness.

\begin{lemma}\label{lem_xfar}
If $F'(x)>0$ on $(0,\widehat{x}]$ 
or $F'(x)<0$ on $(0,\widehat{x}]$ 
for some $\widehat{x}>0$,
then
no periodic orbit of \eqref{deq_xy}
lies entirely in the strip $\{(x,y): 0<x<\widehat{x}\}$
for any sufficiently small $\epsilon>0$.
\end{lemma}

\begin{proof}[Proof of Lemma \ref*{lem_xfar}]
Assume that $\ell$
is a periodic orbit for \eqref{deq_xy} for some $\epsilon>0$.
We parameterize $\ell$ 
by \[
  \ell
  = \{(x,Y_-(x): x\in [x_L,x_R]\}
  \cup \{(x,Y_+(x): x\in [x_L,x_R]\},
\]
with $Y_-(x)\ge F(x)$ and $Y_+(x)\ge F(x)$,
and also by $\ell=\{(x(t),y(t)): t\in [0,T]\}$,
where $T$ is the minimal period of $\ell$.
Then Floquet multiplier can be calculated as \beq{mu_xfar}
  \mu(\ell)
  &\equiv
  \int_0^{T}
  \mathrm{div}\begin{pmatrix}
    p(x)\big(F(x)-y\big)\\
    y(-\epsilon+ cp(x))
  \end{pmatrix}\;dt
  \\[.5em]
  &=\int_0^{T} p'(x)\big(F(x)-y\big)+ p(x)F'(x)-\epsilon+ cp(x)\; dt\\[.5em]
  &=\int_{\ell} \frac{p'(x)}{p(x)}\;dx
  + \int_{\ell}\frac{F'(x)}{F(x)-y}\;dx
  + \int_{\ell}\frac{1}{y}\;dy.
\]
The first and the third integrals equal zero by the periodicity of $\ell$, so \[
  \mu(\ell)
  = \int_{x_L}^{x_R} \frac{F'(x)}{F(x)-Y_-(x)}\;dx
  + \int_{x_L}^{x_R} \frac{F'(x)}{Y_+(x)-F(x)}\;dx.
\]

Assume that $F'(x)>0$ on $(0,\widehat{x}]$.
Suppose $\ell$ is contained in $\{(x,y): 0<x<\widehat{x}\}$.
Then it follows from \eqref{mu_xfar} that $\mu(\ell)>0$.
By standard Floquet theory,
$\ell$ is orbitally unstable.
By calculating the Jacobian matrix,
it is easy to show
that the positive equilibrium 
of \eqref{deq_xy}
is unstable
for all sufficiently small $\epsilon>0$.
By the above calculation of $\mu(\ell)$,
any periodic orbit that lies between $\ell$ and the positive equilibrium
must be orbitally unstable.
This contradicts the Poincar\'{e}-Bendixon Theorem.

The case that $F'(x)<0$ on $(0,\widehat{x}]$ can treated similarly.
\end{proof}

\begin{figure}[t]
\begin{center}
\begin{tabular}{cc}
\frame
{\includegraphics[trim = 2cm 1.7cm 1cm .3cm, clip, width=.44\textwidth]{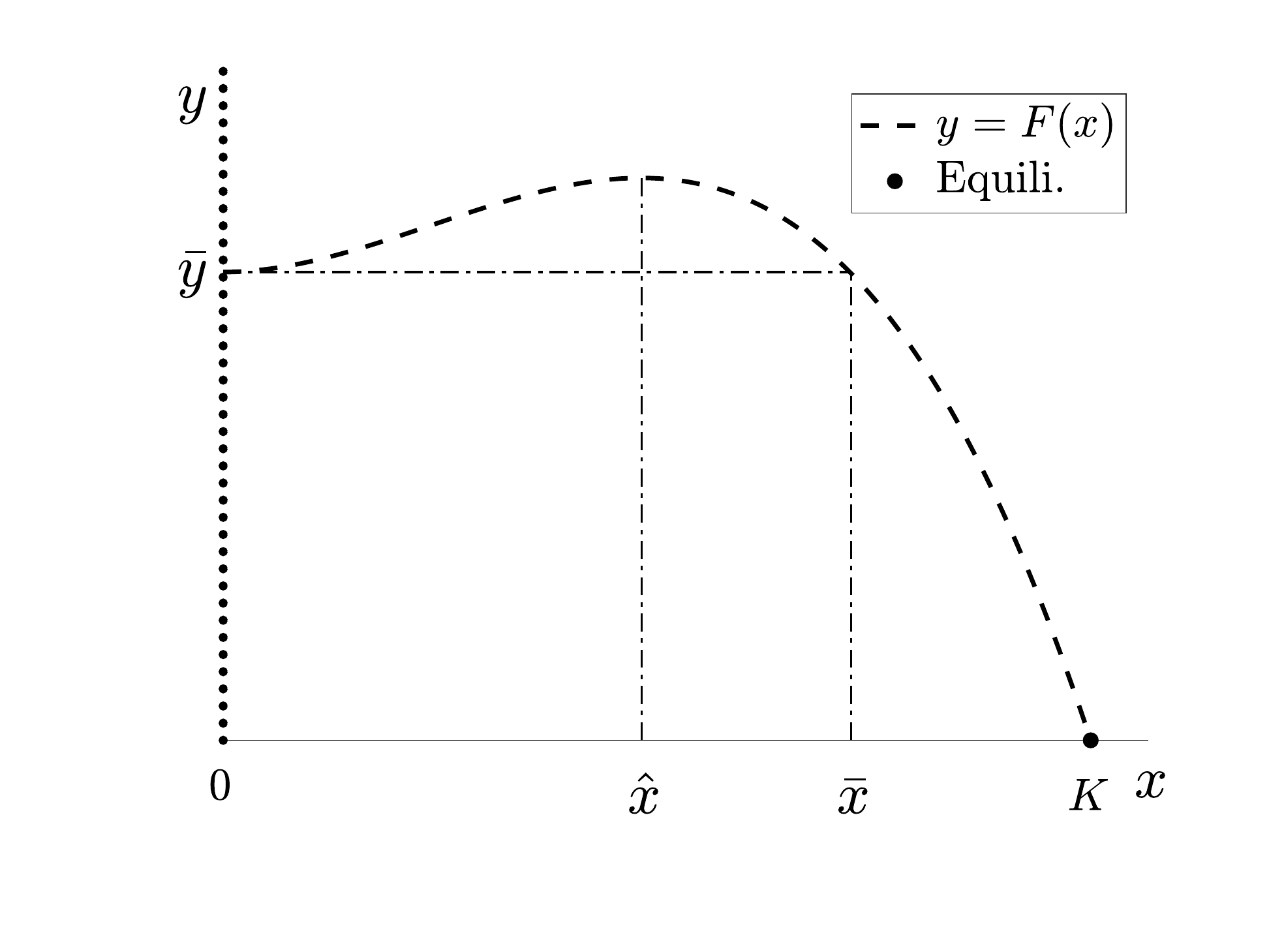}}
&
\frame
{\includegraphics[trim = 2cm 1.7cm 1cm .3cm, clip, width=.44\textwidth]{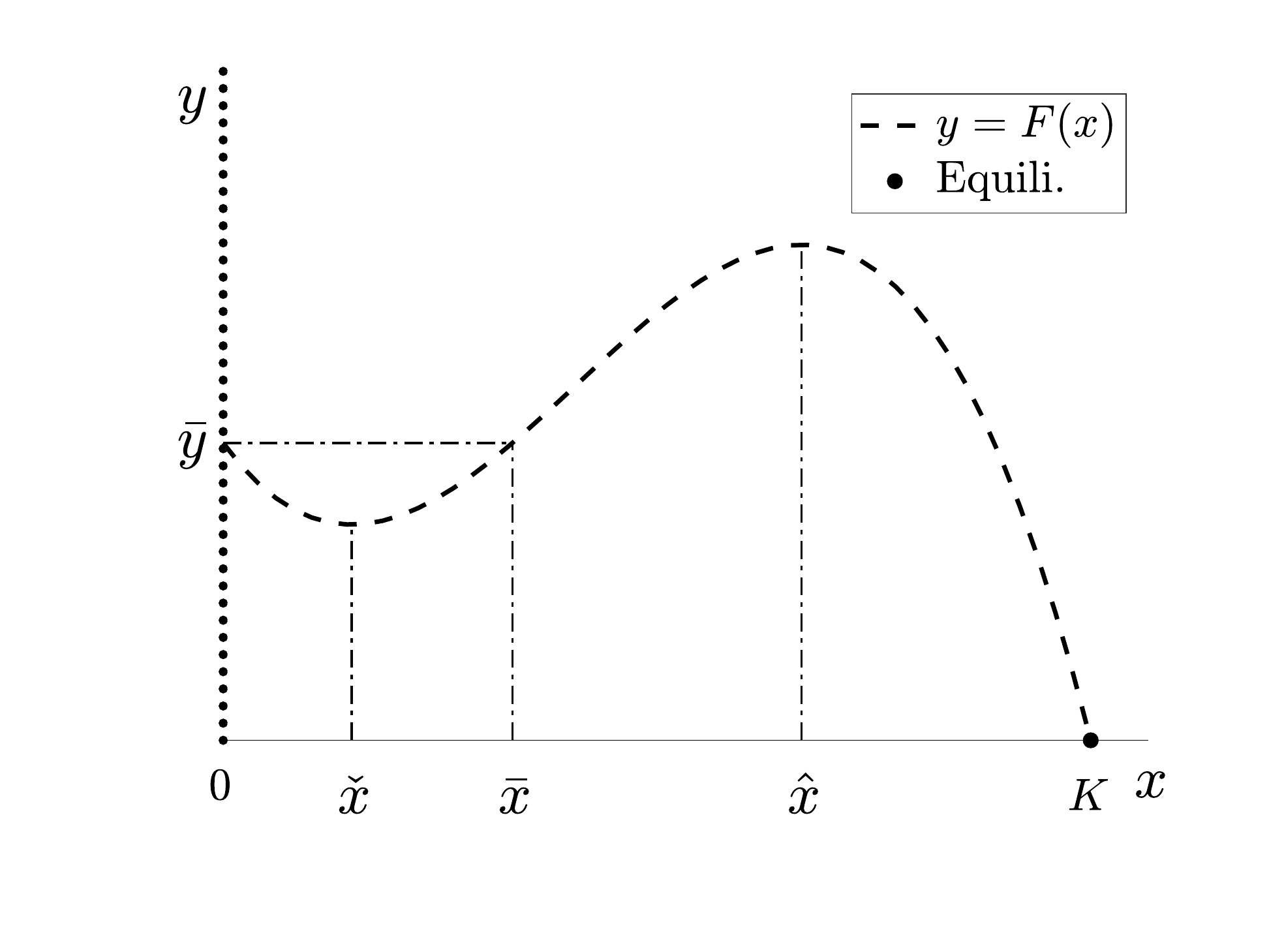}}
\\
(A)
&
(B)
\end{tabular}
\end{center}
\caption{
(A) The one-hump condition \eqref{cond_1hump}:
$F'(x)>0$ for $x\in (0,\widehat{x})$ and $F'(x)<0$ for $x\in (\widehat{x},K)$.
(B) The two-hump condition \eqref{cond_2hump}:
$F'(x)<0$ for $x\in (0,\widecheck{x})\cup (\widehat{x},K)$
and $F'(x)>0$ for $x\in (\widecheck{x},\widehat{x})$.
}
\label{fig_labels}
\end{figure}

\begin{proof}[Proof of Corollary \ref*{cor_gas}]
For (i),
by Theorem \ref{thm_y0}, Lemma \ref{lem_xfar} and
the uniform boundedness of solutions,
system \eqref{deq_xy} has no periodic orbit
in the first quadrant
for any sufficiently small $\epsilon>0$.
Since the stable manifolds of the boundary equilibria
do not intersect the interior of the first quadrant
$E_*$ is globally asymptotically stable
by the Poincar\'{e}-Bendixon Theorem.

For (ii),
since $\chi(x_j)=0$ and $\lambda(x_j)\ne 0$ for $j=0,1,\dots,n-1$,
by Theorem \ref{thm_y0}
there are $n$ periodic orbits of \eqref{deq_xy}
corresponding to $n$ relaxation oscillations.
By Theorem \ref{thm_y0}, Lemma \ref{lem_xfar} and
the uniform boundedness of solutions of \eqref{deq_xy},
there are no other periodic orbits.
\end{proof}

\section{The Case with One Hump}
\label{sec_1hump}
In this section we study the case that 
the function $F(x)$ has a single interior local extremum on the interval $(0,K)$.
That is, 
the prey-isocline
satisfies the one-hump condition (see Figure \ref{fig_labels}{(A)}):
For some $\widehat{x}\in (0,K)$,
\beq{cond_1hump}
  F'(x)>0\;\;
  \forall\;x\in (0,\widehat{x})
  \quad\text{and}\quad
  F'(x)>0\;\;
  \forall\;x\in (\widehat{x},K).
\]

\begin{theorem}\label{thm_1hump}
Assume \eqref{cond_1hump}.
Then, for any sufficiently small $\epsilon>0$,
system \eqref{deq_xy}
has a unique periodic orbit ${\ell}_\epsilon$.
Moreover, as $\epsilon\to 0$,
${\ell}_\epsilon$ approaches $\Gamma(x_0)$
defined in \eqref{def_gamma12},
for a unique $x_0\in (\widehat{x},K)$ satisfying $\chi(x_0)=0$.
\end{theorem}

Note that Theorem \ref{thm_1hump}
does not require the concavity of $F(x)$.
Examples \ref{ex_H4_1hump} and \ref{ex_Ivlev} below
contain cases with non-concave $F(x)$ for which the theorem applies
(see Remarks \ref{rmk_noncovex_H4} and \ref{rmk_noncovex_Ivlev}).

\begin{remark}
In the case that $F(x)$ is concave,
since it is a special case of \eqref{cond_1hump},
by Theorem \ref{thm_1hump}
there is a globally orbitally asymptotically stable periodic orbit
for any sufficiently small $\epsilon>0$.
A predator-prey model
with concave prey-isocline
and at least two limit cycles
has been constructed by Hofbauer and So \cite{Hofbauer:1990}
with parameters such that the equilibrium point
undergoes a subcritical Hopf bifurcation.
Using a continuation tool, e.g.\ \texttt{XPPAUT} \cite{Ermentrout:2002},
one can show numerically that,
as $\epsilon$ varies in the model in \cite{Hofbauer:1990},
the inner limit cycle, which emerged from the subcritical Hopf point,
encounters a saddle-node bifurcation,
switches to the outer limit cycle,
and survives as $\epsilon$ decreases.
Hence the system has a single limit cycle when $\epsilon>0$ becomes small.
\end{remark}

To prove the theorem,
the first step
is the observation that \beq{chi_limit_K}
  \lim_{x\to K^-} \chi(x)=-\infty.
\]
To derive \eqref{chi_limit_K},
we write $\chi(x)$ defined in \eqref{def_chi} by $\chi(x)=H(y_\omega(x))-H(y_\alpha(x))$,
where $H(y)$ is the function defined in \eqref{def_H}.
Note that the functions $y_\alpha(x)$ and $y_\omega(x)$
satisfy \beq{y_limits_K}
  \lim_{x\to K^-}y_\alpha(x)= 0
  \quad\text{and}\quad
  \lim_{x\to K^-}y_\omega(x)
  \quad\text{exists and is finite}.
\] 
Since $\lim_{y\to 0^+}H(y)=\infty$, 
\eqref{chi_limit_K} follows from \eqref{y_limits_K}.

If $\chi(\bar{x})>0$ for some $\bar{x}\in (0,K)$,
then by \eqref{chi_limit_K}
the continuous function $\chi(x)$
has at least one root in $(\bar{x},K)$.
The following lemma
asserts that this is the case when $F(x)>F(0)$ on $(0,\bar{x})$.

\begin{lemma}\label{lem_xmin}
The following statements hold.
\begin{enumerate}[label=$\mathrm{(\roman*)}$]
\item
If $F(x)>F(0)$ on $(0,\bar{x})$ for some $\bar{x}>0$,
then $\chi(x)<0$ for all $x\in (0,\bar{x}]$,
and $\chi(x)$ has at least one root in $(\bar{x},K)$.
If, in addition,
$F(\bar{x})=F(0)$ and, for some $\widehat{x}$ and $x_0$
with $\widehat{x}<\bar{x}<x_0$,
\beq{cond_F_climax}
  F'(x)>0\;\;\forall\; x\in (0,\widehat{x})
  \quad\text{and}\quad
  F'(x)<0\;\;\forall\; x\in (\widehat{x},x_0),
\]
then $\lambda(x_0)<0$.
\item
If $F(x)<F(0)$ on $(0,\bar{x})$ for some $\bar{x}>0$,
then $\chi(x)>0$ for all $x\in (0,\bar{x}]$.
If, in addition,
$F(\bar{x})=F(0)$ and, for some $\widecheck{x}$ and $x_0$
with $\widecheck{x}<\bar{x}<x_0$,
\beq{cond_F_valley}
  F'(x)<0\;\;\forall\; x\in (0,\widecheck{x})
  \quad\text{and}\quad
  F'(x)>0\;\;\forall\; x\in (\widecheck{x},x_0),
\]
then $\lambda(x_0)>0$.
\end{enumerate}
\end{lemma}

See Figure \ref{fig_labels} for an illustration of the notations in Lemma \ref{lem_xmin}.

Note that the above lemma only guarantees a root of $\chi(x)$
in the first case.
In the second case,
$\chi(x)$ may not have a root in $(0,K)$.

\begin{proof}[Proof of Lemma \ref*{lem_xmin}]
System \eqref{fast_xy}
has the same trajectories as \beq{fast_xy_normal}
  \dot{x}=F(x)-y,\quad
  \dot{y}=cy.
\] Along any solution of \eqref{fast_xy_normal}, we have \beq{deq_HF}
  \frac{d}{dt}(H(y)-cx)
  =c(F(x)-F(0)).
\]
For each $x_0\in (0,K)$,
parameterize $\gamma(x_0)$ by \beq{def_Ypm}
  \gamma(x_0)
  =\{(x,Y_-(x)): x\in (0,x_0]\}
  \cup \{(x,Y_+(x)): x\in (0,x_0]\}
\] 
with \beq{ineq_Ypm_F}
  Y_-(x)<F(x)
  \quad\text{and}\quad
  Y_+(x)>F(x)
  \quad\forall\; x\in (0,x_0).
\]
By \eqref{fast_xy_normal} and \eqref{deq_HF}, \beq{chi_Ypm}
  \chi(x_0)
  =\int_0^{x_0}
  \big(F(x)-F(0)\big)
  \left(
    \frac{1}{F(x)-Y_-(x)}+\frac{1}{Y_+(x)-F(x)}
  \right) dx.
\]

If $F(x)>F(0)$ on $(0,\bar{x})$,
then by \eqref{chi_Ypm} we have $\chi(x_0)>0$ for all $x_0\in (0,\bar{x}]$.
From \eqref{chi_limit_K},
it follows that the continuous function $\chi$ has at least one root in $(\bar{x},K)$.

If $F(x)<F(0)$ on $(0,\bar{x})$,
then by \eqref{chi_Ypm} we have $\chi(x_0)<0$ for all $x_0\in (0,\bar{x}]$.

Next we investigate the sign of $\lambda(x_0)$.
By \eqref{fast_xy_normal} 
and the definition of $\lambda(x_0)$ in \eqref{def_lambda},
\beq{lambda_int_F}
  \lambda(x_0)
  =\int_{0}^{x_0}F'(x)\left(
    \frac{1}{F(x)-Y_-(x)}+\frac{1}{Y_+(x)-F(x)}
  \right) dx.
\]
If $x_0\in (\bar{x},K)$ satisfies \eqref{cond_F_climax},
then by \eqref{ineq_Ypm_F} and \eqref{lambda_int_F}
we have \[
  \lambda(x_0)
  <\int_0^{\bar{x}}
  F'(x)\left(
    \frac{1}{F(x)-Y_-(x)}+\frac{1}{Y_+(x)-F(x)}
  \right) dx.
\]
Since $Y_+(x)$ is decreasing and $Y_-(x)$ is increasing,
by condition \eqref{cond_F_climax}, \[
  \lambda(x_0)
  &<\int_0^{\bar{x}}
  F'(x)\left(
    \frac{1}{F(x)-Y_-(\widehat{x})}+\frac{1}{Y_+(\widehat{x})-F(x)}
  \right) dx
  \\[.5em]
  &=\left.\log\left(
    \frac{F(x)-Y_-((\widehat{x})}{Y_+(\widehat{x})-F(x)}
  \right)\right|_{x=0}^{\bar{x}}
  = 0.
\] The last equality followed from the condition $F(\bar{x})=F(0)$.

Similarly,
if $x_0\in (\bar{x},K)$ satisfies \eqref{cond_F_valley},
then by \eqref{ineq_Ypm_F} and \eqref{lambda_int_F}
we have \[
  \lambda(x_0)
  >\int_0^{\bar{x}}
  F'(x)\left(
    \frac{1}{F(x)-Y_-(x)}+\frac{1}{Y_+(x)-F(x)}
  \right) dx,
\]
which implies $\lambda(x_0)>0$ under the condition that $F(\bar{x})=F(0)$.
\end{proof}

\begin{proof}[Proof of Theorem \ref*{thm_1hump}]
Since $F'(0)>0$,
by Lemma \ref{lem_xmin}
the function $\chi$
has at least one root in $(0,K)$.
We claim that $\chi$ has exactly one root in  $(0,K)$.
Suppose $\chi$ has two distinct roots, say $x_0<x_1$.
Then by Lemma \ref{lem_xmin}, \beq{lambda_negative_1hump}
  \lambda(x)<0\quad\forall\; x\in [x_0,x_1].
\]
For all sufficiently small $\epsilon>0$,
by Theorem \ref{thm_y0}
there are 
locally orbitally asymptotically stable
periodic orbits 
$\ell^{(0)}_\epsilon$ and $\ell^{(1)}_\epsilon$
near $\Gamma(x_0)$ and $\Gamma(x_1)$, respectively.
Note that $\Gamma(x_0)$ is enclosed by $\Gamma(x_1)$.
From \eqref{lambda_negative_1hump} and Theorem \ref{thm_y0}
it follows that there is no unstable periodic orbit 
between $\ell^{(0)}_\epsilon$ and $\ell^{(1)}_\epsilon$.
Also note that
no equilibrium lies between $\ell^{(0)}_\epsilon$ and $\ell^{(1)}_\epsilon$.
This contradicts the Poincar\'{e}-Bendixson Theorem.
Therefore $\chi$ has exactly one root.
Denote the unique root of $\chi$ by $x_0$.
By Lemma \ref{lem_xmin}, $\lambda(x_0)<0$.
Hence the result follows from Corollary \ref{cor_gas}.
\end{proof}

The following is a classical model.
We use it as an example for applying Theorem \ref{thm_1hump}
although the results are included in the literature.

\begin{ex}\label{ex_H2}
Consider \eqref{deq_xy} with
Holling type II functional response, namely \beq{pq_H2}
  p(x)=\frac{mx}{a+x},
\] 
where $m$ and $a$ are positive constants.
\begin{itemize}
\item[$\mathrm{(i)}$]
If $a>K$,
then 
for all sufficiently small $\epsilon>0$,
the interior equilibrium is
globally asymptotically stable.
\item[$\mathrm{(ii)}$]
If $a<K$,
then, as $\epsilon\to 0$,
there is a relaxation oscillation
formed by globally asymptotically stable periodic orbits.
\end{itemize}
\end{ex}

\begin{proof}
The function $F(x)=rx(1-x/K)/p(x)$ equals \[
  F(x)= \frac{r}{m}\left(1-\frac{x}{K}\right)(a+x).
\] 
The graph of $F(x)$ is a parabola with maximal point $x=\frac{K-a}{2}$.
If $a<K$, then \eqref{cond_1hump} is satisfied,
and the results in Theorem \ref{thm_1hump} hold.
If $a>K$, then $F(x)<F(0)$ for all $x\in (0,K]$.
By Lemma \ref{lem_xmin}, $\chi(x)$ has no root in the interval $(0,K)$.
From Theorem \ref{thm_1hump} it follows that
\eqref{deq_xy} has no periodic orbit
for all sufficiently small $\epsilon>0$.
\end{proof}

\begin{remark}
In Example \ref{ex_H2},
assertion (i)
is covered in \cite{Hsu:1978competing}.
For assertion (ii),
the existence of a globally asymptotically stable periodic orbits $\ell_\epsilon$
is covered in \cite{Cheng:1981,Kuang:1988}.
Their theorems do not require the smallness of $\epsilon>0$.
The conclusion that $\ell_\epsilon$ forms a relaxation oscillation as $\epsilon\to 0$
is covered in \cite{Hsu:2009}.
\end{remark}

Some numerical simulations for Example \ref{ex_H2}
are presented in Figure \ref{fig_H2eps}.
To produce accurate numerical solutions of \eqref{deq_xy},
in all figures throughout this paper
we adopt the change of variable \[
  (x,y)\mapsto (\log x,y)
\] because the minimum
of the $x$-coordinate
on the periodic orbit $\ell_\epsilon$
is exponentially small.
The numerical solver \texttt{ode23}
in MATLAB \cite{Matlab:2017b}
with tolerance $10^{-6}$ was used.

\begin{figure}[t]
\centering
\begin{tabular}{cc}
\frame
{\includegraphics[trim = 1cm 0cm .5cm 0cm, clip, width=.44\textwidth]{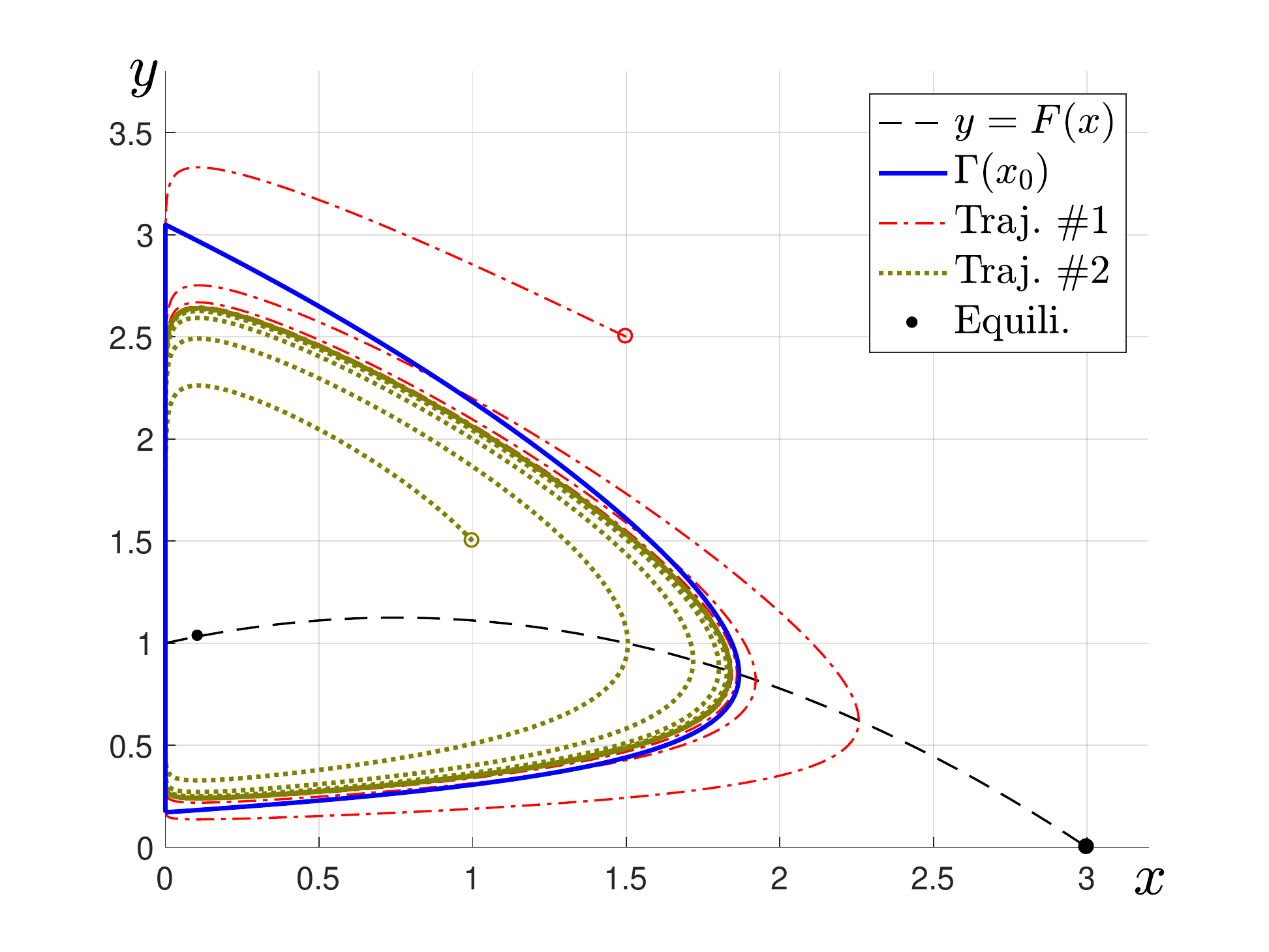}}
&
\frame
{\includegraphics[trim = 1cm 0cm .5cm 0cm, clip, width=.44\textwidth]{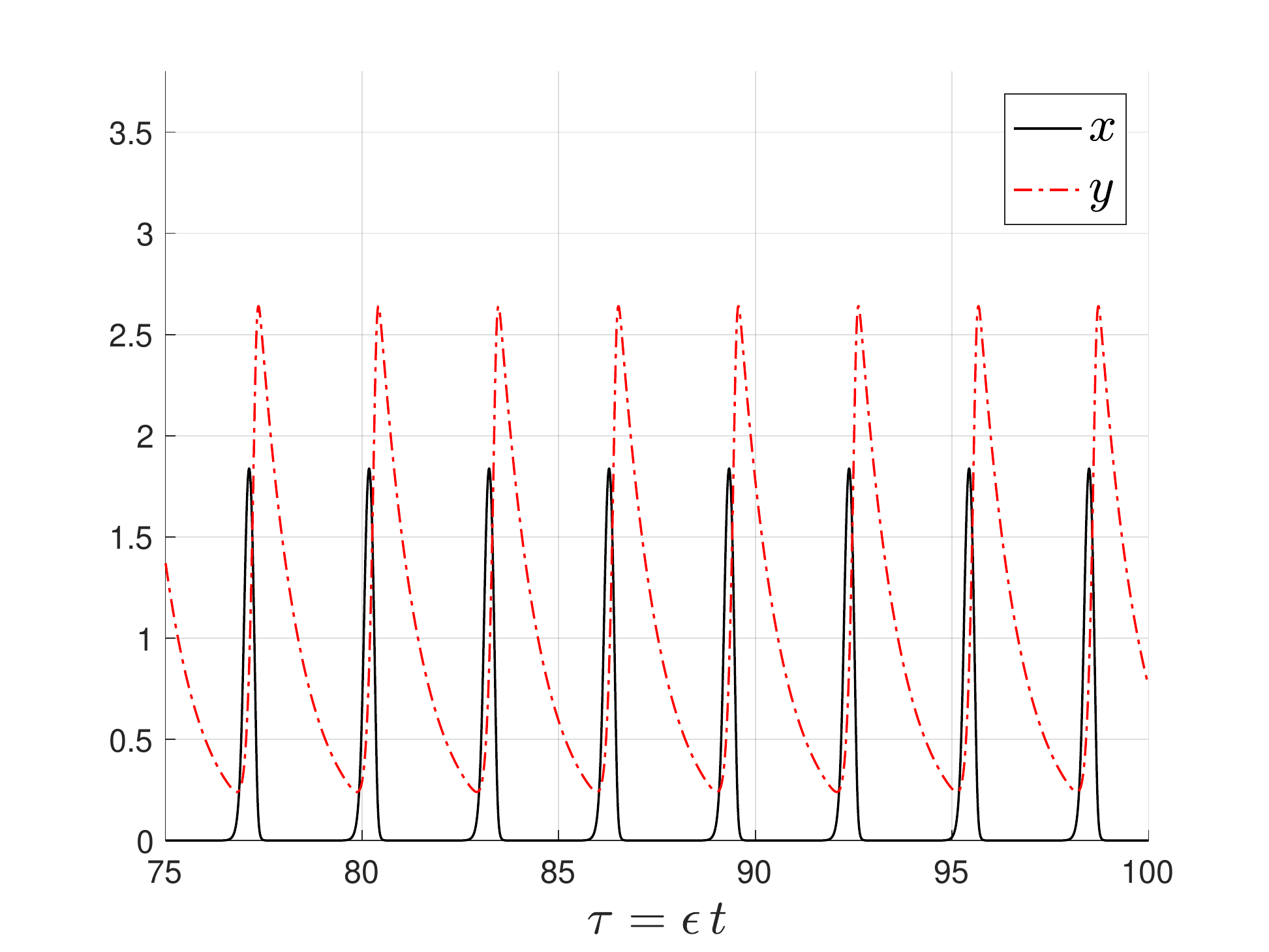}}
\\
(A)
&
(B)
\\[1em]
\frame
{\includegraphics[trim = 1cm 0cm .5cm 0cm, clip, width=.44\textwidth]{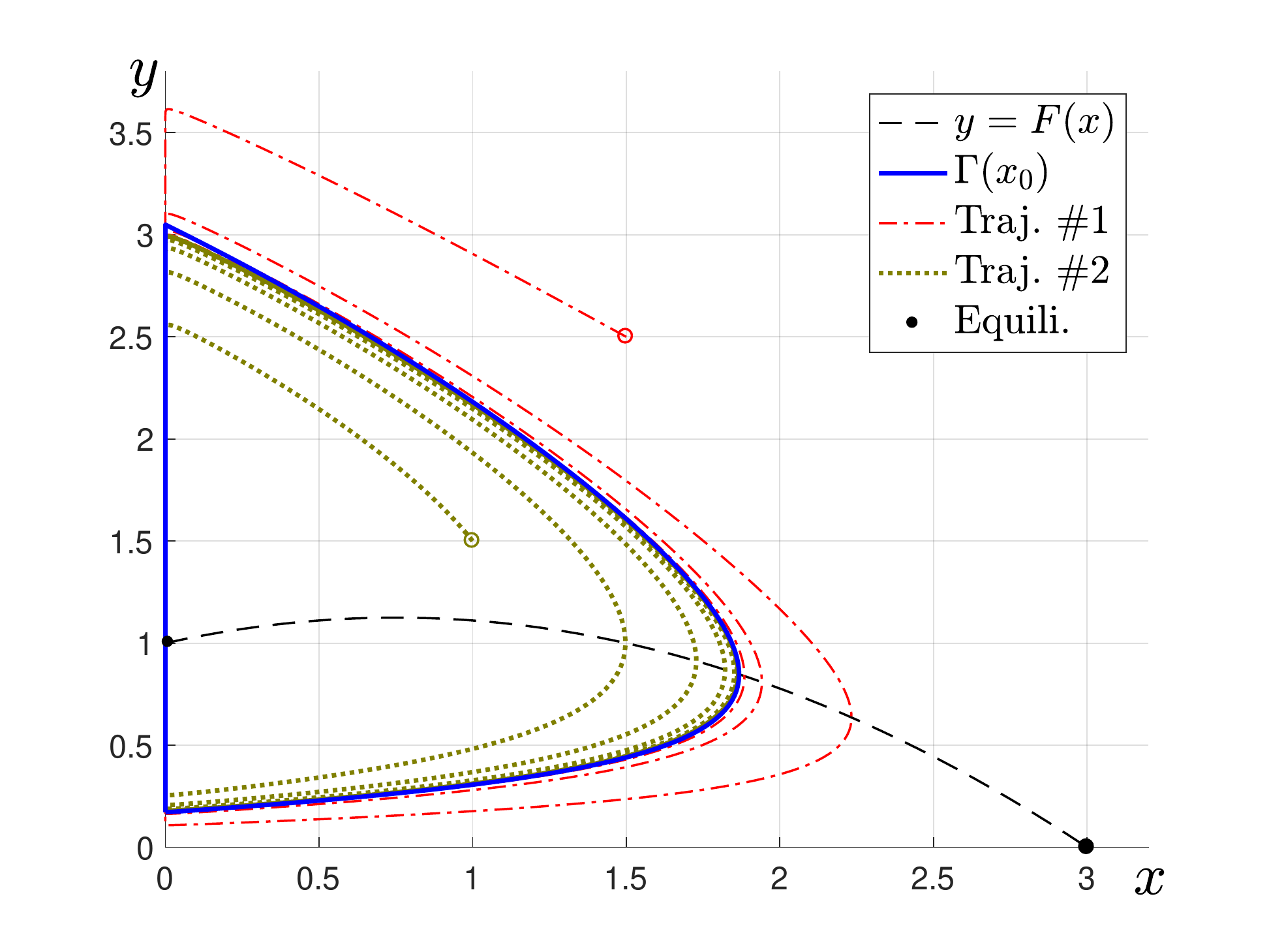}}
&
\frame
{\includegraphics[trim = 1cm 0cm .5cm 0cm, clip, width=.44\textwidth]{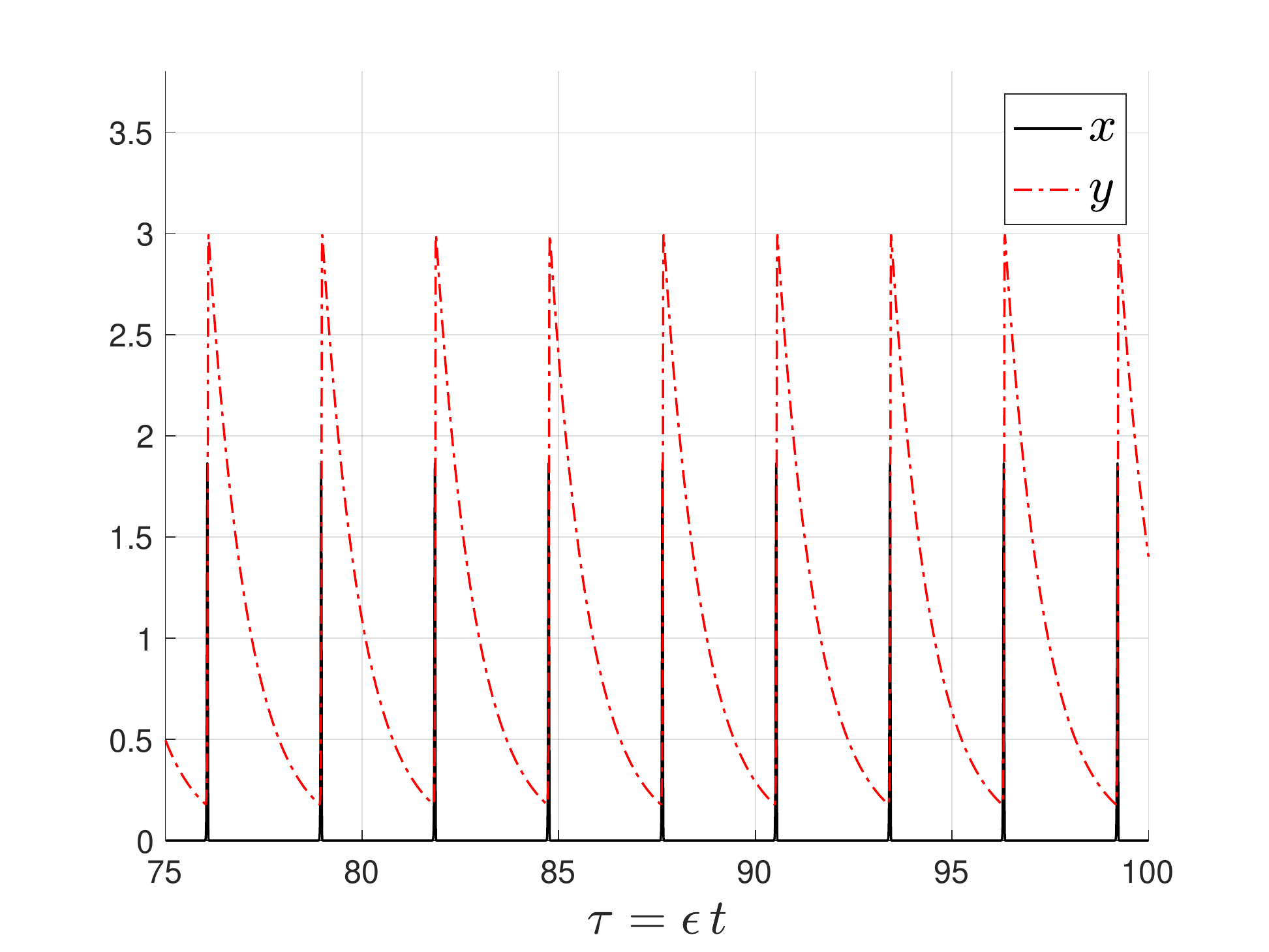}}
\\
(C)
&
(D)
\end{tabular}
\caption{
Forward trajectories for \eqref{deq_xy}
with $p(x)=mx/(a+x)$
and parameters $(r,K,a,m)=(2,3,3,1.5)$, $c=0.5$.
(A)--(B) $\epsilon=0.1$.
(C)--(D) $\epsilon=0.01$.
The initial value for (B) and (D) is $(x,y)=(1.5,2.5)$.
In (A) and (C),
the trajectories approach a periodic orbit $\ell_\epsilon$.
As $\epsilon\to 0$,
$\ell_\epsilon$ approaches 
$\Gamma(x_0)$ given in Theorem \ref{thm_1hump}.
In (B) and (D),
the period of $\ell_\epsilon$
in time $\tau=\epsilon t$ tends to a nonzero value,
and hence the period in time $t$ becomes unbounded.
}
\label{fig_H2eps}
\end{figure}

\begin{ex}\label{ex_H4_1hump}
Consider \eqref{deq_xy} with
the generalized Holling type IV functional response \beq{pq_H4g}
  p(x)=\frac{mx}{ax^2+bx+1},
\]
where $m>0$, $a>0$, and $b>-2\sqrt{a}$
(so that $p(x)>0$ for all $x\ge 0$).

If $b>1/K$,
then, as $\epsilon\to 0$,
there is a relaxation oscillation
formed by globally asymptotically stable periodic orbits.
\end{ex}

\begin{proof}
The function $F(x)=rx(1-K/x)/p(x)$ equals \[
  F(x)= \frac{r}{m}\left(1-\frac{x}{K}\right)(ax^2+bx+1).
\] 
It is straightforward to show that
$F'(0)=b-1/K$ and $F'''(x)=-6a/K$.
Hence $F(x)$ has a single local extremum in the interval $(0,K)$
when $b>1/K$.
Therefore the conclusions in Theorem \ref{thm_1hump} hold.
\end{proof}

\begin{remark}\label{rmk_noncovex_H4}
In Example \ref{ex_H4_1hump},
$F(x)$ is non-concave
when $b<aK$
because $F''(0)=2(a-b/K)$.
Therefore,
when $1/K<b<aK$,
system \eqref{deq_xy}
possesses a non-concave prey-isocline
and a globally stable periodic orbit,
for any sufficiently small $\epsilon>0$,
\end{remark}

\begin{ex}\label{ex_Ivlev}
Consider \eqref{deq_xy} with
the Ivlev's functional response \beq{def_Ivlev}
  p(x)= m(1-e^{-ax}).
\] 
\begin{itemize}
\item[$\mathrm{(i)}$]
If $aK\le 2$,
then 
for all sufficiently small $\epsilon>0$,
the interior equilibrium is
globally asymptotically stable.
\item[$\mathrm{(ii)}$]
If $aK>2$,
then, as $\epsilon\to 0$,
there is a relaxation oscillation
formed by globally asymptotically stable periodic orbits.
\end{itemize}
\end{ex}

\begin{proof}
For $p(x)=m(1-e^{-ax})$,
it has been proved
by Seo and Wolkowicz \cite{Seo:2018}
that the function
$F(x)=rx(1-x/K)/p(x)$
satisfies $F'''(x)<0$ for $x\in (0,K]$, and \[
  F'(0)= \frac{r(Ka-2)}{2Kma}.
\]

When $aK\le 2$,
we have $F'(0)<0$ and $F''(x)<0$ on $(0,K)$,
so $F'(x)<0$ on $(0,K)$.
By Lemma \ref{lem_xmin}, $\chi(x)$ has no root in $(0,K)$.
Hence assertion (i) follows from Corollary \ref{cor_gas}.

When $aK>2$,
we have $F'(0)>0$,
so the facts that $F(0)>0$, $F(K)=0$ and $F'''(x)<0$ on $(0,K]$
imply that
$F(x)$ has a single local extremum in $(0,K)$.
Hence the conclusions in Theorem \ref{thm_1hump} hold.
\end{proof}

\begin{remark}
The existence and uniqueness of a limit cycle
in assertion (ii) of Example \ref{ex_Ivlev}
is covered in \cite{Sugie:1998}.
That paper characterized $(a,\epsilon)$-space
according to the number of periodic orbits.
The conclusion of (i) in Example \ref{ex_Ivlev} is true
without assuming smallness of $\epsilon$
by what
is now a standard Lyapunov function
constructed by Harrison \cite{Harrison:1979}.
\end{remark}

\begin{remark}\label{rmk_noncovex_Ivlev}
In Example \ref{ex_Ivlev},
$F(x)$ is non-concave for $aK<6$
because $F''(0)= \frac{r(aK-6)}{6Km}$,
as shown in \cite{Seo:2018}.
Thus the condition $2<aK<6$
implies that,
for any sufficiently small $\epsilon>0$,
system \eqref{deq_xy}
possesses a non-concave prey-isocline
and a globally stable periodic orbit.
\end{remark}

\begin{ex}\label{ex_log}
Consider \eqref{deq_xy} with
\beq{def_log}
  p(x)= m\log(1+ax)
\] 
where $a,m>0$.
Then, as $\epsilon\to 0$,
there is a relaxation oscillation
formed by globally asymptotically stable periodic orbits.
\end{ex}

\begin{proof}
By Theorem \ref{thm_1hump},
it suffices to show that
the function $F(x)=\frac{r}{mK}x(K-x)/\log(1+ax)$
is concave on the interval $(0,K)$.

By the rescaling $\tilde{x}=ax$ and $\tilde{K}=Ka$, 
we may assume
$F(x)= x(K-x)/\log(1+x)$.
Let $g(x)=1/\log(1+x)$.
Then \[
  F''(x)
  &= -2g(x)+ 2(K-2x)g'(x)+ x(K-x)g''(x)
  \\[.5em]
  &= -2(g(x)+ xg'(x))+ (K-x)(2g'(x)-xg''(x))
  \\[.5em]
  &= -2q_1(x)+ (K-x)q_2(x),
\] where $q_1(x)=g(x)+ xg'(x)$ and $q_2(x)=2g'(x)-xg''(x)$.
It suffices to show that $q_1(x)>0$ and $q_2(x)<0$ for $x>0$.
On one hand, \[
  q_1(x)= \log(1+x)-\frac{x}{1+x}>0.
\]
The last inequality followed from the concavity of $\log(1+x)$
and the fact that $\frac{d}{dx}\log(1+x)= 1/(1+x)$.
On the other hand, \[
  q_2(x)
  &=\frac{-2}{(\log(1+x))^2(1+x)}
  +\frac{x(2+\log(1+x))}{(\log(1+x))^3(1+x)^2}
  \\[.5em]
  &=\frac{-2(\log(1+x))(1+x)+x(2+\log(1+x))}{(\log(1+x))^3(1+x)^2}.
\] Let $q_3(x)$ denotes the numerator of the last expression.
Note that $q_3(0)=0$ and \[
  q_3'(x)
  &=\big(-2-2\log(1+x)\big)
  +\left(2+\log(1+x)+\frac{x}{1+x}\right)
  \\[.5em]
  &=-\log(1+x)+\frac{x}{1+x}
  <0.
\] Hence $q_3(x)<0$,
and therefore $q_2(x)<0$ for $x>0$.
\end{proof}

Before closing this section,
we study the asymptotic behavior of the unique root of $\chi(x)$
given in Theorem \ref{thm_1hump}
as the parameter $c$ tends to $0$.
Note that, in the limiting case that $c=0$, all trajectories for \eqref{fast_xy} are horizontal segments.
Under the one-hump condition \eqref{cond_1hump},
it is easy to show, as illustrated in Figure \ref{fig_c0}{(A)}, that
\beq{x0_limit_1hump}
  y_\alpha(x,c)
  \to
  \begin{cases}
    \bar{y},&\text{if }x\in (0,\bar{x}],
    \\
    F(x),&\text{if }x\in [\bar{x},K),
  \end{cases}
  \quad\text{and}\quad
  y_\omega(x,c)
  \to \begin{cases}
    F(x),& \text{if }x\in (0,\widehat{x}],
    \\
    F(\widehat{x}),& \text{if }x\in [\widehat{x},K),
  \end{cases}
\]
as $c\to 0$.

\begin{figure}[t]
\centering
\begin{tabular}{cc}
\frame
{\includegraphics[trim = 2.4cm 1.2cm 1cm .3cm, clip, width=.44\textwidth]{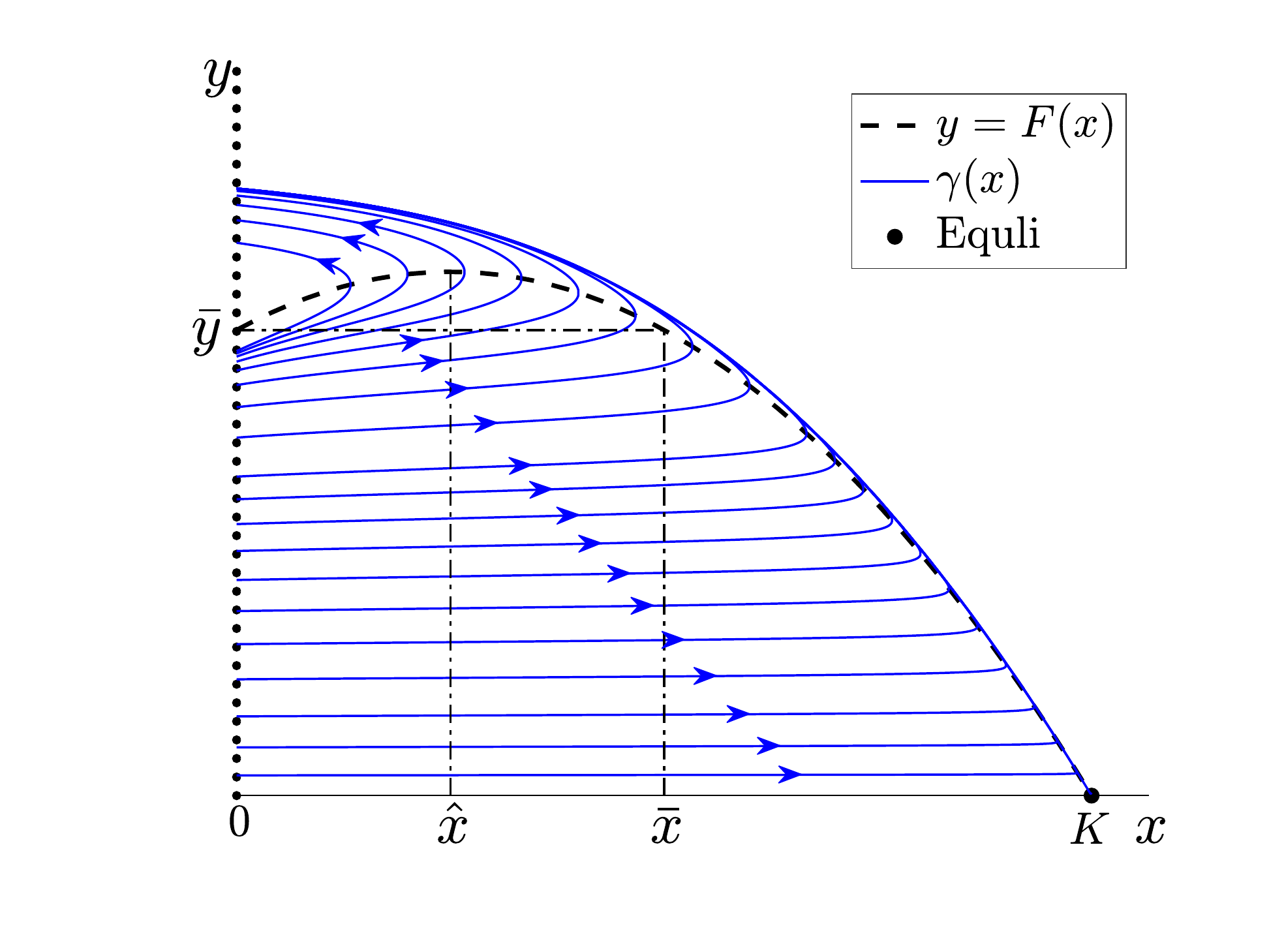}}
&
\frame
{\includegraphics[trim = 2.4cm 1.2cm 1cm .3cm, clip, clip, width=.44\textwidth]{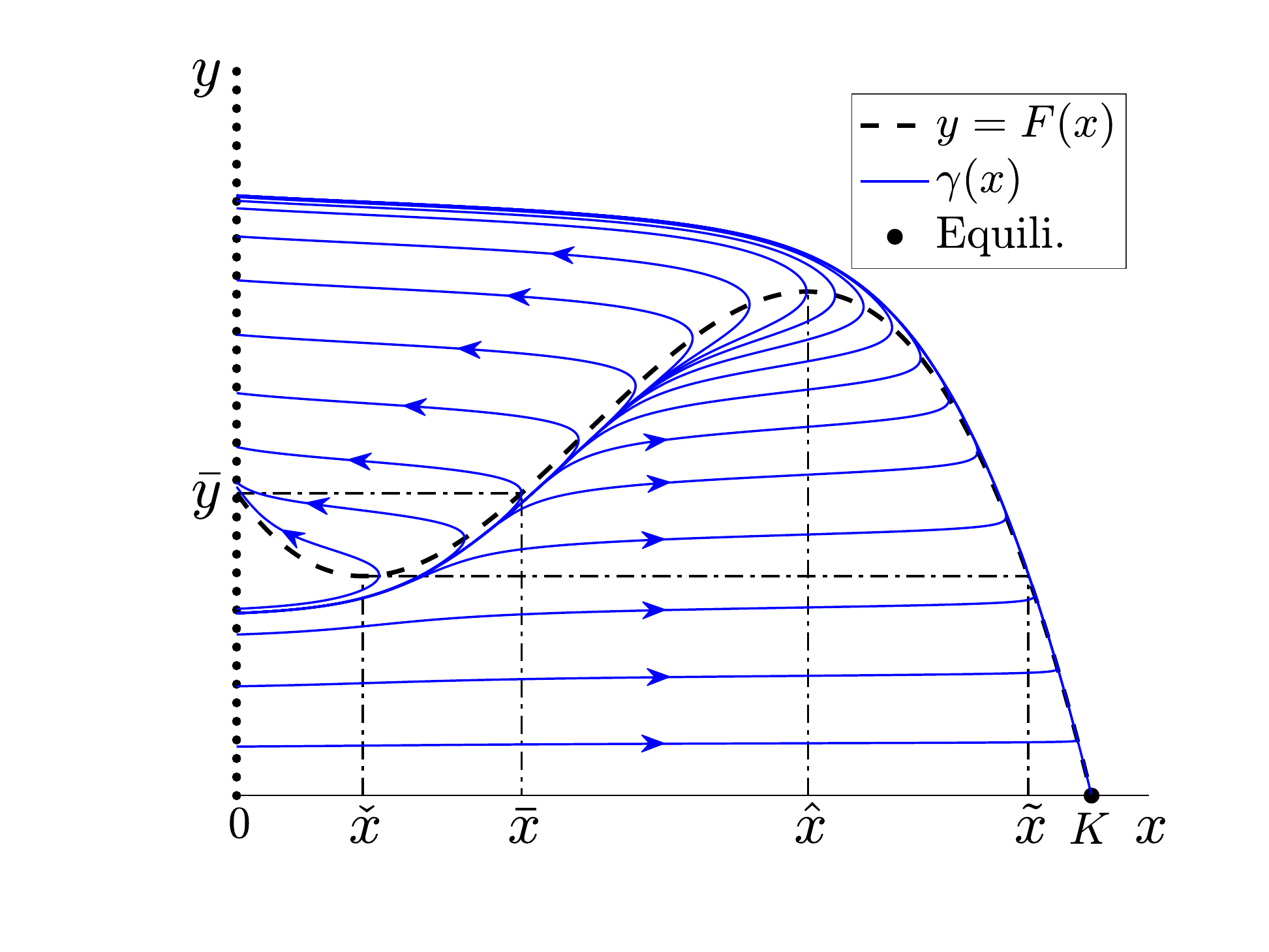}}
\\
(A)
&
(B)
\end{tabular}
\caption{
When $c>0$ is small,
the portions of trajectories of \eqref{fast_xy} 
away from the $x$-isocline $y=F(x)$
are nearly horizontal.
(A) $F(x)$ has a single local extremum.
$y_\alpha(x,c)$ approaches either $\bar{y}$ or $F(x)$,
and $y_\omega(x,c)$ approaches either $\widehat{y}$ or $F(x)$.
(B) $F(x)$ has two local extrema.
$y_\alpha(x,c)$ approaches either $\bar{y}$ or $F(x)$,
and $y_\omega(x,c)$ approaches either $\bar{y}$ or $F(x)$
as $c\to 0$.
}
\label{fig_c0}
\end{figure}

\begin{proposition}\label{prop_1hump}
Assume the one-hump condition \eqref{cond_1hump}.
Let $x_0(c)$
be the unique root of $\chi(\cdot,c)$ given in Theorem \ref{thm_1hump}.
Then \[
  \lim_{c\to 0^+}x_0(c)= \underline{x}
\] where $\underline{x}$ is the unique value in $(\widehat{x},K)$
that satisfies \beq{def_xlower_1hump}
  H(F(\widehat{x}))= H(F(\underline{x})).
\]
\end{proposition}

\begin{proof}
By Lemma \ref{lem_xmin},
$x_0(c)\ge \bar{x}$.
By \eqref{x0_limit_1hump}, \beq{chi_limit_1hump}
  \lim_{c\to 0^+}\chi(x,c)
  =H(F(\widehat{x}))-H(F(x))
  \quad\forall\;x\in (\bar{x},K).
\]
Note that (see Figure \ref{fig_H2_HLog}) \beq{monotone_H}
  H'(y)<0\quad\forall\; y\in (0,\bar{y})
  \quad\text{and}\quad
  H'(y)>0\quad\forall\; y\in (\bar{y},\infty).
\]
Since $F(\widehat{x})>\bar{y}$, $F(K)=0$,
and $F(x)$ is strictly decreasing on $(\widehat{x},K)$,
by \eqref{monotone_H}
and the fact that $\lim_{y\to 0^+}H(y)=\infty$,
there is a unique $\underline{x}\in (\widehat{x},K)$ that satisfies \eqref{def_xlower_1hump},
and \beq{diff_HF_1hump}
  H(F(\widehat{x}))-H(F(x))
  \begin{cases}
  >0& \text{if }x\in (\widehat{x},\underline{x}),
  \\
  <0& \text{if }x\in (\underline{x},K).
  \end{cases}
\]
Note that the convergence in \eqref{chi_limit_1hump}
is uniform on any compact subset of $(\widehat{x},K)$,
so \eqref{def_xlower_1hump} follows from \eqref{chi_limit_1hump} and \eqref{diff_HF_1hump}.
\end{proof}

\section{The Case with Two Humps}
\label{sec_2hump}
In this section we study the case that 
the function $F(x)$ has exactly two interior local extremum on the interval $(0,K)$.
That is, 
the prey-isocline
satisfies the two-hump condition (see Figure \ref{fig_labels}{(B)}):
For some $0<\widecheck{x}<\widehat{x}<K$,
\beq{cond_2hump}
  &F'(x)<0\quad\forall\;x\in (0,\widecheck{x})\cup (\widehat{x},K)
  \quad\text{and}\quad
  F'(x)>0\quad\forall\;x\in (\widecheck{x}, \widehat{x}).
\]

In the following theorem,
for technical reasons, we assume in addition that 
\beq{cond_convex}
  F''(x)<0\quad \forall\; x\in (\widehat{x},K).
\]
That is, the prey-isocline is concave
on the right of the interior local maximum.

\begin{theorem}\label{thm_2humps}
Assume \eqref{cond_2hump} and \eqref{cond_convex}.
Then $\chi(x)$ has at most three distinct roots
in the interval $(0,K)$.
More precisely, one of the following holds.
\begin{enumerate}[label=$\mathrm{(\roman*)}$]
\item
$\chi(x)$ has no root in $(0,K)$.
\item
$\chi(x)$ has exactly one root in $(0,K)$, say $x_0$,
and $\lambda(x_0)=0$.
\item
$\chi(x)$ has exactly two distinct zeros in $(0,K)$,
say $x_0<x_1$, \beq{lambda01_ineq}
  \lambda(x_0)\ge 0\ge \lambda(x_1)
  \quad\text{and}\quad
  \lambda(x_0)\ne \lambda(x_1).
\]
\item
$\chi(x)$ has exactly three distinct roots in $(0,K)$,
say $x_0< x_1< x_2$, and \beq{lambda012_ineq}
  \lambda(x_0)
  >\lambda(x_1)=0
  >\lambda(x_2).
\]
\end{enumerate}
\end{theorem}

An implication of Theorem \ref{thm_2humps}
is that the set \eqref{union_Gamma}
is the union of at most three loops of the form $\Gamma(x)$.

This theorem is a consequence of the following lemma.

\begin{figure}[t]
\centering
\begin{tabular}{cc}
\frame
{\includegraphics[trim = 2cm 1.4cm 1cm .3cm, clip, width=.46\textwidth]{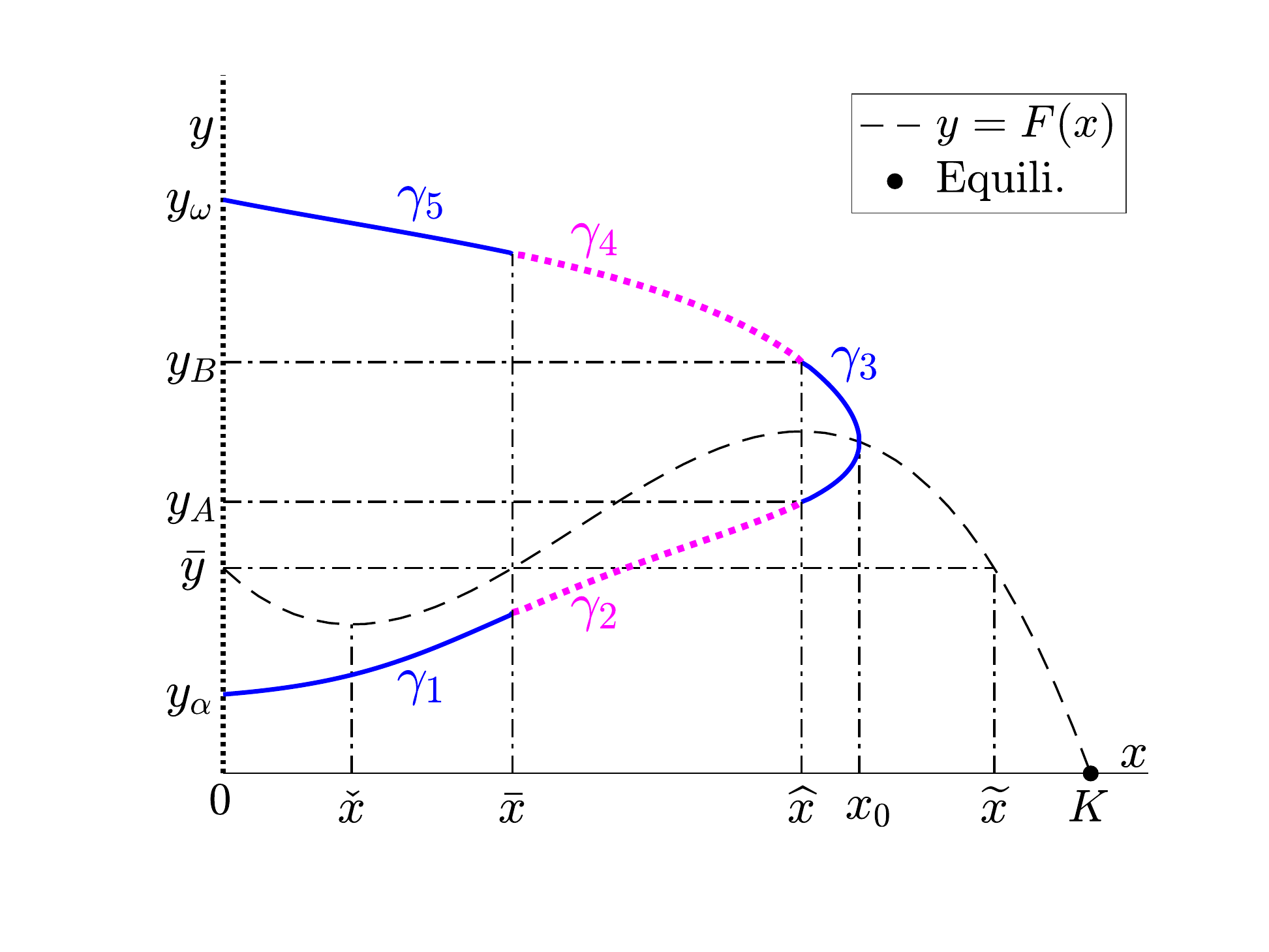}}
&
\frame
{\includegraphics[trim = 2cm 1.4cm 1cm .3cm, clip, width=.46\textwidth]{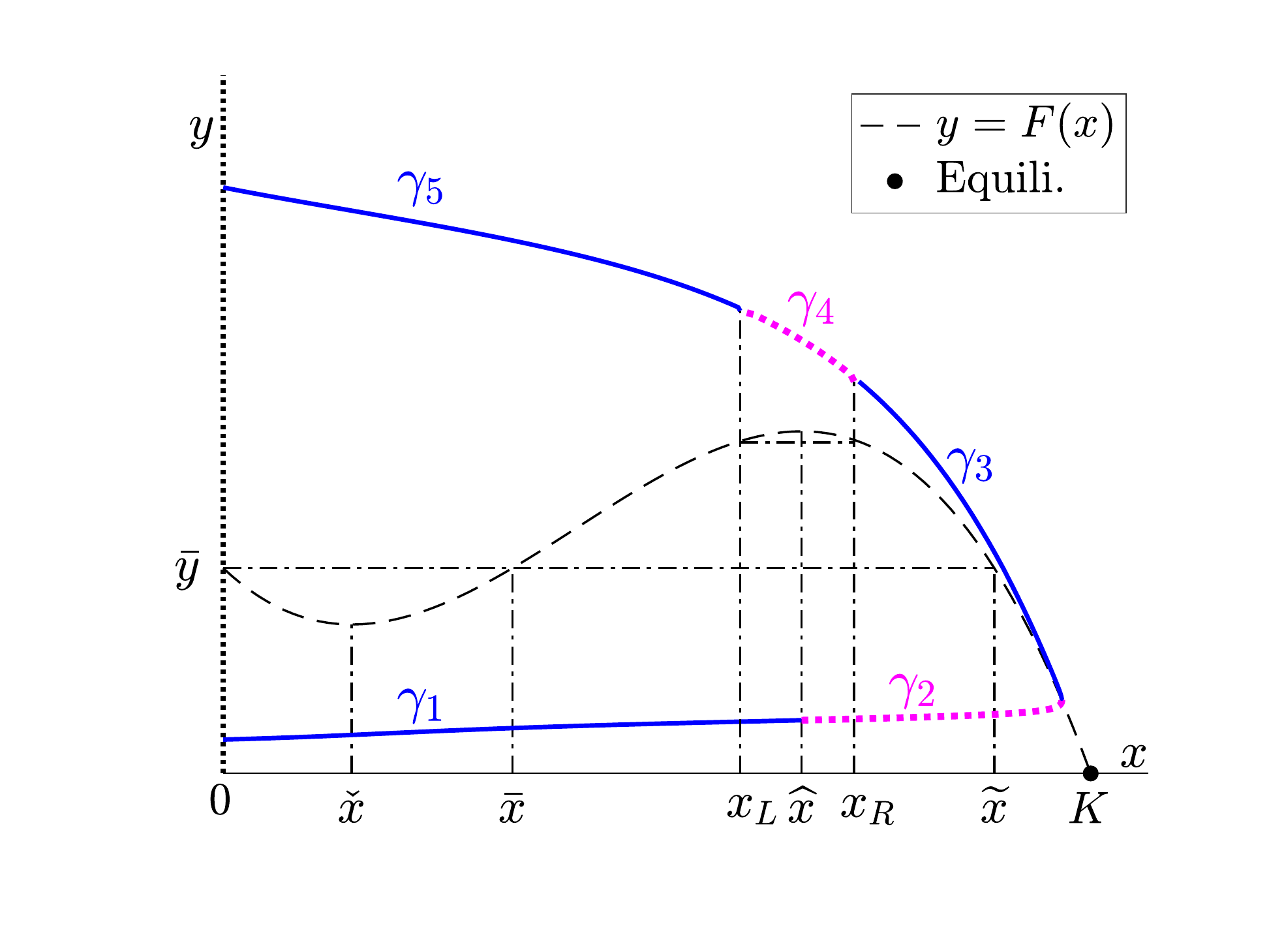}}
\\
(A)
&
(B)
\end{tabular}
\caption{
(A)
$y_A(x_0)$ and $y_B(x_0)$
in the proof of Lemma \ref*{lem_lambda12}
are defined by the intersection of $\Gamma(x_0)$ and 
the vertical line $\{x=\widehat{x}\}$.
(B)
$x_L$ and $x_R$
in the proof of Proposition \ref*{prop_2humps}
are defined by the intersection of the graph of $F(x)$
and a horizontal line
near the climax of the graph.
}
\label{fig_H4_labelsXY}
\end{figure}

\begin{lemma}\label{lem_lambda12}
Assume \eqref{cond_2hump} and \eqref{cond_convex}.
Then $\lambda(x)$ is strictly decreasing on $(\widehat{x},K)$.

In particular,
if $x_0<x_1$ satisfy $\chi(x_0)=\chi(x_1)=0$ and $\lambda(x_0)= 0$,
then $\lambda(x_1)<0$.
\end{lemma}

\begin{proof}[Proof of Lemma \ref*{lem_lambda12}]
Using the relation $dx/dy= (F(x)-y)/(cy)$ from \eqref{fast_xy}
and the definition of $\lambda(x_0)$ in \eqref{def_lambda}, we have \beq{lambda_int_xy}
  \lambda(x_0)
  =\int_{\gamma(x_0)}\frac{F'(x)}{y}\;dy
  =c\int_{\gamma(x_0)}\frac{F'(x)}{F(x)-y}\;dx.
\]
We split $\gamma=\gamma(x_0)$
into the following five parts (see Figure \ref{fig_H4_labelsXY}{(A)}): \[
  &\gamma_1= \gamma\cap \{x\le \bar{x},\;y<F(x)\},
  \;\;
  \gamma_2= \gamma\cap \{x\in [\bar{x},\widehat{x}],\; y<F(x)\},
  \;\;
  \gamma_3= \gamma\cap \{x\ge \widehat{x}\},
  \\[.5em]
  &\gamma_4= \gamma\cap \{x\in [\bar{x},\widehat{x}],\; y>F(x)\},
  \;\;
  \gamma_5= \gamma\cap \{x\le \bar{x},\;y>F(x)\}.
\]
Then $\lambda(x_0)=c\sum_{j=1}^5 I_j(x_0)$, where \beq{def_Ij_YAB}
  I_j(x_0)
  = \int_{\gamma_j(x_0)} \frac{F'(x)}{cy}\;dy
  = \int_{\gamma_j(x_0)} \frac{F'(x)}{F(x)-y}\;dx.
\]

We parameterize $\gamma(x_0)$,
for $x_0\in (\widehat{x},K)$,
using \eqref{def_Ypm}
with $y=Y_\pm(x,x_0)$, $x\in (0,x_0]$.
Also we parameterize $\gamma(x_0)$ by \[
  \gamma(x_0)=\{(X(y,x_0),y): y\in (y_\alpha(x_0),y_\omega(x_0))\}.
\]
Fix any $x_0$ and $x_1$ with $\widehat{x}<x_0<x_1<K$.
We claim that $I_j(x_1)<I_j(x_0)$, $j=1,\dots,5$.

Note that \[
  Y_-(x,x_1)< Y_-(x,x_0)< F(x)< Y_+(x,x_0)< Y_+(x,x_1).
\]
Since $F'(x)>0$ on $(\bar{x},\widehat{x})$,
it follows that \[
  I_2(x_1)
  = \int_{\bar{x}}^{\widehat{x}}
  \frac{F'(x)}{F(x)-Y_-(x,x_1)}
  \;dx
  < \int_{\bar{x}}^{\widehat{x}}
  \frac{F'(x)}{F(x)-Y_-(x,x_0)}
  \;dx
  = I_2(x_0).
\] Hence $I_2(x_1)< I_2(x_0)$.
Similarly, $I_4(x_1)< I_4(x_0)$.

By \eqref{def_Ij_YAB}, \beq{diff_I1}
  I_1(x_0)- I_1(x_1)
  &=\int_0^{\bar{x}}
    \frac{F'(x)}{F(x)-Y_-(x,x_0)}
    -\frac{F'(x)}{F(x)-Y_-(x,x_1)}
  \;dx
  \\[1em]
  &=\int_0^{\bar{x}}
    \frac{F'(x)(Y_-(x,x_0)-Y_-(x,x_1))}{(F(x)-Y_-(x,x_0))(F(x)-Y_-(x,x_1))}
  \;dx.
\]
The function $\rho(x)\equiv Y_-(x,x_0)- Y_-(x,x_1)$
is increasing on $(0,\bar{x})$ because $\rho(0)>0$ and \[
  \frac{d\rho}{dx}
  &= \frac{Y_-(x,x_0)}{F(x)-Y_-(x,x_0)}
  - \frac{Y_-(x,x_1)}{F(x)-Y_-(x,x_1)}
  \\[1em]
  &= \frac{F(x)\rho}{(F(x)-Y_-(x,x_1))(F(x)-Y_-(x,x_0))}
  >0.
\] Since $\rho(x)$, $Y_-(x,x_0)$ and $Y_-(x,x_1)$ are increasing functions,
and $F'(x)$ changes sign at $x=\widehat{x}$,
\eqref{diff_I1} yields \[
  &I_1(x_0)- I_1(x_1)
  >\int_0^{\bar{x}}
    \frac{F'(x)\rho(\widehat{x})}
    {(F(x)-Y_-(\widehat{x},x_0))(F(x)-Y_-(\widehat{x},x_1))}
  \;dx
  \\[1em]
  &\quad =\rho(\widehat{x})\,
  \log\big(
    (F(x)-Y_-(\widehat{x},x_1))(F(x)-Y_-(\widehat{x},x_0))
  \big)\Big|_{x=0}^{\bar{x}}
  =0.
\] The last equality followed from the fact that $F(\bar{x})=F(0)$.
Hence $I_1(x_0)>I_1(x_1)$.
Similarly, $I_5(x_0)>I_5(x_1)$.

For $\xi\in (\widehat{x},K)$,
define $y_A(\xi)<y_B(\xi)$ to be the values so that both
$(\widehat{x},y_A(\xi))$ and $(\widehat{x},y_B(\xi))$
lie on $\gamma(\xi)$, that is, \[
  y_A(\xi)= Y_-(\widehat{x},\xi),\quad
  y_B(\xi)= Y_+(\widehat{x},\xi).
\]
Differentiating the expression $
  I_3(\xi)= \int_{y_A(\xi)}^{y_B(\xi)} \frac{F'(X(y,\xi))}{c\,y}\;dy
$ with respect to $\xi$, we obtain  \[
  I_3'(\xi)
  =\frac{y_B'(\xi)F'(\widehat{x})}{c\,y_B(\xi)}
  -\frac{y_A'(\xi)F'(\widehat{x})}{c\,y_A(\xi)}
  + \int_{y_A(\xi)}^{y_B(\xi)}
  \frac{F''(X(y,\xi))}{c\,y}\frac{\partial X(y,\xi)}{\partial \xi}\;dy.
\] Since $F'(\widehat{x})=0$,
$F''(x)<0$ and $\partial X(y,\xi)/\partial \xi>0$,
we obtain $I_3'(x_0)<0$ for all $\xi\in (\widehat{x},K)$.
Hence $I_3(x_0)>I_3(x_1)$ for $\widehat{x}>x_0>x_1$.
We conclude that $\lambda(x_0)>\lambda(x_1)$.

Finally, we assume
$x_0<x_1$ satisfy $\chi(x_0)=\chi(x_1)=0$ and $\lambda(x_0)= 0$.
Since $\lambda(x_0)\le 0$,
we have $x_0>\widehat{x}$ by Lemma \ref{lem_xmin}.
From the monotonicity of $\lambda(x)$ on $(\widehat{x},K)$,
it follows that $\lambda(x_1)<\lambda(x_0)=0$.
\end{proof}

\begin{proof}[Proof of Theorem \ref*{thm_2humps}]
First we claim that
the set of roots of $\chi$ in the interval $(0,K)$ is a discrete set.
Suppose $x_0$ is an accumulation point of the set of roots of $\chi$.
If $\lambda(x_0)>0$,
then there exists $x_1\ne x_0$, $\chi(x_1)=0$ such that \[
  \lambda(x)>0
  \quad\text{$\forall\; x$ between $x_0$ and $x_1$}.
\]
Then, by Theorem \ref{thm_y0},
for any sufficiently small $\epsilon>0$
there are two unstable periodic orbits,
with one of them near $\Gamma(x_0)$ and the other near $\Gamma(x_1)$,
and no orbitally stable periodic orbit lies between them.
This contradicts the Poincar\'{e}-Bendixon Theorem.
The case $\lambda(x_0)<0$ can be treated similarly.
If $\lambda(x_0)= 0$,
then $x_0> \widehat{x}$ by Lemma \ref{lem_xmin}.
Since $x_0$ is an accumulation point,
from one of the two intervals, $(\widehat{x},x_0)$ and $(x_0,K)$,
we can choose two roots $x_1<x_2$ of $\chi$.
In either case,
by the monotonicity of $\lambda(x)$ on $(\widehat{x},K)$ given in Lemma \ref{lem_lambda12},
$\lambda(x)$ is nonzero
and does not change sign on the interval $(x_1,x_2)$.
By Theorem \ref{thm_y0}, this contradicts the Poincar\'{e}-Bendixon Theorem.

Suppose $\chi$ has four consecutive distinct roots, say $x_0<x_1<x_2<x_3$.
If $\lambda(x_0)>0$, then $\lambda(x_1)\le 0$
because of the stability of periodic orbits.
By Lemma \ref{lem_lambda12} it follows that \[
  \lambda(x_0)
  >0
  \ge \lambda(x_1)
  > \lambda(x_2)
  > \lambda(x_3),
  \quad
  \lambda(x_i)\lambda(x_{i+1})\le 0
  \quad \forall\;i,
\] which is impossible.
Similarly, if $\lambda(x_0)\le 0$,
by Lemma \ref{lem_lambda12} we have \[
  0
  \ge \lambda(x_0)
  > \lambda(x_1)
  > \lambda(x_2)
  > \lambda(x_3),
  \quad
  \lambda(x_i)\lambda(x_{i+1})\le 0
  \quad \forall\;i,
\] which is also impossible.
Hence there are at most three distinct roots.

Since $\chi$ has at most three distinct roots in $(0,K)$,
and the values of $\lambda$ cannot have the same sign
at two consecutive roots of $\chi$,
by Lemma \ref{lem_lambda12} and the fact that $E_*$ is locally asymptotically stable,
only the situations listed in the theorem are possible.
\end{proof}

Under the two-hump condition \eqref{cond_2hump},
it is easy to show, as illustrated in Figure \ref{fig_c0}{(B)}, that
\beq{x0_limit_2humps}
  y_\alpha(x,c)
  \to
  \begin{cases}
    F(\widecheck{x}),& \text{if }x\in [\widecheck{x},\widetilde{x}]\\
    F(x),& \text{otherwise}
  \end{cases}
  \quad\text{and}\quad
  y_\omega(x,c)
  \to\begin{cases}
    \bar{y},& \text{if }x\in (0,\bar{x}]\\
    F(x),& \text{if }x\in [\bar{x},\widehat{x}]\\
    F(\widehat{x}),& \text{if }x\in [\widehat{x},K)
  \end{cases}
\]
as $c\to 0$,
where $\widetilde{x}\in (\widehat{x},K)$
satisfies $F(\widetilde{x})=F(\widecheck{x})$,
and the convergence is $C^1$-uniform
on any compact subset of $(0,K)\setminus \{\widecheck{x},\bar{x},\widehat{x},\widetilde{x}\}$.

The following proposition describes
the roots of $\chi(\cdot,c)$ when $c$ is sufficiently small.

\begin{figure}[t]
\centering
\begin{tabular}{cc}
\frame
{\includegraphics[trim = 1.2cm .3cm 1cm 0cm, clip, width=.46\textwidth]{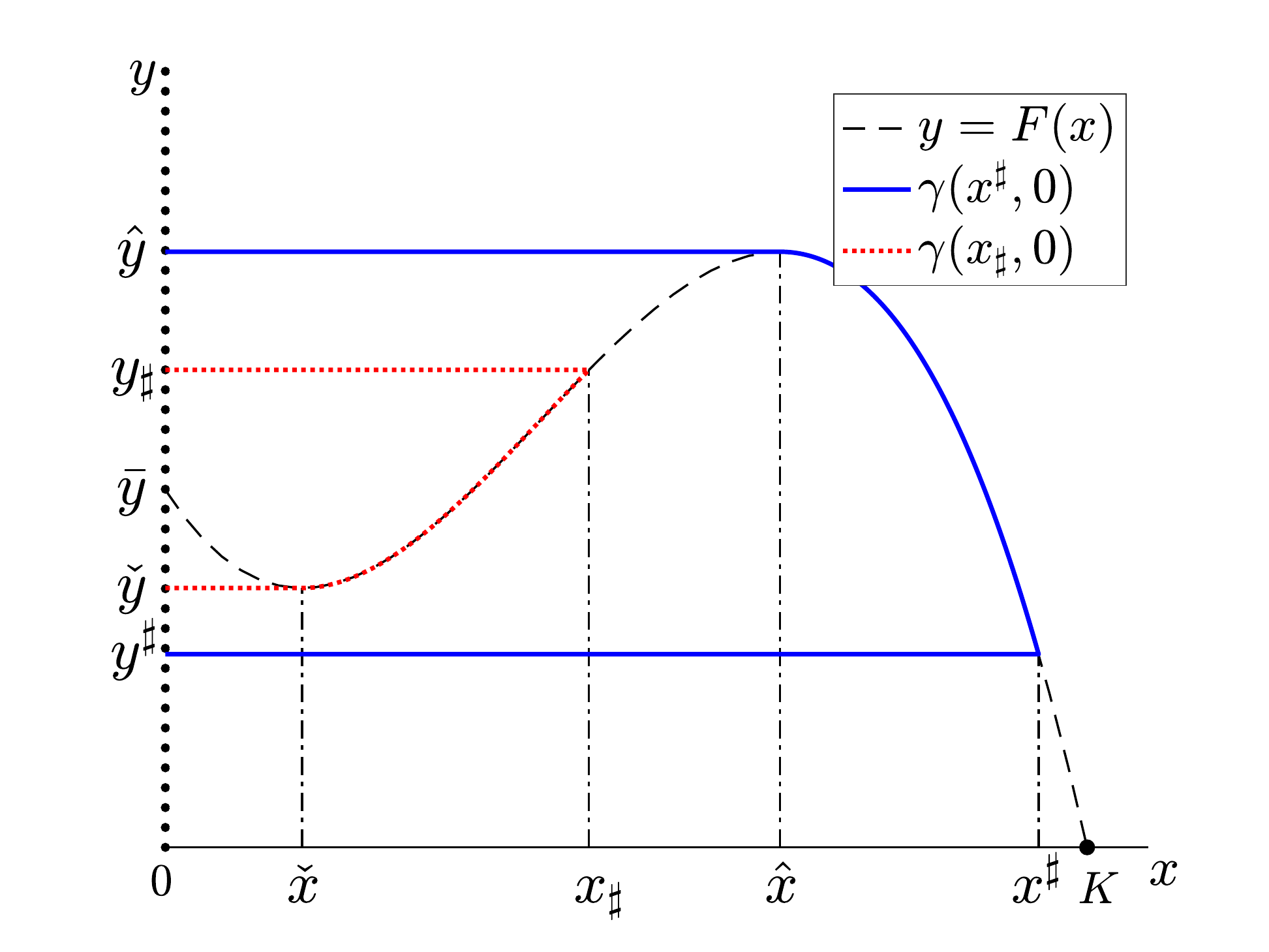}}
&
\frame
{\includegraphics[trim = 1.2cm .3cm 1cm 0cm, clip, width=.46\textwidth]{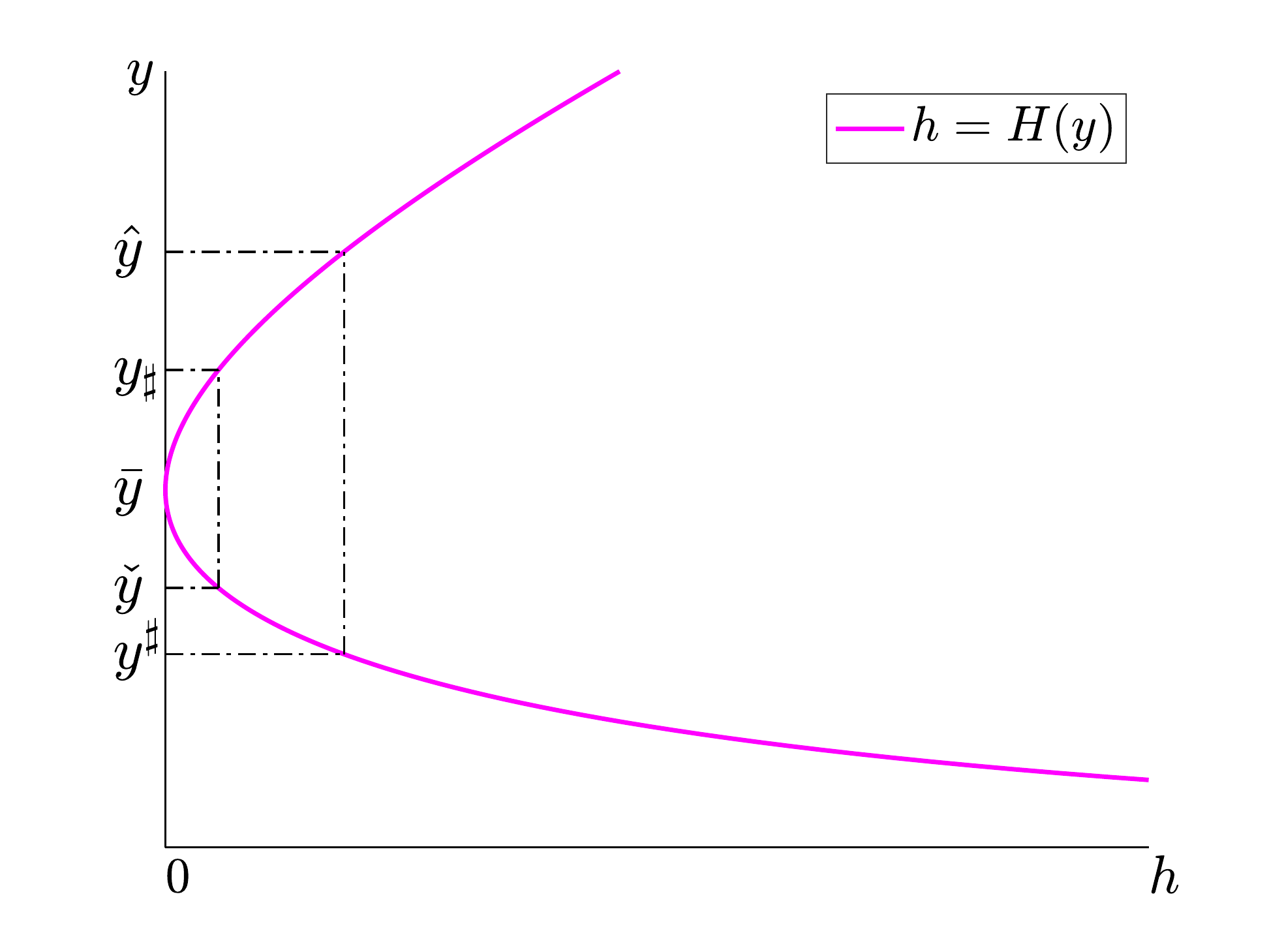}}
\end{tabular}
\caption{
$\gamma(x,0)$ is the limit of $\gamma(x,c)$ as $c\to 0^+$.
The values $x_\sharp$ and $x^\sharp$
in Proposition \ref{prop_2humps}
are determined by $F(x_\sharp)=y_\sharp$
and $F(x^\sharp)=y^\sharp$,
with $H(\widecheck{y})= H(y_\sharp)$
and $H(\widehat{y})= H(y^\sharp)$.
}
\label{fig_H4_HLog}
\end{figure}

\begin{proposition}\label{prop_2humps}
Assume the two-hump condition \eqref{cond_2hump}.
\begin{enumerate}[label=$\mathrm{(\roman*)}$]
\item
If $H(F(\widehat{x}))< H(F(\widecheck{x}))$,
then, for any sufficiently small $c>0$,
$\chi(\cdot,c)$ has no root in the interval $(0,K)$.
\item
If $H(F(\widehat{x}))> H(F(\widecheck{x}))$,
then,
for any sufficiently small $c>0$,
$\chi(\cdot,c)$ has exactly two distinct roots in the interval $(0,K)$,
say $x_0(c)< x_1(c)$.
Moreover, $\lambda(x_0(c))>0$, $\lambda(x_1(c))<0$, and \beq{lim_x01c}
  \lim_{c\to 0^+}x_{0}(c)= x_\sharp,\qquad
  \lim_{c\to 0^+}x_{1}(c)= x^\sharp,
\] where $x_\sharp\in (\widecheck{x},\widehat{x})$
and $x^\sharp\in (\widehat{x},K)$ satisfy \beq{def_xsharp}
  H(F(x_\sharp))= H(F(\widecheck{x})),
  \quad
  H(F(x^\sharp))=H(F(\widehat{x})).
\]
\end{enumerate}
\end{proposition}

Note that Proposition \ref{prop_2humps}
does not assume condition \eqref{cond_convex}.

\begin{proof}
By Lemma \ref{lem_xmin}, no root of $\chi$ lies in the interval $(0,\bar{x})$,
so we only need to estimate $\chi(x,c)$ for $x\in (\bar{x},K)$.
By \eqref{x0_limit_2humps}, \beq{limit_chi_2hump}
  \lim_{c\to 0^+}\chi(x,c)
  =\begin{cases}
    H(F(x))-H(F(\widecheck{x}))& \text{if }x\in (\bar{x},\widehat{x}),\\
    H(F(\widehat{x}))-H(F(\widecheck{x}))& \text{if }x\in (\widehat{x},\widetilde{x}),\\
    F(\widehat{x})-F(x)& \text{if }x\in (\widetilde{x},K),
  \end{cases}
\]
and the convergence is $C^1$-uniform
on any compact subset of $(\bar{x},K)\setminus \{\widehat{x},\widetilde{x}\}$.

By \eqref{cond_2hump} and \eqref{monotone_H},
we have $H(F(x))< H(F(\widehat{x}))$ $\forall\; x\in (\bar{x},\widehat{x})$
and $H(F(x))> H(F(\widecheck{x}))$ $\forall\; x\in (\widecheck{x},K)$.
If $H(F(\widehat{x}))< H(F(\widecheck{x}))$,
then the right-hand side of \eqref{limit_chi_2hump}
is bounded away from $0$ on $(\bar{x},K)$.
It is easy to show that $\chi(x,c)$ is uniformly bounded away from $0$, as $c\to 0$,
for $x\in (\bar{x},K)$.
Hence $\chi(\cdot,c)$ has no root in $(0,K)$.

If $H(F(\widehat{x}))> H(F(\widecheck{x}))$,
by \eqref{cond_2hump} and \eqref{monotone_H},
there exist a unique $x_\sharp\in (\widecheck{x},\widehat{x})$
and a unique $x^\sharp\in (\widehat{x},K)$
that satisfy \eqref{def_xsharp}.
Choose $\delta>0$ such that 
$x_\sharp\in (\hat{x}+\delta,\widehat{x}-\delta)$
and $x^\sharp\in (\widehat{x}+\delta,K-\delta)$.
Then the right-hand side of \eqref{limit_chi_2hump}
is bounded away from $0$ on the set \[
  {}[\bar{x},\bar{x}+\delta]\cup [\widehat{x}-\delta,\widehat{x}+\delta]\cup [K-\delta,K).
\]
Hence $\chi(x,c)$ is uniformly bounded away from $0$, as $c\to 0^+$,
for $x$ on this set.

On the other hand, 
we have the uniform convergence \[
  \lim_{c\to 0^+}\frac{\partial\chi(x,c)}{\partial x}
  =\begin{cases}
    H'(F(x))F'(x)& \text{if }x\in [\bar{x}+\delta,\widehat{x}-\delta],\\
    -H(F(x))F'(x)& \text{if }x\in [\widehat{x}+\delta,K-\delta].
  \end{cases}
\]
Therefore, by \eqref{cond_2hump} and \eqref{monotone_H},
$\partial\chi(x,c)/\partial x$
is uniformly bounded away from $0$, as $c\to 0^+$,
for $x\in [\bar{x}+\delta,\widehat{x}-\delta]\cup [\widehat{x}+\delta,K-\delta]$.
From the implicit function theorem
it follows that $\chi(\cdot,c)$ has exactly roots,
$x_0(c)\in [\bar{x}+\delta,\widehat{x}-\delta]$
and $x_1(c)\in [\widehat{x}+\delta,K-\delta]$,
and \eqref{lim_x01c} holds.

Since $x_0(c)\in (\bar{x},\widehat{x})$,
by Lemma \ref{lem_xmin} we have $\lambda(x_0(c))>0$.
It remains to show $\lambda(x_1(c))<0$.
Choose any $x_L\in (\widehat{x}-\delta,\widehat{x})$
and $x_R\in (\widehat{x},\widehat{x}+\delta)$
so that $F(x_L)=F(x_R)>\bar{y}$.
We split $\gamma=\gamma(x_1(c))$
into the following five parts (see Figure \ref{fig_H4_labelsXY}{(B)}): \[
  &\gamma_1= \gamma\cap \{x\le \widehat{x},\;y<F(x)\},
  \;\;
  \gamma_2= \gamma\cap \{x\ge \widehat{x},\; y<F(x)\},
  \\[.5em]
  &\gamma_3= \gamma\cap \{x\ge x_R,\; y>F(x)\},
  \;\;
  \gamma_4= \gamma\cap \{x_L\le x\le x_R,\; y>F(x)\},
  \\[.5em]
  &\gamma_5= \gamma\cap \{x\le x_L,\; y>F(x)\}.
\]
By \eqref{lambda_int_xy},
$\lambda(x_1(c))=c\sum_{j=1}^5 I_j(c)$, where \[
  I_j(c)
  =\int_{\gamma_j(x_1(c))}\frac{F'(x)}{c\,y}\;dy
  =\int_{\gamma_j(x_1(c))}\frac{F'(x)}{F(x)-y}\;dx.
\]

Since $y=F(x^\sharp)+o(1)$ on $\gamma_1$, we have $I_1=O(1)$ as $c\to 0^+$.
Similarly, $I_5=O(1)$.

Since $F'(x)<0$ and $F(x)-y>0$ on $\gamma_2$,
we have $I_2<0$.

We parameterize $\gamma_3\cup \gamma_4$ by $y=Y_+(x,c)$, $x\in [x_L,x_1(c)]$.
Then $Y_+(x,c)$ is a deceasing function of $x$
that approaches $F(x)$ as $c\to 0^+$.
Hecne \[
  &I_4= \int_{x_L}^{x_R}
  \frac{F'(x)}{Y_+(x,c)-F(x)}\;dx
  \\[.5em]
  &\quad< \int_{x_L}^{x_R}
  \frac{F'(x)}{Y_+(\widehat{x},c)-F(x)}\;dx
  = \log\left(
    \frac{Y_+(\widehat{x},c)-F(x_R)}
    {Y_+(\widehat{x},c)-F(x_L)}
  \right)= 0.
\]
Since  $F'(x)<0$, $F(x)-Y_+(x,c)<0$, and $F(x)-Y_+(x,c)\to 0$ on $\gamma_3$, \[
  I_3
  &=\int_{x_R}^{x_1(c)}
    \frac{F'(x)}{Y_+(x,c)-F(x)}
  \;dx
  < \int_{x_R}^{x_R+\delta} 
    \frac{F'(x)}{Y_+(x,c)-F(x)}
  \;dx
  \\[.5em]
  &\quad
  <\frac{F'(x_R)\delta}{Y_+(x_R,c)-F(x_R)}
  \to -\infty.
\]
Therefore,
$\lim_{c\to 0^+}\frac{\lambda(x_1(c))}{c}=\lim_{c\to 0^+}\sum_{j=1}^5 I_j(c)= -\infty$.
It follows that $\lambda(x_1(c))<0$ for all sufficiently small $c>0$.
\end{proof}

The following example is an application of Theorem \ref{thm_2humps} and Proposition \ref{prop_2humps}.
Note that statement (i) in this example is covered by Ruan and Xiao \cite{Ruan:2001}.

\begin{ex}\label{ex_H4}
Consider \eqref{deq_xy} with
the Holling type IV functional response \beq{pq_H4}
  p(x)=\frac{mx}{ax^2+1},
\]
with $m,a>0$.
Then there exists $\kappa_*> 4$ such that the following statements hold.
\begin{enumerate}[label=$\mathrm{(\roman*)}$]
\item
If $aK^2< 4$,
then, for any fixed $c>0$,
the positive equilibrium of \eqref{deq_xy}
is globally asymptotically stable
for all sufficiently small $\epsilon>0$.
\item
If $4\le aK^2<\kappa_*$,
then for any small enough $c>0$,
the positive equilibrium of \eqref{deq_xy}
is globally asymptotically stable
for all sufficiently small $\epsilon>0$.
\item
If $aK^2>\kappa_*$,
then for any small enough $c>0$,
\eqref{deq_xy} has exactly two periodic orbits
for all sufficiently small $\epsilon>0$.
These periodic orbits
form two relaxation oscillations.
\end{enumerate}
\end{ex}

See Figures \ref{fig_H4_positive} and \ref{fig_H4_negative}
for illustrations of Example \ref{ex_H4}.

\begin{remark}
If we fix $a>0$
and regard $K$ as a bifurcation parameter
in Example \ref{ex_H4},
then $K_*=\sqrt{\kappa_*/a}$
is a threshold value of $K$
to determine the dynamics of \eqref{deq_xy}.
\end{remark}

\begin{remark}
Example \ref{ex_H4}
provides no comparison of
the smallness of $c$ and $\epsilon$.
Nonetheless,
since canard cycles have been discovered by Li and Zhu \cite{Zhu:2013}
in some cases where $c=\frac{\epsilon}{p(\widehat{x})}+o(\epsilon)$ as $\epsilon\to 0$,
it is conceivable that
the region
in $(c,\epsilon)$-space
for statement (iii) to be valid
satisfies $\epsilon=O(c)$ as $c\to 0$.
\end{remark}

\begin{figure}[t]
\centering
\begin{tabular}{p{.63\textwidth}p{.3\textwidth}}
\vspace{0em}
\begin{tabular}{c}
\frame
{\includegraphics[trim = 1cm .5cm 1cm 0cm, clip, width=.59\textwidth]{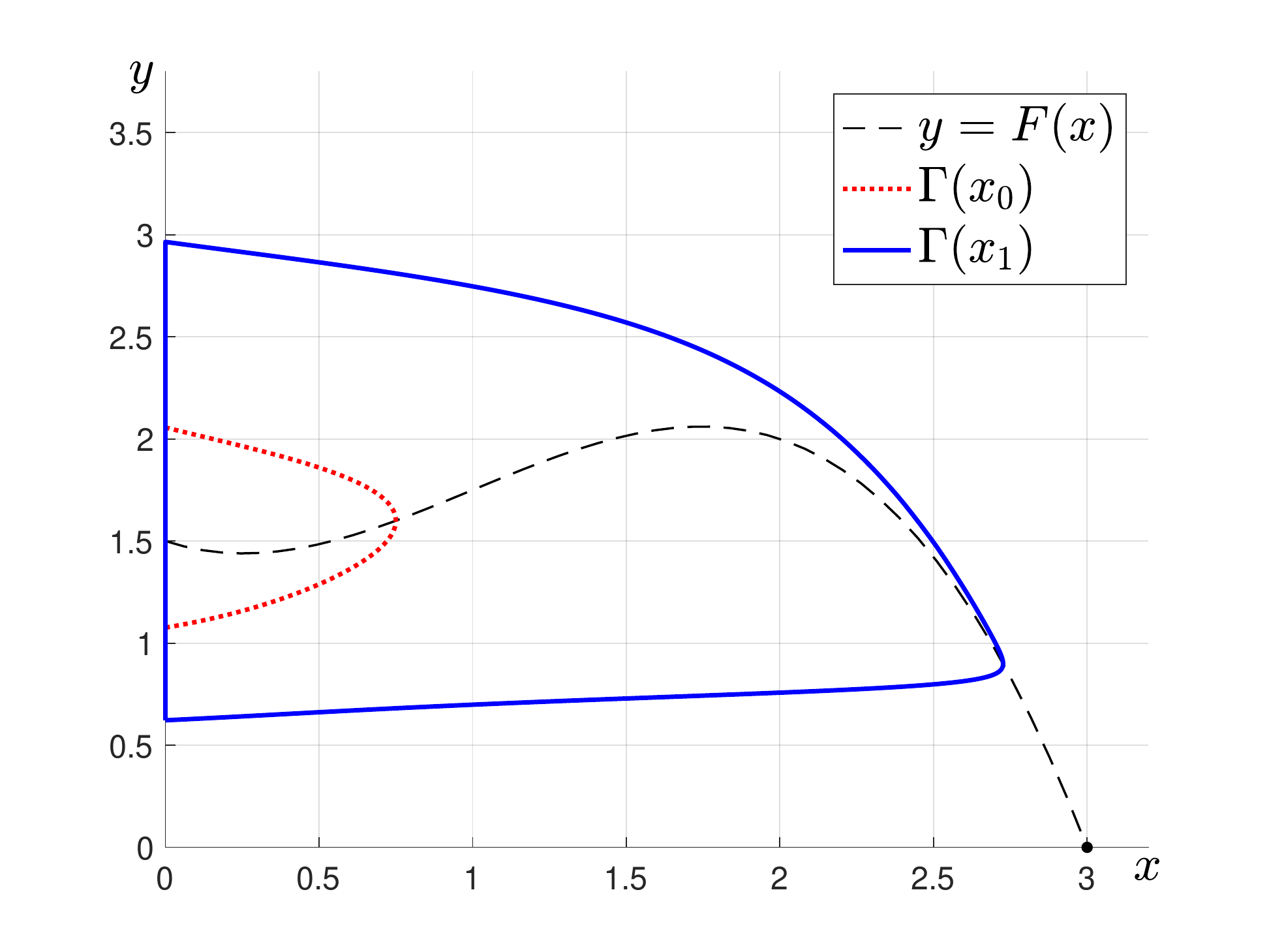}}
\end{tabular}
&
\vspace{0em}
\hspace{-2em}
\begin{tabular}{c}
\frame
{\includegraphics[trim = 1cm .3cm 1cm 0cm, clip, width=.28\textwidth]{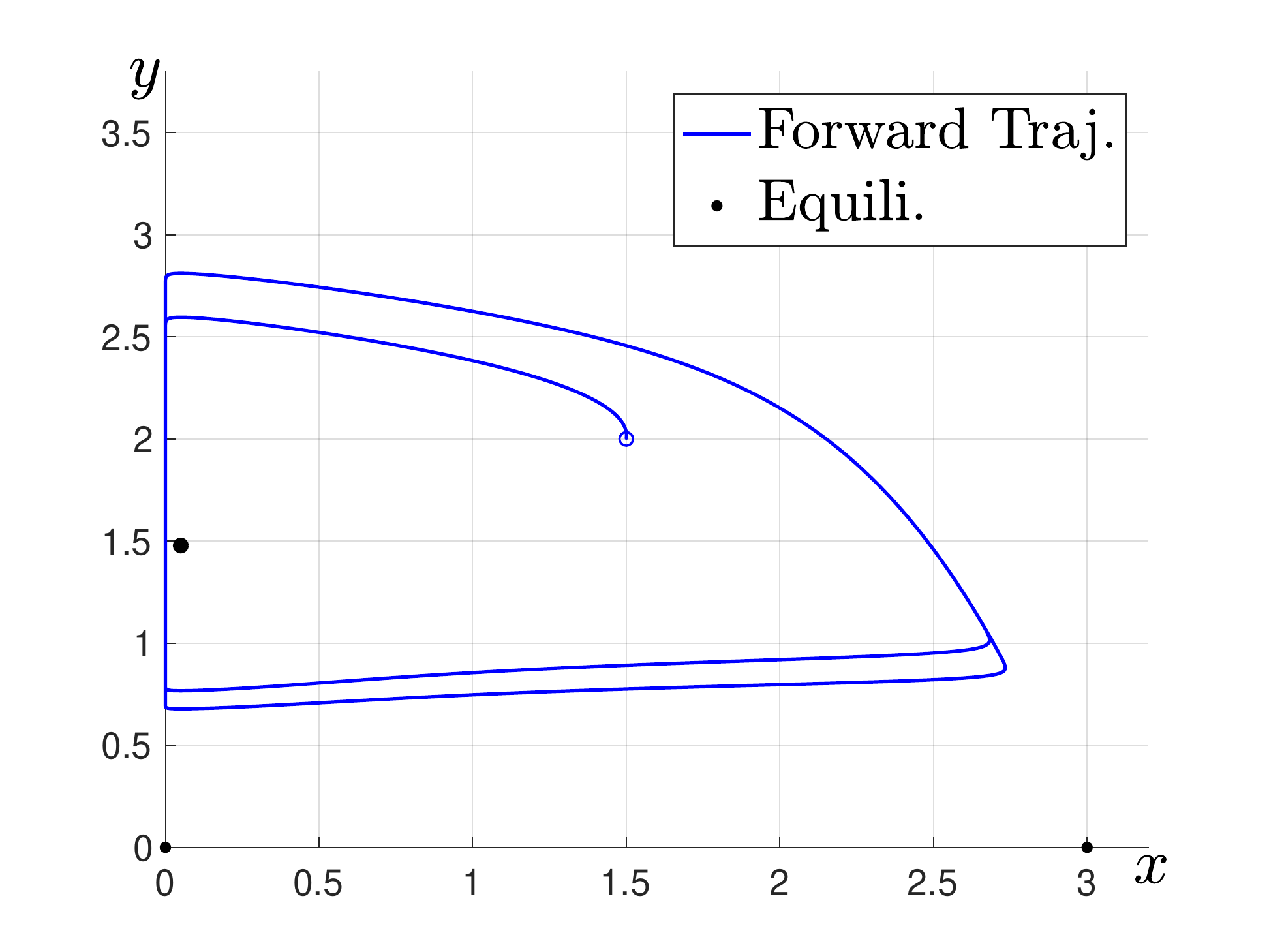}}
\\[.26em]
\frame
{\includegraphics[trim = 1cm .3cm 1cm 0cm, clip, width=.28\textwidth]{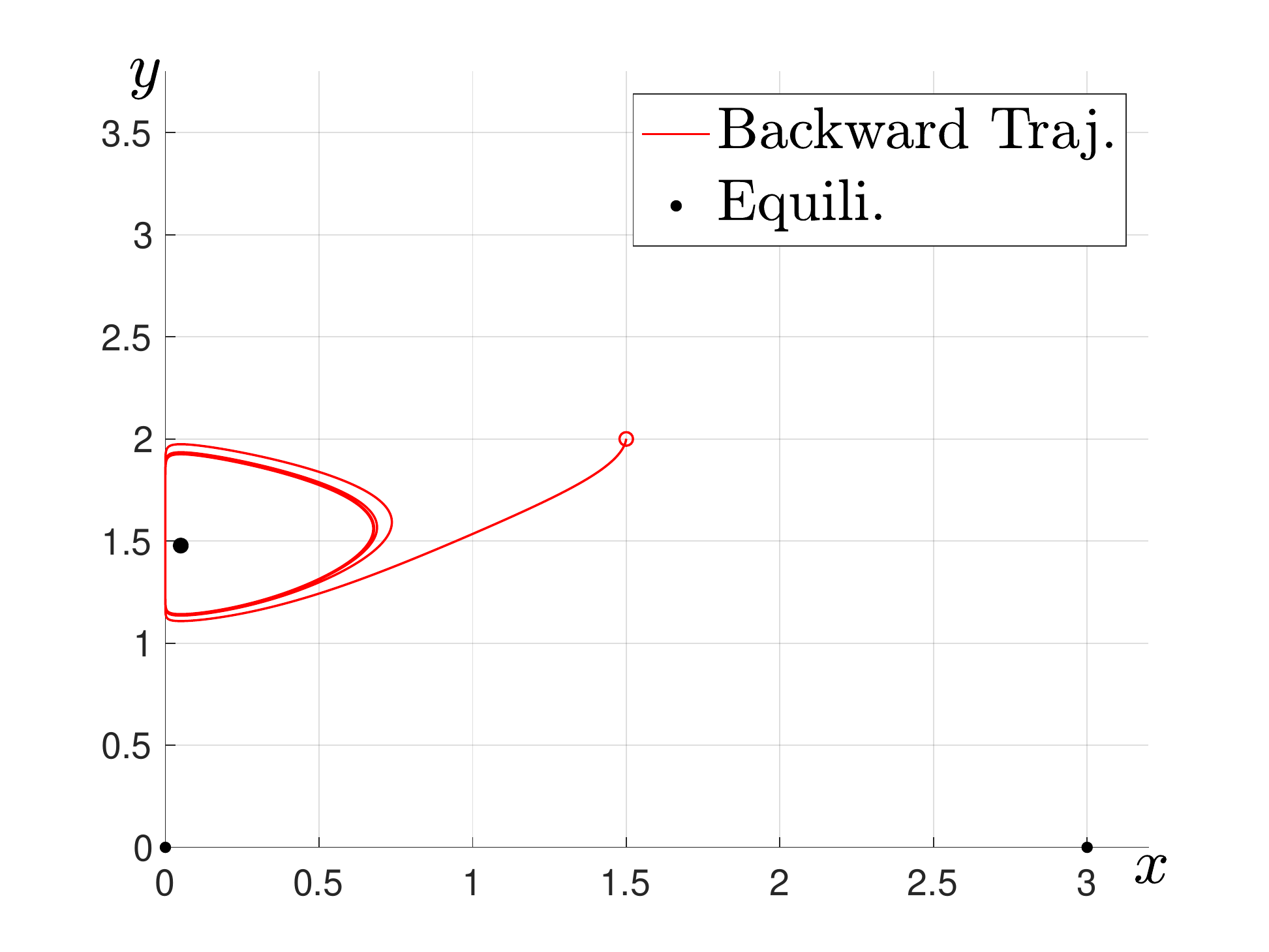}}
\end{tabular}
\end{tabular}
\caption{
$\chi(x)$ has two roots, $x_0<x_1$, in $(0,K]$.
Here $p(x)=mx/(ax^2+1)$,
and the parameters are $(r,K,m,a)=(4,3,2,5,0.75)$,
and $c=0.1$.
The condition
$H(F(\widehat{x}))> H(F(\widecheck{x}))$
in Proposition \ref{prop_2humps} holds.
The configurations $\Gamma(x_0)$
and $\Gamma(x_1)$
are given by Proposition \ref{prop_2humps}.
With $\epsilon=0.1$,
a forward trajectory
approaches a orbitally asymptotically stable periodic orbit near $\Gamma(x_1)$,
and a backward trajectory
approaches an unstable periodic orbit near $\Gamma(x_0)$.
}
\label{fig_H4_positive}
\end{figure}

\begin{figure}[t]
\centering
\frame
{\includegraphics[trim = 1.2cm .5cm .5cm 0cm, clip, width=.46\textwidth]{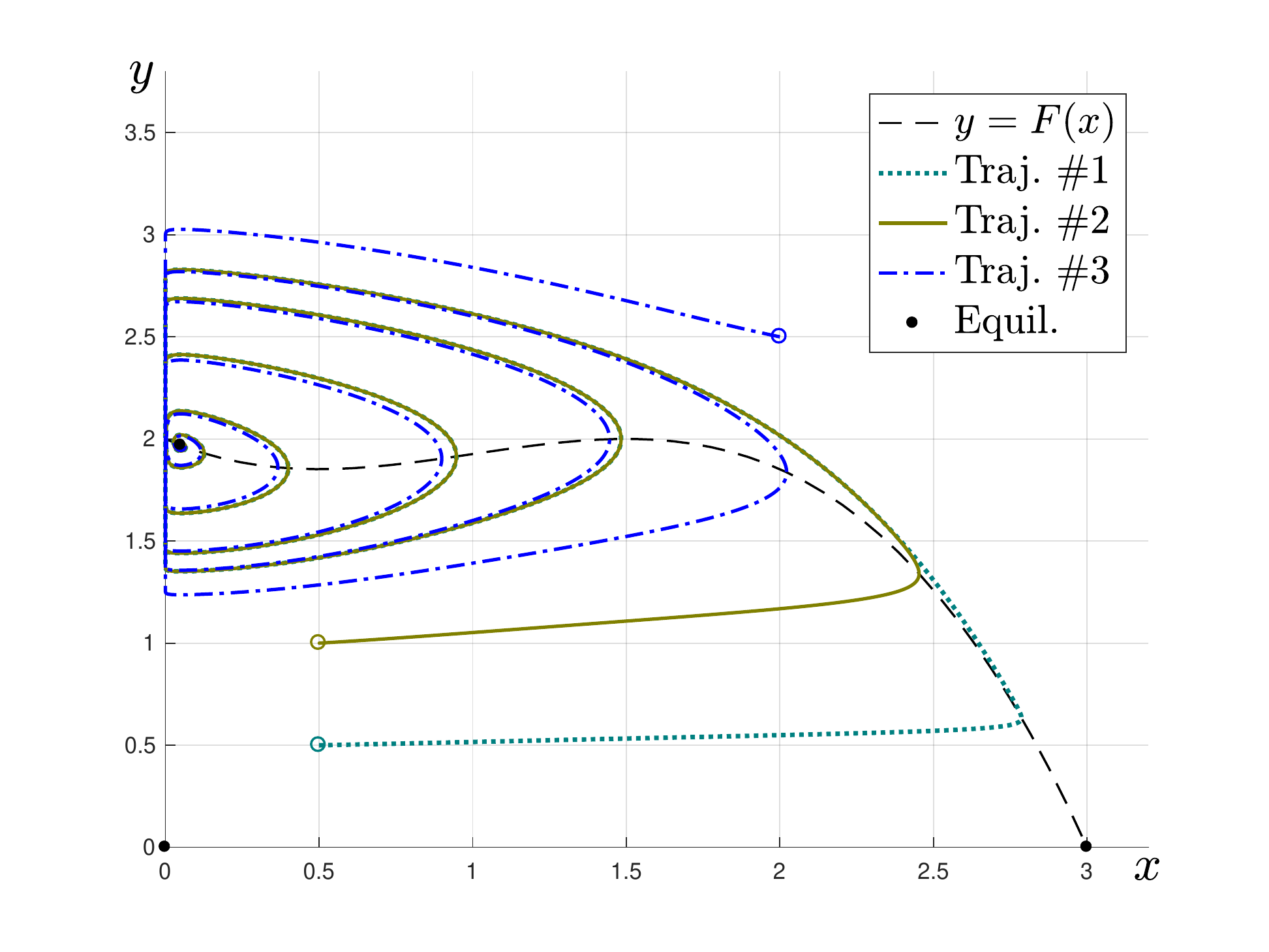}}
\caption{
Forward trajectories for \eqref{deq_xy}
with $p(x)=mx/(ax^2+bx+1)$
and parameters $(r,K,m,a)=(4,3,2,5,4/9)$, $c=0.1$, $\epsilon=0.01$.
The condition
$H(F(\widehat{x}))< H(F(\widecheck{x}))$
in Proposition \ref{prop_2humps} holds.
All trajectories converge to the positive equilibrium.
}
\label{fig_H4_negative}
\end{figure}

\begin{proof}
By \cite[Theorem 2.9]{Ruan:2001},
the positive equilibrium exists
and is globally asymptotically stable
if \beq{cond_H4gas}
  c\, m>\frac{4\epsilon}{\sqrt{3a}}
  \quad\text{and}\quad
  x_*<K<-x_*+2\sqrt{\frac{c\, m\, x_*}{a\,\epsilon}}.
\]
It is straightforward to show that $x_*=\frac{\epsilon}{c\, m}+o(\epsilon)$ as $\epsilon\to 0$,
so \eqref{cond_H4gas}
holds for all sufficiently small $\epsilon>0$
when $K<2/\sqrt{a}$.
Hence statement (i) follows.

The function $F(x)=rx(1-x/K)/p(x)$ equals \[
  F(x)=\frac{r}{m}\left(1-\frac{x}{K}\right)(ax^2+1).
\]
Let $X=x/K$ and $\kappa=aK^2$.
Then \[
  F(x)
  = \frac{r}{m}(1-X)(\kappa X^2+1)
  \equiv F_0(X,\kappa).
\]
The graph of $y=F_0(X,\kappa)$, $X\in [0,1]$,
is obtained by stretching
the graph of $y=F(x)$, $x\in [0,K]$, horizontally.
In particular, 
$F_0(0,\kappa)=\frac{r}{m}=\bar{y}=F(0)$ for all $\kappa>0$,
and $F_0(X,\kappa)$ and $F(x)$ have the same local extremum values.

It is straightforward to show that 
$F_0(X,\kappa)$ has two interior extrema if $\kappa>3$,
and no interior extremum if $\kappa\le 3$.
For $\kappa>3$, 
denote the interior local minimal and local maximal points of $F_0(X,\kappa)$
by $\widecheck{X}(\kappa)$ and $\widehat{X}(\kappa)$, respectively.
It is straightforward to show that both
$F_0(\widecheck{X}(\kappa),\kappa)$
and $F_0(\widehat{X}(\kappa),\kappa)$
are strictly increasing functions of $\kappa\in (3,\infty)$.
Since $H(y)$ is strictly decreasing on the interval $(0,\bar{y})$
and is strictly increasing on the interval $(\bar{y},\infty)$
(see Figure \ref{fig_H2_HLog}),
the function \[
  q(\kappa)
  \equiv H(F_0(\widehat{X}(\kappa),\kappa))- H(F_0(\widecheck{X}(\kappa),\kappa))
\]
is strictly increasing on $(3,\infty)$.
It can be shown that $F_0(\widehat{X}(\kappa),\kappa)\big|_{\kappa=4}=\bar{y}$.
Since $H(\bar{y})=0$, it follows that
$q(4)= - H(F_0(\widecheck{X}(\kappa),\kappa))\big|_{\kappa=4}< 0$.
Therefore, there exists $\kappa_*>4$ such that $q(\kappa)>0$ if $\kappa>\kappa_*$,
and $q(\kappa)<0$ if $\kappa<\kappa^*$.

For $aK^2>3$, denote the interior
local minimal and local maximal points of $F(x)$ in $[0,K]$
by $\widecheck{x}$ and $\widehat{x}$, respectively.
Since the graph of $F(x)$, $x\in [0,K]$,
and $F_0(X,\kappa)$, $X\in [0,1]$, $\kappa=aK^2$,
are the same up to horizontal stretching, \[
  F(\widehat{x})=F_0(\widehat{X}(\kappa),\kappa)
  \quad\text{and}\quad
  F(\widecheck{x})=F_0(\widecheck{X}(\kappa),\kappa).
\]
Hence $H(F(\widehat{x}))<H(F(\widecheck{x}))$ if $aK^2<\kappa_*$,
and $H(F(\widehat{x}))>H(F(\widecheck{x}))$ if $aK^2>\kappa_*$.
Assertions (ii) and (iii)
now follow from Proposition \ref{prop_2humps} and Corollary \ref{cor_gas}.
\end{proof}

\section{Proof of the Criteria}
\label{sec_proof}

To prove Theorem \ref{thm_y0},
we state and prove two preliminary theorems:
Theorem \ref{thm_entryexit2} is a variation of bifurcation delay,
and is proved using geometric singular perturbation theory;
Theorem \ref{thm_floquet} is a variation of Floquet theory.

\begin{theorem}\label{thm_entryexit2}
Consider system \eqref{sf_ab2},
where $f$, $g$ and $h$ are $C^{r+1}$ functions, $r\in \mathbb N$,
that satisfy \eqref{cond_turning_ab}.
Assume that $a_0<0<a_1$ satisfies \eqref{entryexit_ab},
and that there exist trajectories $\gamma_1$ and $\gamma_2$ of the limiting system \beq{fast_ab2}
  \dot{a}= b\, h(a,b,0),
  \quad
  \dot{b}= b\,g(a,b,0),
\] such that $(a_1,0)$ is the omega-limit point of $\gamma_1$
and $(a_2,0)$ is the alpha-limit point of $\gamma_2$.
Then for all sufficiently small $\delta_1>0$ and $\delta_2>0$,
there exists $\epsilon_0>0$ such that the followings hold.
Let \[
  (a^\din,\delta_1)
  = \gamma_1\cap \{b=\delta_1\}
  \quad\text{and}\quad
  (a^\dout,\delta_1)
  = \gamma_2\cap \{b=\delta_1\}.
\] Let \[
  \Sigma^\din
  = \{(a,\delta_1): |a-a^\din|<\delta_2\}
  \quad\text{and}\quad
  \Sigma^\dout
  = \{(a,\delta_1): |a-a^\dout|<\frac{|a_1|}{2}\}.
\] 
Then the transition mapping from $\Sigma^\din$ to $\Sigma^\dout$ of \eqref{sf_ab2}
is well-defined for $\epsilon\in (0,\epsilon_0]$,
and is $C^r$ up to $\epsilon=0$.
That is, there exists a $C^r$ function \[
  \pi_\epsilon(z): \Sigma^\din\times [0,\epsilon_0]\to \Sigma^\dout
\] such that, for each $z\in \Sigma^\din$ and $\epsilon\in (0,\epsilon_0]$,
$z$ and $\pi_\epsilon(z)$ are connected
by a trajectory of \eqref{sf_ab2}, and \[
  \pi_0(z_0)= \pi_0(z_1)
  \quad
  \text{for $z_0=(a_0,\delta)$ and $z_1=(a_1,\delta)$ satisfying \eqref{entryexit_ab}.}
\]

Denote the trajectory connecting $z\in \Sigma^\din$ and $\pi_\epsilon(z)$ by $\sigma_\epsilon$,
and the time span of $\sigma_\epsilon$ by $T_{\epsilon,\delta_1}$. Then \beq{est_Teps_ab2}
  T_{\epsilon,\delta_1}=\frac{1}{\epsilon}\left(
    \int_{a_0}^{a_1}\frac{1}{f(a,0,0)}\;da
    + o(1)
  \right)\quad \text{as }\epsilon\to 0.
\] Moreover, there exists $M>0$ such that for each $\Delta\in (0,\delta_1]$,
if we parameterize $\sigma_\epsilon\cap \{b<\Delta\}$
by $(a_\epsilon(t),b_\epsilon(t))$, $t\in [0,T_{\epsilon,\Delta}]$,
then there exists $\epsilon_\Delta>0$ satisfying
\beq{ineq_bint_Delta}
  \int_0^{T_{\epsilon,\Delta}}
  b_\epsilon(t)\; dt
  \le M \Delta
  \quad\forall\; \epsilon\in (0,\epsilon_\Delta].
\]
\end{theorem}

\begin{figure}[t]
\centering
\begin{tabular}{cc}
\frame
{\includegraphics[trim = 1.5cm .8cm 0cm 0cm, clip, width=.46\textwidth]{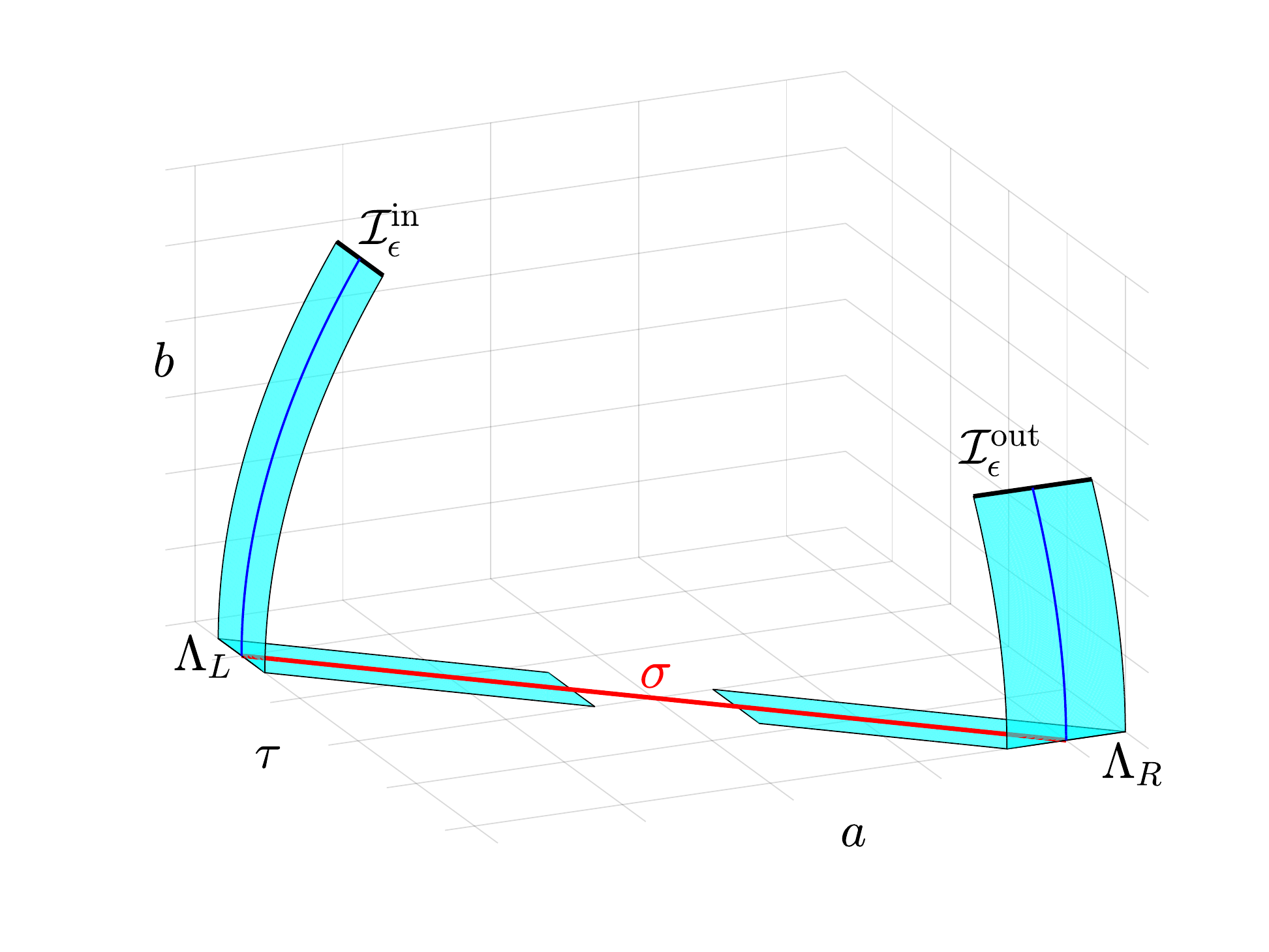}}
&
\frame
{\includegraphics[trim = 1.5cm .8cm 0cm 0cm, clip, width=.46\textwidth]{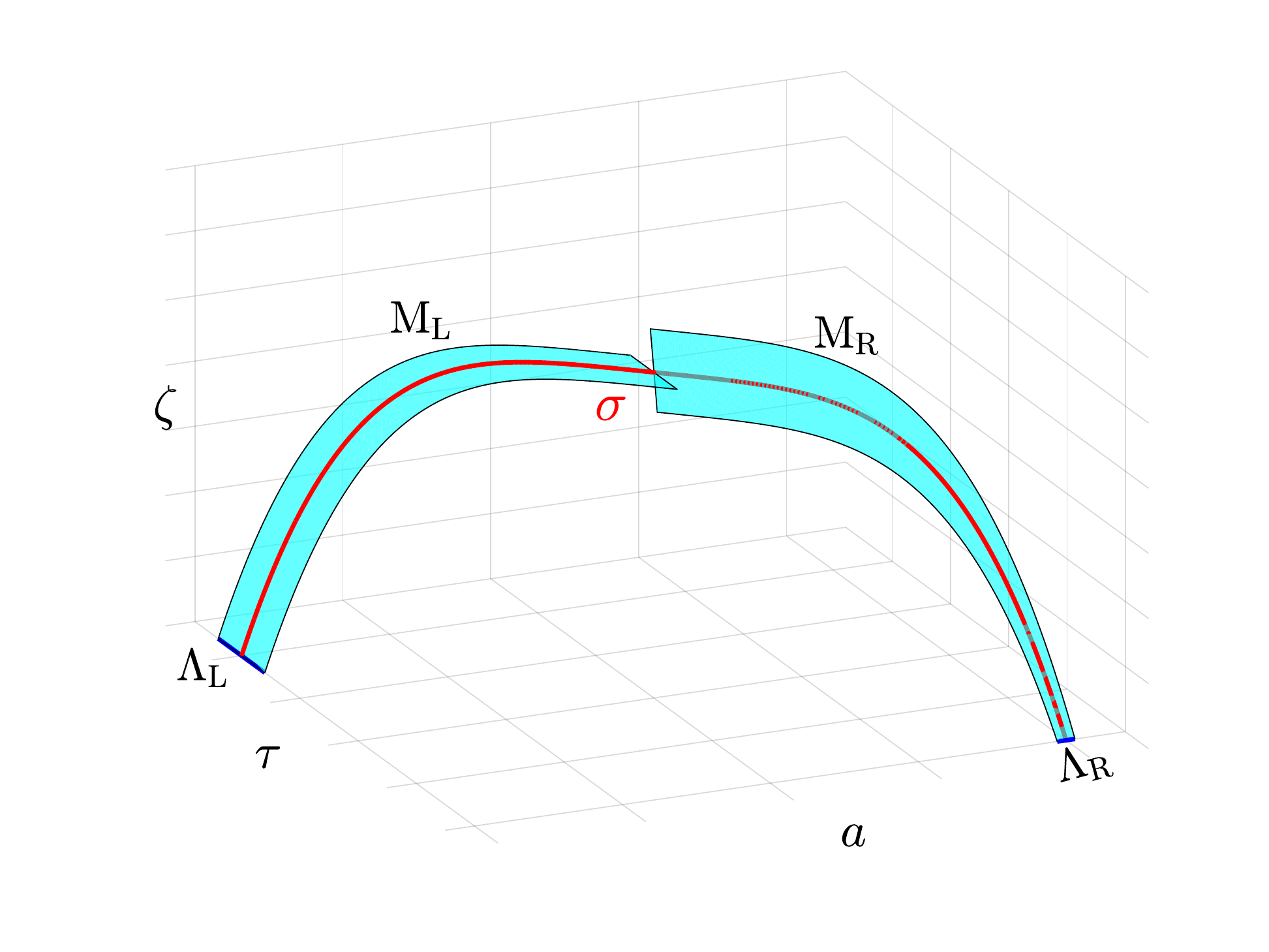}}
\\
(A)
&
(B)
\end{tabular}
\caption{
$\Lambda_L$ and $\Lambda_R$
are the omega- and alpha-limit sets, respectively,
of $\mathcal{I}^\din_0$ and $\mathcal{I}^\dout_0$ along
the fast system \eqref{fast_ab2}.
$M_L$ and $M_R$ are the manifolds
evolved from $\Lambda_L$ and $\Lambda_R$, respectively,
along the slow system \eqref{slow_abz}.
The projections of $M_L$ and $M_L$
in $(a,\zeta,\tau)$-space
intersect transversally along the curve $\sigma$.
}
\label{fig_connection}
\end{figure}

\begin{proof}
The first half of the proof
is similar to that in \cite{THHsu:2017},
so we will skip some details.
Let $\zeta=\epsilon\log b$ and $\tau=\epsilon t$,
where $t$ is the time variable in \eqref{sf_ab2}.
Then \eqref{sf_ab2} is equivalent to \beq{sf_abz}
  \dot{a}= \epsilon f(a,b,\epsilon)+ b\, h(a,b,\epsilon),
  \quad
  \dot{b}= b\,g(a,b,\epsilon),
  \quad
  \dot{\zeta}= \epsilon g(a,b,\epsilon),
  \quad
  \dot{\tau}=\epsilon.
\]
The limiting fast system (or layer system) of \eqref{sf_abz} is \beq{fast_abz}
  \dot{a}= b\, h(a,b,0),
  \quad
  \dot{b}= b\,g(a,b,0),
  \quad
  \dot{\zeta}= 0,
  \quad
  \dot{\tau}=0.
\]
On the invariant set $\{b=0\}$,
the limiting slow system (or reduced system) is \beq{slow_abz}
  a'= f(a,0,0),\quad
  b=0,\quad
  \zeta'=g(a,0,0),\quad
  \tau'=1,
\] where \,$\prime$\, denotes $\frac{d}{d\tau}$.

Choose $\delta_1>0$ small enough so that
the trajectories that approach $a_0$ or $a_1$
can be parameterized as \[
  (a^\din_0(\delta),\delta)
  \quad\text{and}\quad
  (a^\dout_0(\delta),\delta),\quad \delta\in [0,\delta_1].
\]
Let $0<\delta_2<\min\{\frac{|a_1|}{2},\frac{|a_0|}{2}\}$ and \[
  \mathcal{I}^\din_\epsilon=\left\{
    \p a\\ b\\ \zeta\\ \tau\pp:
    \begin{array}{l}
      a=a^\din_0(\delta_1)\\
      b=\delta_1\\
      \zeta=\epsilon\log \delta_1\\
      |\tau|<\delta_2
    \end{array}
  \right\},\quad
  \mathcal{I}^\dout_\epsilon=\left\{
    \p a\\ b\\ \zeta\\ \tau\pp:
    \begin{array}{l}
      |a-a^\dout_0(\delta_1)|<\delta_2\\
      b=\delta_1\\
      \zeta=\epsilon\log \delta\\
      \tau=\tau_1
    \end{array}
  \right\},
\]
where \beq{def_tau1}
  \tau_1=\int_{a_0}^{a_1}\frac{1}{f(a,0,0)}\;da.
\]
Let $\Lambda_L=W^s_0(\mathcal{I}^\din_0)$
and $\Lambda_R=W^u_0(\mathcal{I}^\dout_0)$ 
be the omega- and alpha-limit sets, respectively,
of $\mathcal{I}^\din_0$ and $\mathcal{I}^\dout_0$. Then 
\[
  \Lambda_L
  =\left\{
    \p a\\ b\\ \zeta\\ \tau\pp:
    \begin{array}{l}
      a=a_0\\
      b=0\\
      \zeta=0\\
      |\tau|<\delta_2
    \end{array}
  \right\},\quad
  \Lambda_R
  =\left\{
    \p a\\ b\\ \zeta\\ \tau\pp:
    \begin{array}{l}
      |a-a_1|<\delta_2\\
      b=0\\
      \zeta=0\\
      \tau=\tau_1
    \end{array}
  \right\}.
\]
Let $M_L$ and $M_R$
be the manifolds evolved from $\Lambda_L$ and $\Lambda_R$,
respectively, along \eqref{slow_abz}.
Since system \eqref{slow_abz} can be decoupled,
by writing down the explicit solutions
it is easy to check that,
if both \eqref{entryexit_ab} and  \eqref{def_tau1} are satisfied,
the projections of $M_L$ and $M_R$
in $(a,\zeta,\tau)$-space
intersect transversally along the curve (see Figure \ref{fig_connection})
\beq{def_sigma_abz}
  \sigma
  =\left\{
    \p a\\ \tau\\ \zeta\pp:
      a\in [a_0,a_1],\;\;
      \tau=\int_{a_0}^a \frac{1}{f(a,0,0)}\;da,\;\;
      \zeta=\int_{a_0}^a \frac{g(a,0,0)}{f(a,0,0)}\;da
  \right\}.
\]

Note that $\Lambda_L$ and $\Lambda_R$
are both normally hyperbolic with respect to \eqref{fast_abz}.
Let $\widetilde{\mathcal{I}}^\din_\epsilon$
and $\widetilde{\mathcal{I}}^\dout_\epsilon$
be the manifolds evolved from
$\mathcal{I}^\din_\epsilon$
and $\mathcal{I}^\dout_\epsilon$, respectively, along \eqref{sf_abz}.
From the Exchange Lemma \cite{Jones:2009,Schecter:2008b},
by decreasing $\delta_1$ and $\delta_2$ if necessary,
\begin{align}
  &\widetilde{\mathcal{I}}^\din_\epsilon\cap \{a=a_0+\delta_2\}
  \quad\text{is\, $C^r$ $O(\epsilon)$-close to}\;\;
  M_L\cap \{a=a_0+\delta_2\}
  \label{close_iin_delta}
  \\
\intertext{and}
  &\widetilde{\mathcal{I}}^\dout_\epsilon\cap \{a=a_1-\delta_2\}
  \quad\text{is\, $C^r$ $O(\epsilon)$-close to}\;\;
  M_R\cap \{a=a_1-\delta_2\}.
  \label{close_iout_delta}
\end{align}

Also note that system \eqref{sf_ab2} is equivalent to \beq{deq_azeta}
  &\frac{da}{d\tau}= 
  f(a,e^{-\zeta/\epsilon},\epsilon)
  + \frac{e^{-\zeta/\epsilon}}{\epsilon}\; h(a,e^{-\zeta/\epsilon},\epsilon),
  \\[.5em]
  &\frac{d\zeta}{d\tau}= -g(a,e^{-\zeta/\epsilon},\epsilon).
\]
On $\sigma\cap \{a_0+\delta_2<a<a_1-\delta_2\}$,
from \eqref{entryexit_ab} and \eqref{def_sigma_abz} we have
\beq{min_zeta_abz}
  \zeta
  \ge \min\left\{
    \int_{a_0}^{a_0+\delta_2}\frac{g(a,0,0)}{f(a,0,0)}\;da,\;\;
    \int_{a_1-\delta_2}^{a_1}\frac{-g(a,0,0)}{f(a,0,0)}\;da
  \right\}
  \equiv 2\nu > 0.
\] 
Hence \eqref{deq_azeta} is a regular perturbation of \eqref{slow_abz}
on $\{a_0+\delta_2<a<a_1-\delta_2,\zeta\ge \nu\}$.
Therefore, from \eqref{close_iin_delta}
and \eqref{close_iout_delta}
the projections of $\widetilde{\mathcal{I}}^\din_\epsilon$
and $\widetilde{\mathcal{I}}^\dout_\epsilon$
in $(a,\zeta,\tau)$-space
intersect transversally
along $\sigma\cap \{a_0+\delta_2<a<a_1-\delta_2\}$
for any sufficiently small $\epsilon>0$. 
Hence we can uniquely define the intersection point \[
  z_\epsilon
  = \widetilde{\mathcal{I}}^\din_\epsilon
  \cap \widetilde{\mathcal{I}}^\dout_\epsilon\cap \{a=0\}.
\]
By the definition of $\widetilde{\mathcal{I}}^\din_\epsilon$
and $\widetilde{\mathcal{I}}^\dout_\epsilon$,
the trajectory of \eqref{sf_abz} containing $z_\epsilon$
has unique intersection points with 
$\mathcal{I}^\din_\epsilon$ and $\mathcal{I}^\dout_\epsilon$.
Denote these intersections points by $z_\epsilon^\din\in \mathcal{I}^\din_\epsilon$
and $z_\epsilon^\dout\in \mathcal{I}^\dout_\epsilon$.
By definition of $\mathcal{I}^\din_\epsilon$,
the $a$-coordinate of $z_\epsilon^\din$ is $a^\din(\delta_1)$.
Denote the $a$-coordinate of $z_\epsilon^\dout$
by $a_\epsilon^\dout$,
$\epsilon\in (0,\epsilon_0]$,
and $a_{\epsilon=0}^\dout=a_1$.
Then $a^\dout_\epsilon$ is a $C^r$ function of $\epsilon\in [0,\epsilon_0]$
such that $(a^\dout_\epsilon,\delta_1)$
lies in the trajectory of \eqref{sf_ab2} containing  $(a^\din(\delta_1),\delta_1)$.

Now we consider $a_0$ in the above discussion
as an independent variable in a compact subset $\mathcal{A}$ of $(-\infty,0)$.
Assume that each $a_0\in \mathcal{A}$
corresponds to an $a_1>0$ satisfying \eqref{entryexit_ab}.
By the construction, decreasing $\epsilon_0$ and $\delta_2$ if necessary,
the mapping \[
  (\epsilon,a_0)\mapsto a^\dout_\epsilon
\] is $C^r$ on $[0,\epsilon_0]\times \mathcal{A}$,
and the transition map $\pi_\epsilon:\Sigma^\dout\to \Sigma^\din$
satisfies \[
  \pi_\epsilon(a_0,\delta_1)= (a^\dout_\epsilon(a_0),\delta_1),\quad
  \epsilon\in (0,\epsilon_0].
\] Setting $\pi_{\epsilon=0}(a_0,\delta_1)=(a_1,\delta_1)$,
then $\pi_\epsilon$
is also $C^r$ on $[0,\epsilon_0]\times \Sigma^\dout$.

By the relation $t=\epsilon\tau$
and the expression of $\sigma$ in \eqref{def_sigma_abz},
the time span in $t$
form $z^\din\in \Sigma^\din$ to $\Sigma^\dout$ is
$T_{\epsilon,\delta_1}=\tau_1/\epsilon+O(1)$.
This proves \eqref{est_Teps_ab2}.

Next we prove \eqref{ineq_bint_Delta}.
Let $\sigma_\epsilon$ be the trajectory of \eqref{sf_ab2} passing through
a point $(a_0,\delta_1)\in \Sigma^\din$.
Fix any $\Delta\in (0,\delta_1]$.
Parameterize $\sigma_\epsilon\cap \{b<\Delta\}$ by
$(a,b)=(a_\epsilon(t),b_\epsilon(t))$, $t\in [0,T_{\epsilon,\Delta}]$.
We split $\sigma_\epsilon\cap \{b<\Delta\}$
into three parts: \[
  &\sigma_\epsilon^{(1)}= \sigma_\epsilon\cap \{a\le a_0+\delta_2, b<\Delta\}
  =\{(a_\epsilon(t),b_\epsilon(t)): t\in [0,T_{\epsilon,\Delta}^{(1)}]\},
  \\[.5em]
  &\sigma_\epsilon^{(2)}= \sigma_\epsilon\cap \{a\in [a_0+\delta_2,a_1-\delta_2], b<\Delta\},
  \\[.5em]
  &\sigma_\epsilon^{(3)}= \sigma_\epsilon\cap \{a\ge a_1-\delta_2, b<\Delta\}.
\] 
Since $g(a,b,\epsilon)<0$ for $a<0$
there exist $0<\mu<C$ such that \beq{gnegtive_ab2}
  g(a,b,\epsilon)<-\mu
  \quad\forall\;(a,b,\epsilon)\in [a_0,a_0+\delta_2]\times [0,\delta_1]\times [0,\epsilon_0]
\]
Choose $\epsilon_\Delta>0$ such that \[
  a<a_0+\delta_1\quad \forall\; (a,b)\in \sigma_\epsilon^{(1)}
  \qquad\text{and}\qquad
  a>a_1-\delta_1\quad \forall\; (a,b)\in \sigma_\epsilon^{(3)}
\] for all $\epsilon\in (0,\epsilon_\Delta]$.
From \eqref{gnegtive_ab2} and the equation $\dot{b}=-bg(a,b,\epsilon)$,
it follows that \[
  \dot{b}<-\mu b
  \quad\text{on }\sigma_\epsilon^{(1)}.
\]
Since $b_\epsilon(0)=\Delta$, we then obtain \[
  b_\epsilon(t)\le \Delta\, e^{-\mu t},\quad t\in [0,T_\epsilon^{(1)}].
\] It follows that \beq{ineq_bint_sigma1}
  \int_{\sigma_\epsilon^{(1)}}b\;dt
  \le \int_0^{T_{\epsilon,\Delta}^{(1)}} \Delta\, e^{-\mu t}\;dt
  < \frac{C\Delta}{\mu}.
\]
Similarly, \beq{ineq_bint_sigma3}
  \int_{\sigma_\epsilon^{(3)}}b\;dt
  < \frac{C\Delta}{\mu}.
\]

From \eqref{min_zeta_abz} we have \[
  -\epsilon\log b= \zeta\ge \nu>0
  \quad\text{on }\sigma_\epsilon^{(2)},
\] that is, \beq{ineq_b_sigma2}
  b\le \exp(-\nu/\epsilon)
  \quad\text{on }\sigma_\epsilon^{(2)}.
\]
Therefore, using \eqref{ineq_b_sigma2} on \eqref{sf_ab2}, 
on $\sigma_\epsilon^{(2)}$ we have \[
  \left|\frac{b}{da/dt}\right|
  < \frac{b}{\mu\epsilon-Cb}
  <\frac{1}{\mu\,\epsilon\exp(\nu/\epsilon)-C}
  \to 0
  \quad\text{as }\epsilon\to 0.
\]
Hence \beq{ineq_bint_sigma2}
  \int_{\sigma_\epsilon^{(2)}}b\;dt
  = \int_{a_0+\delta_2}^{a_1-\delta_2}\frac{b}{da/dt}\;da
  \to 0
  \quad\text{as }\epsilon\to 0.
\] Summing up \eqref{ineq_bint_sigma1}, \eqref{ineq_bint_sigma3} and \eqref{ineq_bint_sigma2},
we conclude that,
the inequality \eqref{ineq_bint_Delta}
with $M=3C\Delta/\mu$
holds for any sufficiently small $\epsilon$.
\end{proof}

Floquet multipliers
are useful for calculating
the determinant of the Jacobian
of the Poincar\'{e} map
at a cross section of a periodic orbit
(see e.g.\ \cite[Chapter 11]{Teschl:2012}).
Our configuration $\Gamma(x_0)$
defined in \eqref{def_gamma12}
is a closed loop,
so the nearby trajectories
are nearly periodic orbits.
Since these trajectories are not necessarily periodic,
standard Floquet theory does not directly apply.
The following theorem is a variation of that theory suiting our context.

\begin{theorem}\label{thm_floquet}
Consider systems in $\mathbb R^N$, $N\ge 2$, 
of the form \beq{deq_z}
  \dot{z}= {h}_\epsilon({z}),
\]
where ${h}_\epsilon(z)=h(z,\epsilon)$
is a $C^2$ function of $({z},\epsilon)\in \mathbb R^N\times [0,\epsilon_0]$.
Let $z_0\in \mathbb R^N$ with $h(z_0,0)\ne 0$,
and let $\Sigma$ be a cross section of $z_0$
transversal to $h_0(z_0)$.
Assume that there exist $\epsilon_0>0$
and a neighborhood $\Sigma_{(1)}\subset \Sigma$ of $z_0$
such that the return map from $\Sigma_{(1)}$ to $\Sigma$
is well-defined for $\epsilon\in (0,\epsilon_0]$
and is $C^1$ up to $\epsilon=0$.
That is, there is a $C^1$ function
\[
  P_\epsilon(z): 
  \Sigma_{(1)}\times [0,\epsilon_0]\to \Sigma
\] such that
for any $z\in \Sigma_{(1)}$ and $\epsilon\in (0,\epsilon_0]$
there is a trajectory of \eqref{deq_z}
that starts at $z\in \Sigma_{(1)}$
and returns to $\Sigma$ at $P_\epsilon(z)$.

Let $\zeta_\epsilon(t)$, $0\le t\le T_\epsilon$,
be a trajectory of \eqref{deq_z} that starts at $z_0$
and ends at $P_\epsilon(z_0)$.
Assume that \[
  \int_0^{T_\epsilon}
  \mathrm{div}(h_\epsilon(\zeta_\epsilon(t)))\;dt
  \to \lambda_0
  \quad\text{as }\epsilon\to 0
\] for some $\lambda_0\in \mathbb R$.
Then \beq{eq_detDP}
  \det\big(DP_{0}(u_0)\big)
  = \exp(\lambda_0),
\]
where $DP_{\epsilon}(u_0)$ is regarded
as a linear transform on the tangent space $T_{z_0}\Sigma$.
\end{theorem}

\begin{proof}
The proof is similar to that of \cite[Theorem 11.4]{Teschl:2012}.
Let $\zeta_\epsilon(t)$, $t\in [0,T_\epsilon]$,
be the solution of \eqref{deq_z} with trajectory $\Gamma_\epsilon$.
Let $\beta$ be an ordered basis 
of $\mathbb R^N$
consisting of the vector $h_0(z_0)$
and $(N-1)$ vectors $e_2,\dots,e_n$ of the tangent space $T_{z_0}\mathcal{I}$.
Let $A_\epsilon(t)=[D{h}(\zeta_\epsilon(t))]_\beta$
be the 
matrix form of the
linearization of ${h}_\epsilon$ along $\Gamma_\epsilon$
with respect to the basis $\beta$.
Let $\Phi(t)$ be the fundamental matrix of \[
  \dot{u}= A_\epsilon(t)u
\] with $\Phi(0)$ equal to the identity matrix.
Then \beq{eq_phi_dp}
  \Phi_\epsilon(T_\epsilon)
  = \begin{pmatrix}
  v_\epsilon & *\\
  w_\epsilon& [DP_\epsilon(u_0)]_{\beta'}
  \end{pmatrix}_{N\times N}.
\] where $\beta'=\{e_2,\dots,e_N\}$.

On the other hand, by Liouville's formula, \[
  \det(\Phi(T_\epsilon))
  = \exp\left(
    \int_0^{T_\epsilon}
    \mathrm{tr}(A_\epsilon(t))
    \;dt
  \right)
  \det(\Phi(0)).
\]
Since $\mathrm{tr}(A_\epsilon)= \mathrm{div}(h_\epsilon)$ and $\det(\Phi(0))=\det(I)=1$,
it follows that \beq{eq_phi_div}
  \det(\Phi(T_\epsilon))
  = \exp\left(
    \int_{0}^{T_\epsilon}
      \mathrm{div}(h_\epsilon(\zeta_\epsilon(t))
    \;dt
  \right).
\]
To prove \eqref{eq_detDP},
by \eqref{eq_phi_dp} and \eqref{eq_phi_div},
it suffices to show that \beq{claim1}
  v_\epsilon= 1+O(\epsilon)
  \quad\text{and}\quad
  w_\epsilon= O(\epsilon).
\] 
Let \[
  \rho_\epsilon(t)= h_\epsilon(\zeta(t)).
\] 
Then \[
  \Phi_\epsilon(t) \big(h_\epsilon(z_0)\big)
  =\Phi_\epsilon(t) \big(\rho_\epsilon(0)\big)
  = \rho_\epsilon(t)
  \quad\forall\;t\in[0,T_\epsilon].
\] 
In particular, \beq{phi_eps_rho}
  \Phi_\epsilon(T_\epsilon) \big(h(z_0)\big)
  = \rho_\epsilon(T_\epsilon).
\]
Since $\Gamma_\epsilon$ is in a $O(\epsilon)$-neighborhood of 
the closed configuration $\Gamma_0$
and $P_\epsilon\to P_0$ as $\epsilon\to 0$,
we have \beq{rho_eps_h}
  \rho_\epsilon(T_\epsilon)
  =h_\epsilon(P_\epsilon(z_0))
  =h_0(z_0)+O(\epsilon).
\]
From \eqref{phi_eps_rho} and \eqref{rho_eps_h} it follows that \[
  \Phi_\epsilon(T_\epsilon) \big(h(z_0)\big)
  = h_0(z_0)+O(\epsilon).
\]
Since $h_0(z_0)$ is the first element in the ordered basis $\beta$,
we obtain \eqref{claim1}.
This completes the proof.
\end{proof}

\begin{figure}[t]
\begin{center}
\begin{tabular}{cc}
\frame
{\includegraphics[trim = 2.5cm 1cm 1cm .5cm, clip, width=.46\textwidth]{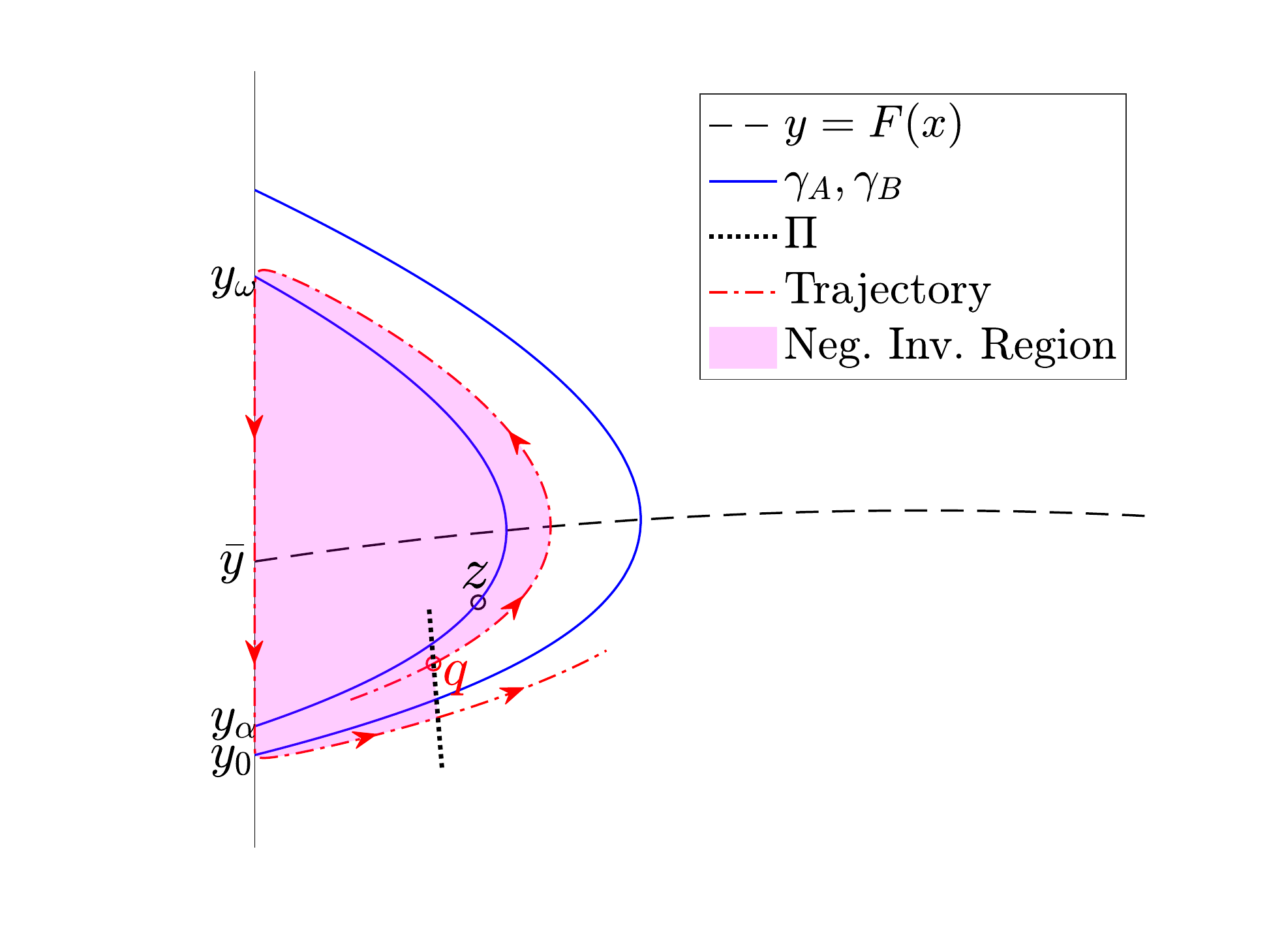}}
&
\frame
{\includegraphics[trim = 2.5cm 1cm 1cm .5cm, clip, width=.46\textwidth]{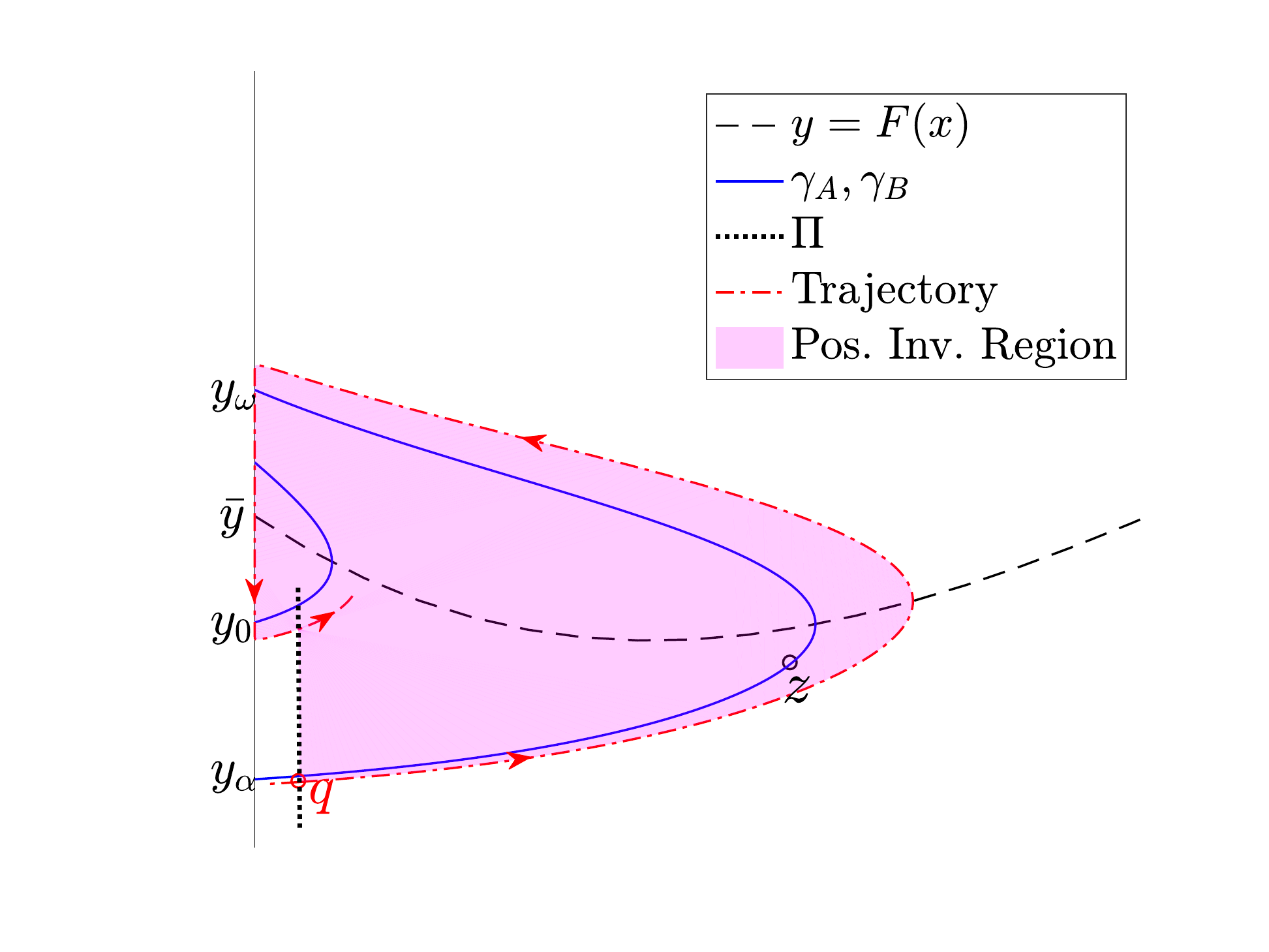}}
\\
(A)
&
(B)
\end{tabular}
\end{center}
\caption{
(A)
$H(y_\omega)>H(y_\alpha)$.
The shaded region is negatively invariant.
(B)
$H(y_\omega)<H(y_\alpha)$.
The shaded region is positively invariant.
In either case, 
since the region contain a proper subset of the trajectory passing through $z$,
a neighborhood of $z$ does not intersect any periodic orbit.
}
\label{fig_trapping}
\end{figure}

Now we are ready to prove our main criteria.

\begin{proof}[Proof of Theorem \ref*{thm_y0}]
Suppose a point $z$ lies
in the interior of $\gamma(x_0)$
for some $x_0\in (0,K)$ with $\chi(x_0)\ne 0$.
Then the alpha-limit point $(0,y_\alpha)=(0,y_\alpha(x_0))$
and omega-limit point $(0,y_\omega)=(0,y_\omega(x_0))$
satisfy $H(y_\omega)\ne H(y_\alpha)$.

In the case that $H(y_\omega)>H(y_\alpha)$,
by the monotonicity of $H(y)$ on the interval $(0,\bar{y})$ (see Figure \ref{fig_H2_HLog}),
there is a unique number $y_0<y_\alpha$ such that $H(y_0)=H(y_\omega)$.
By Theorem \ref{thm_entryexit2},
a trajectory of \eqref{deq_xy}
that enters the vicinity of the $y$-axis near the point $(0,y_\omega)$
must leave near the point $(0,y_0)$.
Let $\gamma_A$ be the trajectory of \eqref{fast_xy} containing $z$,
and $\gamma_B$ the trajectory of \eqref{fast_xy} with alpha-limit point $(0,y_0)$.
Let $\Pi$ be a common cross section of $\gamma_A$ and $\gamma_B$.
Choose a point $q\in \Pi$ between $\gamma_A$ and $\gamma_B$
(see Figure \ref{fig_trapping}{(A)}).
For each sufficiently small $\epsilon$,
using a portion of the forward trajectory of \eqref{deq_xy} from $q$
and a portion the cross section $\Pi$
one can construct a negatively invariant region
that contains a proper subset of the trajectory passing through $z$.
Hence there is a neighborhood of $z$ that does not intersect any periodic orbit.
The case with $H(y_\omega)<H(y_\alpha)$ can be treated similarly,
and is illustrated in Figure \ref{fig_trapping}{(B)}.

Next we fix any $x_0\in (0,K)$ that satisfies $\chi(x_0)=0$.
Then $y_\omega(x_0)$ and $y_\alpha(x_0)$
satisfies the entry-exit relation \eqref{entryexit_ab} for \eqref{deq_xy}
with $(y,x)$ playing the role of $(a,b)$ in \eqref{sf_ab2}.
Choose $\delta_1>0$ and $\delta_2>0$ small enough
so that the conclusion of Theorem \ref{thm_entryexit2}
holds for \eqref{deq_xy}
with $y_\omega(x_0)$ and $y_\alpha(x_0)$
playing the roles of $a_0$ and $a_1$, respectively.
Let \[
  &z^\din=
  (\delta_1,y^\din)
  = \gamma(x_0)\cap \{x=\delta_1,y>\bar{y}\},
  \\
  &z^\dout=
  (\delta_1,y^\dout)
  = \gamma(x_0)\cap \{x=\delta_1,y<\bar{y}\},
\] and \[
  \Sigma^\din
  = \{(\delta_1,y): |y-y^\din|<\delta_2\},\quad
  \Sigma^\dout
  = \{(\delta_1,y): |y-y^\dout|<y^\dout/2\}.
\]
From Theorem \ref{thm_entryexit2}
there are $\epsilon_0>0$ and a $C^1$ function \[
  \pi_\epsilon^{(1)}: \Sigma^\din\times [0,\epsilon_0]
  \to \Sigma^\dout
\] such that
$\pi_\epsilon^{(1)}$ is the transition map from $\Sigma^\din$ to $\Sigma^\dout$
for each $\epsilon\in (0,\epsilon_0]$.

Note that $z^\din\in \gamma_1$ and $z^\dout\in \gamma_2$
both lie on the same trajectory of \eqref{fast_xy}.
Since \eqref{fast_xy} is a regular perturbation of \eqref{deq_xy}
and the segment $\Sigma^\dout$ is compact,
the transition map from $\Sigma^\dout$
to a segment $\Pi\subset \{(\delta_1,y): y>\bar{y}\}$ containing $\Sigma^\din$
is well-defined and smooth for all sufficiently small $\epsilon\ge 0$.
That is, by decreasing $\epsilon_0$ if necessary,
there is a $C^1$ function \[
  \pi_\epsilon^{(2)}:
  \Sigma^\dout\times [0,\epsilon_0]
  \to 
  L
\] such that
$\pi_\epsilon^{(2)}$ is the transition map from $\Sigma^\dout$ to $\Pi$
for all $\epsilon\in [0,\epsilon_0]$.

Let \[
  \pi_\epsilon=\pi_\epsilon^{(2)}\circ \pi_\epsilon^{(1)}:
  \Sigma^\din\to \Pi,
  \quad \epsilon\in [0,\epsilon_0]
\]
and \[
  \Sigma^\din_{(1)}
  = \{z\in \Sigma^\din: \pi_\epsilon(z)\in \Sigma^\din\;\;\forall\;\epsilon\in [0,\epsilon_0]\}.
\] As a pre-image of a compact interval,
 $\Sigma^\din$ is a compact interval containing $z^\din$.
Let $P_\epsilon$ 
be the restriction of $\pi_\epsilon$ on $\Sigma^\din_{(1)}$.
Now we have constructed a $C^1$ function \[
  P_\epsilon: \Sigma^\din_{(1)}\times [0,\epsilon_0]
  \to \Sigma^\din
\] such that
$P_\epsilon$ is the return from $\Sigma^\din_{(1)}\subset \Sigma^\din$ to $\Sigma^\din$
for all $\epsilon\in (0,\epsilon_0]$.
See Figure \ref{fig_return} for a visualization of these segments.

Let $\Gamma_\epsilon$ be a trajectory of \eqref{deq_xy}
that starts at $z^\din$ and ends at $P_\epsilon(z^\dout)\in \Sigma^\din$.
Denote the solution of \eqref{deq_xy}
with the trajectory $\Gamma_\epsilon$ by $(x(t),y(t))$, $t\in [0,T_\epsilon]$.
Then \[
  &\int_0^{T_\epsilon}
  \mathrm{div}\begin{pmatrix}
    p(x)\big(F(x)-y\big)\\
    y(-\epsilon+ cp(x))
  \end{pmatrix}\;dt
  \\[.5em]
  &=\int_0^{T_\epsilon} p'(x)\big(F(x)-y\big)+ p(x)F'(x)-\epsilon+ cp(x)\; dt\\[.5em]
  &=\int_{\Gamma_\epsilon} \frac{p'(x)}{p(x)}\;dx
  + \int_{\Gamma_\epsilon}\frac{p(x)F'(x)}{y(-\epsilon+cp(x))}\;dy
  + \int_{\Gamma_\epsilon}\frac{1}{y}\;dy
  \\[.5em]
  &=
  \log(p(x))\Big|_{x=x(0)}^{x(T_\epsilon)}
  +\int_{\Gamma_\epsilon}\frac{p(x)F'(x)}{y(-\epsilon+cp(x))}\;dy
  + \log(y)\Big|_{y=y(0)}^{y(T_\epsilon)}.
\]
By the continuity of $P_\epsilon$ at $\epsilon=0$,
$\Gamma_\epsilon(x_0)\to \Gamma(x_0)$
uniformly as $\epsilon\to 0$
on any compact subset of $\gamma(x_0)$,
so the first and the third terms in the last expression tends to $0$,
and \eqref{ineq_bint_Delta} in Theorem \ref{thm_entryexit2} implies that \[
  \int_{\Gamma_\epsilon}\frac{p(x)F'(x)}{y(-\epsilon+cp(x))}\;dy
  =\int_{\gamma(x_0)}\frac{F'(x)}{cy}\;dy+O(\epsilon).
\]
The last integral equals $\lambda(x_0)/c$
by the definition of $\lambda(x_0)$ in \eqref{def_lambda}.
Hence \beq{divF_lambda}
  \int_0^{T_\epsilon}
  \mathrm{div}\begin{pmatrix}
    p(x)\big(F(x)-y\big)\\
    y(-\epsilon+ cp(x))
  \end{pmatrix}\;dt
  =\frac{\lambda(x_0)}{c}+O(\epsilon).
\]
Since $\Sigma^\din$ is one-dimensional,
from \eqref{divF_lambda} and Theorem \ref{thm_floquet} it follows that \beq{eq_dp0_exp}
  DP_0(z^\din)
  = \exp(\lambda(x_0)/c).
\]
If $\lambda(x_0)\ne 0$,
then $DP_0(z^\din)\ne 1$.
Since $P_0(z^\din)=z^\din$,
by the implicit function theorem,
there exists $z^\din_\epsilon\in \Sigma^\din_{(1)}$ 
for all sufficiently small $\epsilon>0$
satisfying $P_\epsilon(z^\din_\epsilon)=z^\din_\epsilon$.
This means there is a periodic orbit $\ell_\epsilon$ passing through $z^\din_\epsilon$.
By the $C^1$ continuity of $P_\epsilon$,
the value of $\log(DP_\epsilon(z^\din_\epsilon))$ has the same sign
as $\log(DP_0(z^\din))$ for all sufficiently small $\epsilon>0$.
Therefore, by \eqref{eq_dp0_exp}
and standard Floquet theory,
$\ell_\epsilon$ is orbitally locally asymptotically stable if $\lambda(x_0)<0$,
and is orbitally unstable if $\lambda(x_0)>0$.

Since the time span from $\Sigma^{\dout}$ to $\Pi$
is uniformly bounded for $\epsilon\in [0,\epsilon_0]$,
by the equation $\dot{y}=\epsilon y+ yp(x)$
and the estimate \eqref{est_Teps_ab2} in Theorem \ref{thm_entryexit2}, \[
  T_\epsilon= \frac{1}{\epsilon}\left(
    \int_{y_\alpha(x_0)}^{y_\omega(x_0)} \frac{1}{y}\;dy
    + o(1)
  \right),
\]
which gives \eqref{est_Teps}.
\end{proof}

\begin{figure}[t]
\centering
\begin{tabular}{cc}
\frame
{\includegraphics[trim = 1.2cm .5cm .7cm 0cm, clip, width=.44\textwidth]{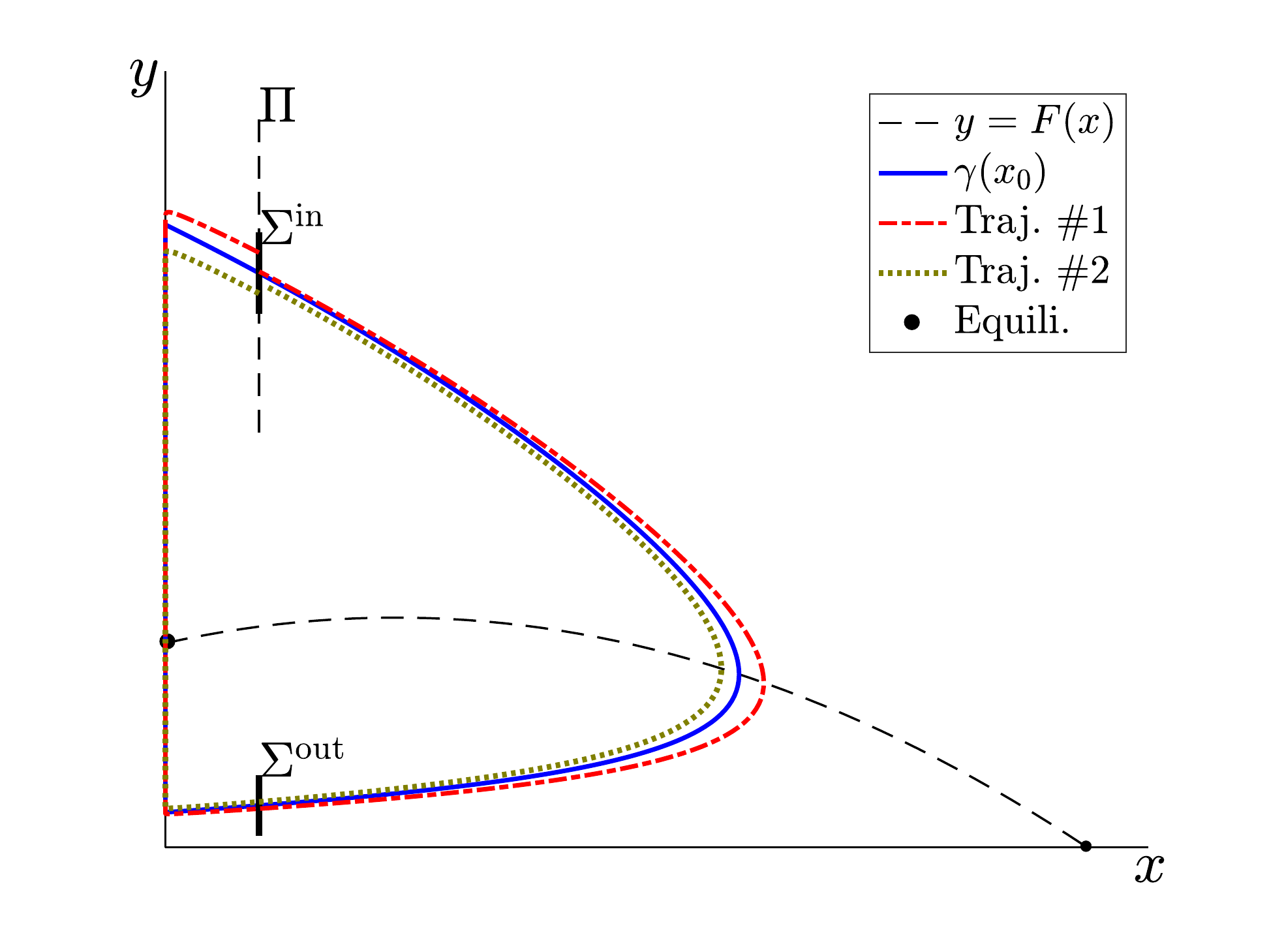}}
&
\frame
{\includegraphics[trim = 1.2cm .5cm .7cm 0cm, clip, width=.44\textwidth]{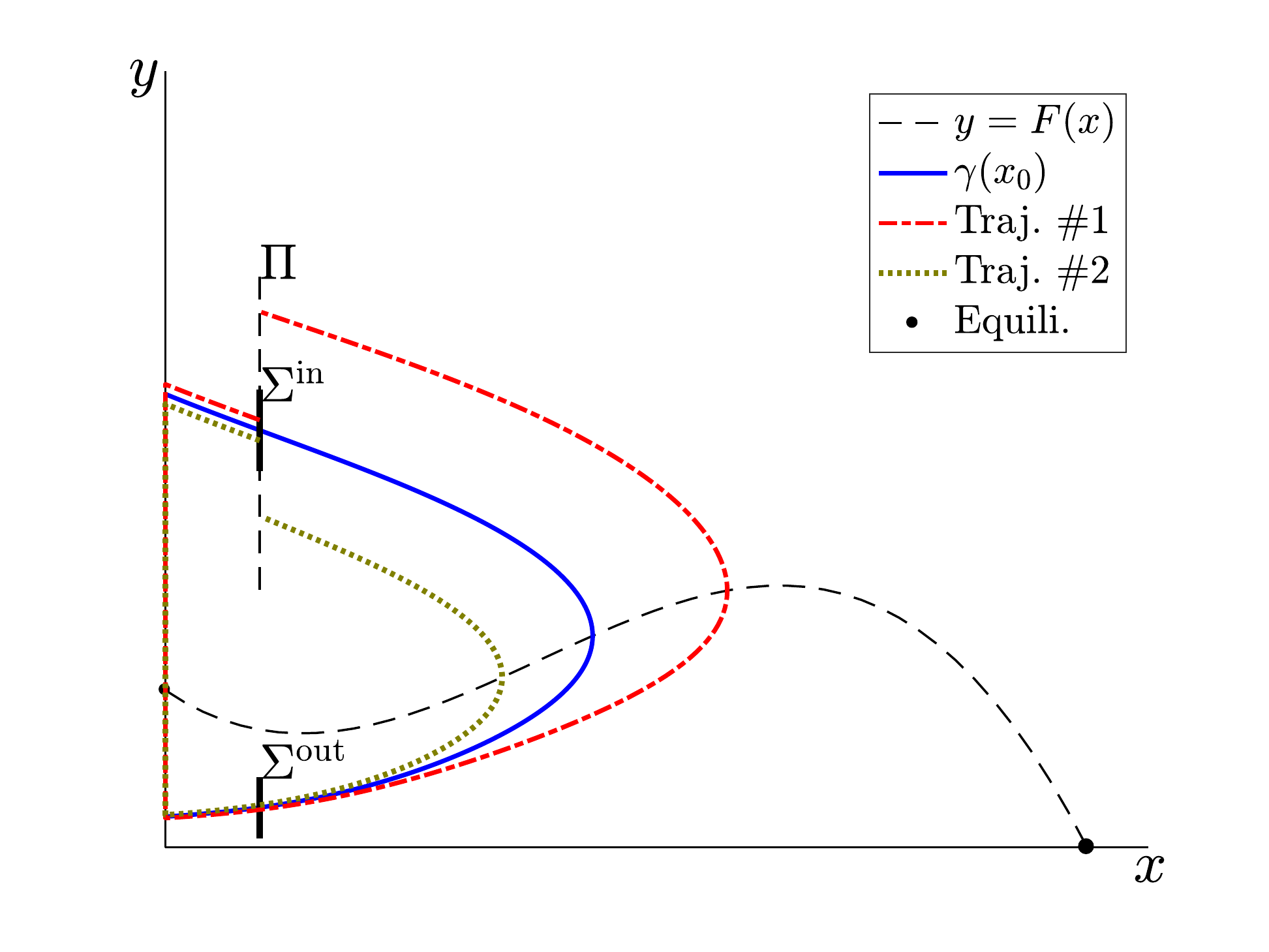}}
\\
(A)
&
(B)
\end{tabular}
\caption{
$\Sigma^\din$
is a cross section of $\Gamma(x_0)$
with $\chi(x_0)=0$, $\lambda(x_0)\ne 0$.
The return map $P_\epsilon:\Sigma^\din_{(1)}\to \Sigma^\din$
is obtained by restricting
the transition map form $\Sigma^{\din}$ to $\Pi$.
(A) $\lambda(x_0)<0$ corresponds to $|\det DP_\epsilon|<1$.
(B) $\lambda(x_0)>0$ corresponds to $|\det DP_\epsilon|>1$.
}
\label{fig_return}
\end{figure}

\section{Discussion}
\label{sec_discussion}

In this paper we considered
the classical predator-prey model
with various functional responses $p(x)$.
Using geometric singular perturbation theory
and Floquet theory,
in Theorem \ref{thm_y0}
we derived characteristic functions $\chi(x)$ and $\lambda(x)$, $x\in (0,K)$,
that depend on the trajectory $\gamma(x)$
for the limiting system with $\epsilon=0$,
such that
\[\text{\begin{tabular}{l}
  (i)\; $\chi(x_0)\ne 0$
  \;$\Rightarrow$\;
  no periodic orbit passes near $\gamma(x_0)$ for any $0<\epsilon\ll 1$.
  \\[1em]
  (ii) $\chi(x_0)=0$ and $\lambda(x_0)\ne 0$
  \;$\Rightarrow$\;
  \text{$\gamma(x_0)$ admits a relaxation oscillation as $\epsilon\to 0$.}
\end{tabular}}\]
In the latter case, the sign of $\lambda(x_0)$
determines the stability of the limit cycles.

With these criteria, 
in Theorem \ref{thm_1hump}
we demonstrated that
if the function $F(x)=rx(1-x/K)/p(x)$
has a single interior extremum,
then the system has a unique limit cycle $\ell_\epsilon$
when the death rate $\epsilon>0$ is sufficiently small,
and $\ell_\epsilon$ forms a relaxation oscillation.
A broad class of response functions,
including 
$p(x)=m\log(1+ax)$
and $p(x)=mx/(a+x)$,
satisfy this condition.

When $F(x)$ has two interior local extrema,
assuming the yield rate $c>0$ is sufficiently small
and $F''(x)<0$ on the right of the interior local maximum,
in Proposition \ref{prop_2humps}
we derived a criterion
depending on the values of the local minimum and local maximum of $F(x)$
to conclude that
either the system has no limit cycle
or has exactly two limit cycles
for all sufficiently small $\epsilon>0$.
In particular,
for the Holling type IV functional response $p(x)=mx/(ax^2+1)$,
in Example \ref{ex_H4}
we derived a threshold $K_*$
such that, for all sufficiently small $\epsilon$,
the system has no limit cycle if $K<K_*$,
and the system has exactly two limit cycles if $K>K_*$.
The inner limit cycle is orbitally unstable,
and the outer one is orbitally locally asymptotically stable.
The prey has low population of order $\exp(-1/\epsilon)$
on the outer limit cycle
for a long timespan of order $1/\epsilon$.
This result supports the paradox of enrichment,
which says that
larger carrying capacity $K$
may lead to destabilization of the coexistence equilibrium
and endanger one or both populations.

\section*{Acknowledgements}
The author
thanks Prof.\ Gail~S.~K.~Wolkowicz
for her many helpful comments on this work,
especially the suggestion on tackling the case of two local extrema.
The author would also like to thank Prof.\ Huaiping Zhu
for his insightful introduction to this problem,
and the anonymous referees
for their careful reading and constructive comments.


\end{document}